\documentclass[11pt,reqno]{amsart}
\usepackage{amsmath, amssymb, amscd, amsbsy, amsthm,braket,bm} 
\usepackage{graphicx}
\usepackage{mathrsfs,yhmath}
\usepackage{overpic}
\usepackage{tikz}
\usepackage[all]{xy}

\usetikzlibrary{calc,decorations.markings}
\tikzstyle{mydash}=[dash pattern={on 2pt off 1pt}]
\newtheorem{theorem}{Theorem}[section]
\newtheorem{lemma}[theorem]{Lemma}
\newtheorem{prop}[theorem]{Proposition}
\newtheorem{cor}[theorem]{Corollary}
\theoremstyle{definition}
\newtheorem{definition}[theorem]{Definition}

\theoremstyle{remark}
\newtheorem{remark}[theorem]{Remark}
\numberwithin{equation}{section}

\newcommand{\twist}{\begin{tikzpicture}[baseline=-0.5ex, scale=0.15,semithick]
\draw plot[domain=20:340, variable=\t, smooth] ({sin(\t)},{sin(2*\t)/2.2});
\end{tikzpicture}}

\newcommand{\K}{{\mathcal K}}
\newcommand{\Z}{{\mathbb Z}}
\newcommand{\Q}{{\mathbb Q}}

\newcommand{\Y}{{\mathbb Y}}
\newcommand{\D}{\mathfrak{D}}
\newcommand{\cc}{\mathfrak{c}}
\renewcommand{\ss}{\mathfrak{s}}
\renewcommand{\S}{\mathfrak{S}}
\newcommand{\I}{\mathcal{I}}
\newcommand{\M}{\mathcal{M}}
\renewcommand{\sl}{\operatorname{sl}}
\newcommand{\rank}{\operatorname{rank}}
\renewcommand{\Im}{\operatorname{Im}}
\newcommand{\id}{\operatorname{id}}
\newcommand{\Id}{\operatorname{Id}}
\newcommand{\Ker}{\operatorname{Ker}}
\newcommand{\Coker}{\operatorname{Coker}}
\newcommand{\Aut}{\operatorname{Aut}}
\newcommand{\inc}{\operatorname{inc}}

\newcommand{\Hom}{\operatorname{Hom}}
\newcommand{\Spin}{\operatorname{Spin}}
\newcommand{\rev}{\operatorname{rev}}

\newcommand{\sq}{\operatorname{sq}}
\newcommand{\tor}{\operatorname{tor}}
\newcommand{\Gr}{\operatorname{Gr}}
\newcommand{\Lk}{\operatorname{Lk}}
\newcommand{\Sp}{\mathrm{Sp}}

\newcommand{\AS}{\mathrm{AS}}
\newcommand{\IHX}{\mathrm{IHX}}

\newcommand{\A}{\mathcal{A}}
\newcommand{\C}{\mathcal{C}}
\renewcommand{\H}{\mathcal{H}}

\newcommand{\F}{\mathcal{F}}
\newcommand{\LCob}{\mathcal{LC}\mathit{ob}}

\newcommand{\tsA}{{}^{\mathit{ts}}\!\mathcal{A}}
\newcommand{\Ztilde}{\widetilde{Z}}

\newcommand{\ideg}{\operatorname{i-deg}}

\newcommand{\raisegraph}[3]{\raisebox{#1}{\includegraphics[height=#2]{#3}}}

\newcommand{\strutgraph}[2]{\hspace{-0.2em}\raisebox{-0.7em}{
\begin{overpic}[height=2em]{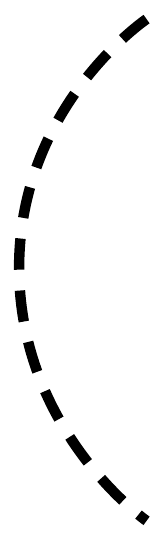}\scriptsize
 \put (35,80){$#1$}
 \put (35,0){$#2$}
\end{overpic}}\hspace{1em}}

\newcommand{\dstrutgraph}[2]{\raisebox{-0.7em}{
\begin{overpic}[height=2em]{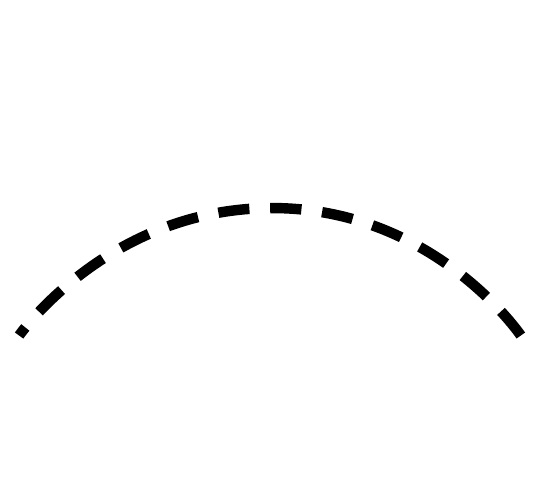}\scriptsize
 \put (-10,-5){$#1$}
 \put (85,-5){$#2$}
\end{overpic}}\hspace{0.5em}}

\newcommand{\uYgraph}[3]{\hspace{0.2em}\raisebox{-0.7em}{
\begin{overpic}[height=2em]{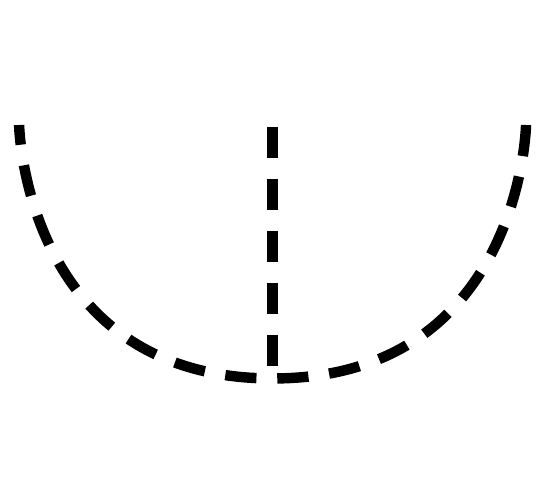}\scriptsize
 \put (-10,80){$#1$}
 \put (40,80){$#2$}
 \put (90,80){$#3$}
\end{overpic}}}

\newcommand{\uHgraph}[4]{\raisebox{-0.9em}{
\begin{overpic}[height=2em]{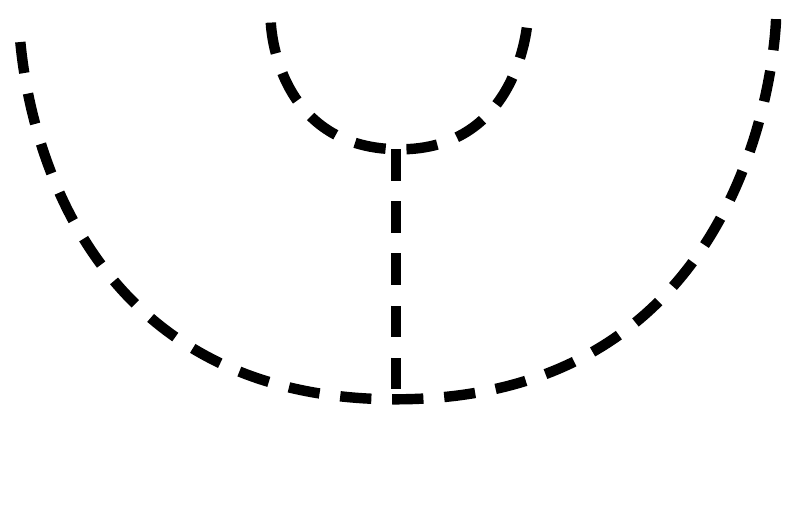}\scriptsize
 \put (-5,70){$#1$}
 \put (25,70){$#2$}
 \put (60,70){$#3$}
 \put (90,70){$#4$}
\end{overpic}}}

\newcommand{\dYgraph}[3]{\raisebox{-0.7em}{
\begin{overpic}[height=2em]{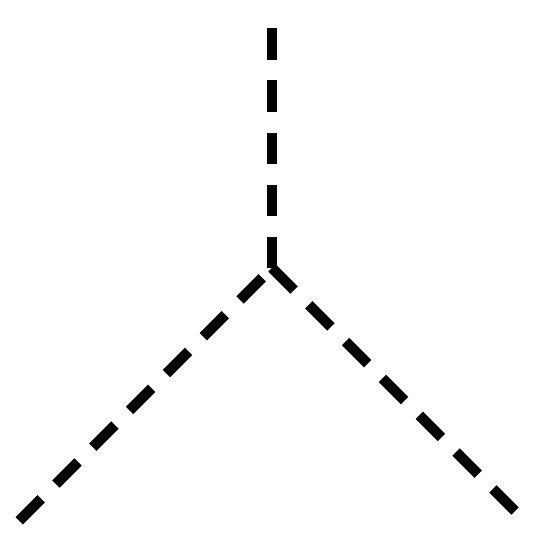}\scriptsize
 \put (60,80){$#1$}
 \put (0,-30){$#2$}
 \put (90,-30){$#3$}
\end{overpic}}}

\newcommand{\ddYgraph}[3]{\hspace{0.2em}\raisebox{-0.7em}{
\begin{overpic}[height=2em]{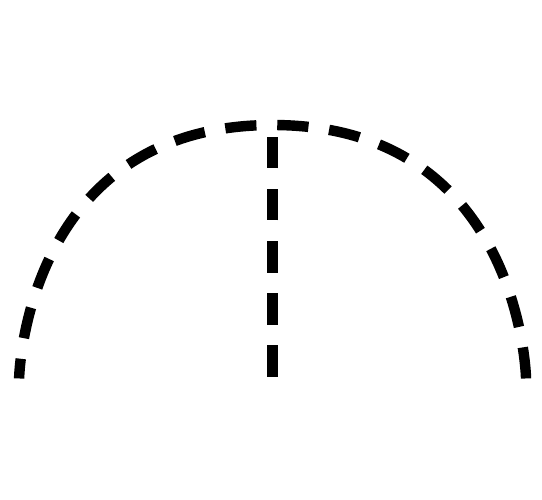}\scriptsize
 \put (-10,-10){$#1$}
 \put (40,-10){$#2$}
 \put (90,-10){$#3$}
\end{overpic}}}

\begin{document}

\title[Abelian quotients of the $Y$-filtration via the LMO functor]{Abelian quotients of the $Y$-filtration on the homology cylinders via the LMO functor}

\author{Yuta Nozaki}
\address{Organization for the Strategic Coordination of Research and Intellectual Properties, Meiji University \\
4-21-1 Nakano, Nakano-ku, Tokyo, 164-8525 \\
Japan
\newline
Current address:
Graduate School of Advanced Science and Engineering, Hiroshima University \\
1-3-1 Kagamiyama, Higashi-Hiroshima City, Hiroshima, 739-8526 \\
Japan}
\email{nozakiy@hiroshima-u.ac.jp}

\author{Masatoshi Sato}
\address{
Department of Mathematics,
Tokyo Denki University,
5 Senjuasahi-cho, Adachi-ku, Tokyo 120-8551, Japan}
\email{msato@mail.dendai.ac.jp}

\author{Masaaki Suzuki}
\address{Department of Frontier Media Science, Meiji University \\
4-21-1 Nakano, Nakano-ku, Tokyo, 164-8525 \\
Japan}
\email{macky@fms.meiji.ac.jp}


\maketitle

\begin{abstract}
We construct a series of homomorphisms from the $Y$-filtration on the monoid of homology cylinders to torsion modules via the mod $\Z$ reduction of the LMO functor.
The restriction of our homomorphism to the lower central series of the Torelli group does not factor through Morita's refinement of the Johnson homomorphism.
We use it to show that the abelianization of the Johnson kernel of a closed surface has torsion elements.
We also determine the third graded quotient $Y_3\I\C_{g,1}/Y_4$ of the $Y$-filtration.
\end{abstract}

\setcounter{tocdepth}{1}
\tableofcontents

\section{Introduction}
\label{section:Intro}
Let $\M = \M_{g,1}$ denote the mapping class group of an oriented surface $\Sigma_{g,1}$ of genus $g$ with one boundary component,
and let $H=H_1(\Sigma_{g,1};\Z)$.
Choosing a symplectic basis of $H$, the group of automorphisms of $H$ preserving the intersection form is identified with $\Sp(2g;\Z)$.
The \emph{Torelli group} $\I = \I_{g,1}$ is the kernel of the homomorphism
$\M\to \Sp(2g;\Z)$ derived from the action on $H$.
Let $\I(n)$ be the $n$-th subgroup of its lower central series, that is, $\I(1)=\I$ and $\I(n+1)=[\I(n),\I]$.
When $g$ is sufficiently large,
the $\Sp(2g;\Z)$-module $\I/\I(2)$ is determined in \cite[Theorem~3]{Joh85III},
and the $\Sp(2g;\Q)$-module $(\I(n)/\I(n+1))\otimes\Q$ is determined in \cite[Proposition~6.3]{Mor99} for $n=3$, in \cite[Theorem~1.2]{MSS20} for $n=4,5,6$.
Especially, for $n=1,3,4,5,6$,
the $n$-th Johnson homomorphism $\tau_n$ induces an injective homomorphism
\[
(\I(n)/\I(n+1))\otimes \Q\to (H\otimes L_{n+1})\otimes \Q,
\]
where $L_n$ is the degree $n$ part of the free Lie algebra generated by $H$.
When $n=2$, $(\I(2)/\I(3))\otimes \Q$ is determined in \cite[Theorem~10.1]{Hai97},
and it is shown that the kernel of the homomorphism
$\tau_2\colon (\I(2)/\I(3))\otimes \Q\to (H\otimes L_3)\otimes \Q$ is of rank $1$, which is detected by the Casson invariant as explained in \cite{Mor89}.
Moreover, a presentation of the associated graded Lie algebra $\bigoplus_{n=1}^\infty\I(n)/\I(n+1)\otimes\Q$ is given in \cite[Theorem~11.1]{Hai97}.
See \cite{Mor99} and \cite{MSS20},
for details of the results on $\I(n)/\I(n+1)\otimes\Q$ mentioned above.

In this paper, we study the torsion subgroup of the graded quotient $\I(n)/\I(n+1)$.
Our key tool is the LMO functor which is regarded as an invariant of 3-dimensional cobordisms.
In particular, we focus on a cobordism $(M,m)$,
where $M$ is homologically the product $\Sigma_{g,1}\times [-1,1]$,
and $m$ is a homeomorphism from $\partial(\Sigma_{g,1}\times[-1,1])$ to $\partial M$.
Such a cobordism is called a \emph{homology cylinder} over $\Sigma_{g,1}$,
and the set $\I\C = \I\C_{g,1}$ of homology cylinders over $\Sigma_{g,1}$ is a monoid
with the multiplication $(M,m)\circ(N,n)$ defined by stacking $N$ on $M$.

Here $\I\C$ is related with $\I$ via the monoid homomorphism $\cc\colon \I\to \I\C$ defined by 
\[
[f]\mapsto (\Sigma_{g,1}\times[-1,1],
\id_{\Sigma_{g,1}}\times\{-1\}\cup\id_{\partial\Sigma_{g,1} \times [-1,1]}\cup f\times\{1\}).
\]
Since this map is known to be injective, one can study $\I$ via $\I\C$.
It is natural to consider a descending series of submonoids of $\I\C$ related to the lower central series of $\I$.
Goussarov and Habiro introduced the $Y_n$-equivalence relation among cobordisms and the filtration $Y_n\I\C$ satisfying $\cc(\I(n)) \subset Y_n\I\C$.
Moreover, $\cc$ induces a homomorphism $\Gr\cc\colon\I(n)/\I(n+1) \to Y_n\I\C/Y_{n+1}$ between abelian groups.

We next review the LMO functor,
which is a main tool to construct our abelian quotients.
Kontsevich introduced an invariant of knots in $S^3$ which takes values in the series of Jacobi diagrams, where a Jacobi diagram is a uni-trivalent graph with additional information.
The Kontsevich invariant unifies finite type invariants of knots,
and was extended to an invariant of tangles.

Le, Murakami and Ohtsuki constructed the LMO invariant of closed 3-manifolds which is universal among $\Q$-valued finite type invariants of rational homology 3-spheres (see \cite{Oht02} for instance).
Using the Kontsevich invariant, Cheptea, Habiro and Massuyeau~\cite{CHM08} extended the LMO invariant to the case of 3-manifolds with boundary.
More precisely,
they introduced the LMO functor $\Ztilde$ from a certain category of cobordisms between surfaces with connected boundary to a category of series of Jacobi diagrams.

The LMO functor $\Ztilde$ is related to the Johnson homomorphisms $\tau_n$ on homology cylinders.
In fact, for $M=(M,m) \in \I\C$,
if the leading term of the 0-loop (namely, tree) part of the $Y$-part $\Ztilde^Y(M)$ has degree $n$, then it coincides with $-\tau_n(M)$.
Note that the sign of $\tau_n$ is different from \cite{CHM08} (compare $\rho_n$ in Section~\ref{section:satolevine} with \cite[(8-3)]{CHM08}).
Here we briefly mention Jacobi diagrams, where each univalent vertex is colored by an element of the set $\{1^{\pm},2^{\pm},\ldots,g^{\pm}\}$ of $2g$ elements.
A connected Jacobi diagram without trivalent vertices is called a \emph{strut}.
Let $\A_n^c$ (resp. $\A_n^Y$) denote the $\Z$-module generated by connected Jacobi diagrams (resp. by Jacobi diagrams without struts) with $n$ trivalent vertices modulo some relations (see Section~\ref{sec:jacobi}).

Let us see a relation between the $Y$-part $\Ztilde^Y$ of the LMO functor (see Section~\ref{sec:LMO}) and $Y$-filtration.
By \cite[Theorem~7.11]{CHM08}, $\Ztilde^Y_{<n}(M)$ is the empty diagram $\emptyset$ for $M \in Y_n\I\C$, and $\Ztilde^Y_{n}$ induces an isomorphism $(Y_n\I\C/Y_{n+1})\otimes\Q \to \A^Y_n\otimes\Q$.
It is natural to study for $M \in Y_n\I\C$ the next term $\Ztilde^Y_{n+1}(M)$.
In fact, we consider ``$\Ztilde^Y_{n+1} \bmod \Z$'', which is invariant under the $Y_{n+1}$-equivalence.

For $n\ge1$, $\F_n\I\C$ is a filtration on the monoid ring $\Z\I\C$ (see Section~\ref{sec:F-filtration}).
We define homomorphisms
$$
\overline{Z}_{n+1}\colon \F_{n}\I\C/\F_{n+1}\I\C \to \A_{n+1}^Y\otimes \Q/\Z \text{\ and\ } \bar{z}_{n+1}\colon Y_{n}\I\C/Y_{n+1} \to \A_{n+1}^c\otimes \Q/\Z
$$
to be the mod $\Z$ reductions of the degree $(n+1)$-st parts of $\Ztilde^Y(M)$ and $\log \Ztilde^Y(M)$, respectively.
These homomorphisms are independent of the choice of an associator used to define the LMO functor (see Remark~\ref{rem:Associator}).
On the other hand, they depend on the choice of a basis of $\pi_1(\Sigma_{g,1})$.

In order to compute $\overline{Z}_{n+1}$ and $\bar{z}_{n+1}$,
we introduce a homomorphism $\delta$ which is explicitly written as an operation on Jacobi diagrams in Section~\ref{sec:delta}.
More precisely, we give formulas of our homomorphisms for 3-manifolds obtained by surgery maps $\S\colon \A^Y_n \to \F_n\I\C/\F_{n+1}\I\C$ and $\ss\colon \A^c_n \to Y_n\I\C/Y_{n+1}$ (see Section~\ref{sec:surgery}).
The following theorem is proved in Section~\ref{sec:Z_and_z}.

\begin{theorem}
\label{thm:main}
The following diagrams commute:
\begin{center}
$
\begin{CD}
\A_n^Y@>\S>>\F_n\I\C/\F_{n+1}\I\C\\
@V\delta+\Y VV@VV\overline{Z}_{n+1}V\\
\A_{n+1}^Y\otimes\Z/2\Z@>\id\otimes\frac{1}{2}>> \A_{n+1}^Y\otimes\Q/\Z,
\end{CD}
$
\quad
$
\begin{CD}
\A_n^c@>\ss>>Y_n\I\C/Y_{n+1}\\
@V\delta VV@VV\bar{z}_{n+1}V\\
\A_{n+1}^c\otimes\Z/2\Z@>\id\otimes\frac{1}{2}>> \A_{n+1}^c\otimes\Q/\Z.
\end{CD}
$
\end{center}
Here, $\Y\colon \A_{n}^Y \to \A_{n+1}^Y$ is defined by $\mathbb{Y}(J) = \sum_{Y\subset J} J\sqcup Y$, where $Y$ runs over every connected component of $J$ with one trivalent vertex.
\end{theorem}

We have two applications of the homomorphisms
$\bar{z}_{n+1}\colon Y_n\I\C/Y_{n+1}\to\A_{n+1}^c\otimes\Q/\Z$ and
$\delta\colon \A_n^Y\to\A_{n+1}^Y\otimes \Z/2\Z$
which are proved in Section~\ref{section:applications}.
The first one is with respect to abelian quotients of subgroups of the Torelli group.
Denote by $\pi_1\Sigma_{g,1}(n)$ the $n$-th subgroup of the lower central series of $\pi_1\Sigma_{g,1}$.
A filtration of the mapping class group whose $n$-th term $J_n\M = \Ker(\M\to \Aut(\pi_1\Sigma_{g,1}/\pi_1\Sigma_{g,1}(n+1)))$ is called the \emph{Johnson filtration}.
Since the Johnson filtration is central, $J_n\M$ contains $\I(n)$.
Recall that the $n$-th Johnson homomorphism 
\[
\tau_n\colon J_n\M\to \Hom(H, \pi_1\Sigma_{g,1}(n+1)/\pi_1\Sigma_{g,1}(n+2))
\]
is defined by $\tau_n(\varphi)([\gamma])=\varphi(\gamma)\gamma^{-1}$.
\begin{theorem}
\label{thm:restriction-torelli}
Let $1\le n\le g-2$.
Each of the $0$-loop part and $1$-loop part of $\bar{z}_{2n} \colon Y_{2n-1}\I\C/Y_{2n}\to \A_{2n}^c\otimes\Q/\Z$ restricted to
\[
\Im(\I(2n-1)/\I(2n)\to Y_{2n-1}\I\C/Y_{2n}) \cap \tor(Y_{2n-1}\I\C/Y_{2n})
\]
is a non-trivial homomorphism.
Especially, neither the $0$-loop part nor $1$-loop part of $\bar{z}_{2n}\circ \Gr\cc\colon \I(2n-1)/\I(2n)\to\A_{2n}^c\otimes\Q/\Z$ factors through the Johnson homomorphism $\tau_{2n-1}$.
\end{theorem}

By Theorem~\ref{thm:restriction-torelli},
the induced homomorphism on $\I(2n-1)/\I(2n)$ by $\tau_{2n-1}$ is not injective.
Since $\Ker\tau_n=J_{n+1}\M$, we have the following.
When $n=1$, it is a corollary of \cite[Theorem~3]{Joh85III}.
\begin{cor}
Let $1\le n\le g-2$.
The natural homomorphism
$\I(2n-1)/\I(2n)\to J_{2n-1}\M/J_{2n}\M$ induced by the inclusion is not injective.
\end{cor}

Morita's refinement 
\[
\tilde{\tau}_n\colon J_n\M \to H_3(\pi_1\Sigma_{g,1}/\pi_1\Sigma_{g,1}(n+1))
\]
of the $n$-th Johnson homomorphism is introduced in \cite{Mor93}.
There are several equivalent homomorphisms such as Heap's homomorphism in \cite{Hea06} and Massuyeau's infinitesimal analogue in \cite[Section~4]{Mas12} (see also \cite[Section~3.3]{HaMa12}).
As we show in Theorem~\ref{thm:Z-homo},
$\bar{z}_{2n}\colon Y_{2n-1}\I\C\to \A_{2n}^c\otimes\Q/\Z$ extends to a homomorphism on $Y_n\I\C$.
\begin{cor}\label{cor:kerneljohnson}
For $1\le n\le g-2$,
each of the $0$-loop part and $1$-loop part of the homomorphism $\bar{z}_{2n}\colon Y_{n}\I\C \to \A_{2n}^c\otimes\Q/\Z$ restricted to $\I(n)$ is non-trivial,
and the restriction does not factor through Morita's refinement $\tilde{\tau}_n$.
\end{cor}
Let $\I\H = \I\H_{g,1}$ be the homology cobordism group of homology cylinders over $\Sigma_{g,1}$. 
We see that the $0$-loop part of $\bar{z}_{2n}$ factors through $Y_{2n-1}\I\H/Y_{2n}$ in Section~\ref{section:satolevine}
and that the 1-loop part does not in Section~\ref{section:proof-mainthm},
where $Y_n\I\H$ is the image of the $Y$-filtration $Y_n\I\C$ under the projection $q\colon \I\C\to\I\H$.

Let $\K\C$ and $\K$ denote the kernels of the first Johnson homomorphisms on $\I\C$ and $\I=J_1\M$, respectively.
In other words, $\K=J_2\M$.
We review a homomorphism $\K\C \to \Z$ arising from the Casson invariant (see \cite[Section~1]{MaMe13} or \cite[Section~6.8]{CHM08}).
Let $e \colon \Sigma_{g,1}\cup_\partial\Sigma_{0,1} \hookrightarrow S^3$ be an embedding such that its image is a Heegaard surface.
Then, for $(M,m) \in \K\C$, one can obtain an integral homology 3-sphere $M_e$ by replacing a tubular neighborhood $e(\Sigma_{g,1}) \times [-1,1] \subset S^3$ with $M$ via the homeomorphism $m$.
The Casson invariant of $M_e$ gives homomorphisms $\K\C \to \Z$ which depend on the choice of $e$,
and the difference of two such homomorphisms restricted to $\K$ is known to be written in terms of the second Johnson homomorphism $\tau_2$.
See \cite[Theorem 6.1(iii)]{Mor89}, for details. 

For $n\ge1$,
let $O(a_1,a_2,a_3,\ldots, a_n)\in\A_{n}^c$ denote the $1$-loop Jacobi diagram
\[
O(a_1,a_2,a_3,\ldots, a_n)=
\begin{tikzpicture}[baseline=1.4ex, scale=0.25, dash pattern={on 2pt off 1pt}]
\draw (0,0) circle [radius=2]; 
\draw (10:2) -- (10:4);
\draw (50:2) -- (50:4);
\draw (90:2) -- (90:4);
\draw (130:2) -- (130:4);
\node [right] at (10:3.8) {$a_3$};
\node [above right] at (50:3.6) {$a_2$};
\node [above] at (90:3.8) {$a_1$};
\node [above left] at (130:3.6) {$a_n$};
\node at (-10:3){$\cdot$};
\node at (-25:3){$\cdot$};
\node at (-40:3){$\cdot$};
\node at (150:3){$\cdot$};
\node at (165:3){$\cdot$};
\end{tikzpicture}
\]
for $a_1,a_2,a_3,\ldots,a_n\in\{1^\pm,2^\pm,\ldots,g^\pm\}$.
\begin{theorem}\label{thm:JohnsonKernel}
For $g\ge 4$, the degree $4$ part of $\Ztilde^Y$ induces a homomorphism
\[
\bar{z}_4\colon\K\C \to \frac{\A_4^c\otimes\Q/\Z}{\braket{\frac{1}{2}\delta O(a,b,b) \mid a,b \in \{1^{\pm},2^{\pm},\ldots, g^{\pm}\}}},
\]
and its restriction to the intersection of
$\Ker\tilde{\tau}_2\subset \K$
and the kernels of the homomorphisms induced by the Casson invariant is non-trivial.
\end{theorem}

The Johnson homomorphism can be considered also for a closed surface of genus $g$, and let $\K_g$ denote the kernel of the first Johnson homomorphism.
In \cite{MSS20},
Morita, Sakasai, and the third author determined the first homology $H_1(\K_g;\Q)$,
and showed that $\tilde{\tau}_2$ and the Casson invariant give an isomorphism between $H_1(\K_g;\Q)$ and their images for $g\ge6$.
We see that Theorem~\ref{thm:JohnsonKernel} holds also in the case of closed surfaces, and conclude the following.

\begin{cor}
\label{cor:H1Kg}
For $g\ge6$, the torsion subgroup of $H_1(\K_g;\Z)$ is non-trivial.
\end{cor}

The second application of $\bar{z}_{n+1}$ and $\delta$ is the determination of the module $Y_3\I\C/Y_4$.
Massuyeau and Meilhan showed that the natural inclusion $\cc \colon \I\to \I\C$ induces an isomorphism $H_1(\I;\Z)\cong \I\C/Y_2$ in \cite[Theorem~1.3]{MaMe03}.
They also investigated the $Y_3$-equivalence classes,
and showed that the surgery map $\ss$ gives an isomorphism $\A_2^c \cong Y_2\I\C/Y_3$ in \cite[Corollary~5.1]{MaMe13}.
We determine the abelian group $Y_3\I\C/Y_4$ as described below.
After fixing a symplectic basis $\{\alpha_1,\beta_1,\dots,\alpha_g,\beta_g\}$ of $H$, we identify $H$ with the free module generated by the set $\{1^{\pm},2^{\pm},\ldots,g^{\pm}\}$ of labels via the correspondence $\alpha_i \leftrightarrow i^{-}$ and $\beta_i \leftrightarrow i^{+}$.
For $n\ge3$, let us denote by $T(a_1,a_2,\ldots, a_n)\in \A_{n-2}^c$ the tree Jacobi diagram as
\[
T(a_1,a_2,\ldots, a_n)=
\begin{tikzpicture}[baseline=2ex, scale=0.25, dash pattern={on 2pt off 1pt}]
\node [left] at (0,0){$a_1$};
\node [above] at (2,2){$a_2$};
\node [above] at (4,2){$a_3$};
\node [below] at (6,2){$\cdots$};
\node [below] at (8.3,2){$\cdot$};
\node [below] at (9,2){$\cdot$};
\node [above] at (10,2){$a_{n-1}$};
\node [right] at (12,0){$a_n$.};
\draw (0,0) -- (12,0);
\draw (2,0) -- (2,2);
\draw (4,0) -- (4,2);
\draw (10,0) -- (10,2);
\end{tikzpicture}
\]

\begin{theorem}\label{thm:str-of-y3}
There is an exact sequence
\[
\begin{CD}
0@>>>
(\Lambda^3H\oplus \Lambda^2H)\otimes\Z/2\Z@>j>>
\A_3^c@>\ss>>
Y_3\I\C/Y_4@>>>0
\end{CD}
\]
of abelian groups, where $j$ is defined by
\begin{align*}
&j(a\wedge b\wedge c)=T(a,b,c,b,a)+T(b,c,a,c,b)+T(c,a,b,a,c),\\
&j(a\wedge b)=O(a,b,b)+O(b,a,a)
\end{align*}
for $a,b,c\in\{1^{\pm},\ldots,g^{\pm}\}$. 
Especially, we have
\[
Y_3\I\C/Y_4\cong
(L_3\oplus S^2H)\otimes\Z/2\Z\oplus (D_3\oplus \Lambda^3 H),
\]
where $D_3=\Ker(H\otimes L_4\to L_5)$ is the kernel of the bracket map.
\end{theorem}

This theorem is related with the Goussarov-Habiro conjecture of finite-type invariants of homology cylinders discussed in Section~\ref{sec:FTI}.

\subsection*{Acknowledgments}
The authors would like to thank \mbox{Gw\'ena\"el Massuyeau}, \mbox{Jean-Baptiste Meilhan} and \mbox{Kazuo Habiro} for useful discussions,
and \mbox{Takuya Sakasai} for informing them about the higher Sato-Levine invariant. 
This study was supported in part by JSPS KAKENHI 18K03310 and 16K05159.
The first author was supported by Iwanami Fujukai Foundation.

\section{Preliminaries}
We establish some notation on homology cylinders, graph claspers, and Jacobi diagrams following \cite{Hab00C} and \cite{CHM08},
and review the definition of the LMO functor.

\subsection{Homology cylinders and the homology cobordism group}
\label{sec:homology cylinder}
A \emph{homology cylinder} over $\Sigma_{g,1}$ is a compact connected oriented 3-manifold $M$ equipped with an orientation-preserving homeomorphism
$m\colon \partial(\Sigma_{g,1}\times [-1,1])\to \partial M$
such that the two inclusions
$m_{\pm}=m|_{\Sigma_{g,1}\times\{\pm1\}}\colon\Sigma_{g,1}\times\{\pm1\}\to M$ induce the same isomorphism on their integral homology groups. 
Two homology cylinders $(M, m)$ and $(M', m')$ are said to be \emph{equivalent} if there exists an orientation-preserving
homeomorphism $\Phi\colon M\to M'$ such that $\Phi\circ m=m'$.
A homology cylinder is a special case of a cobordism which is a 3-manifold whose boundary consists of $\Sigma_{g_{+},1}$, $\Sigma_{g_{-},1}$ and annulus without any homological condition (see \cite[Definition~2.2]{CHM08} for the precise definition).

For two homology cylinders $(M,m)$ and $(M',m')$,
we obtain a new homology cylinder $M\circ M'$
by pasting the top of $M$ and the bottom of $M'$ using the identification $m'_-\circ (m_+)^{-1}$.
More generally, for two cobordisms $M$ from $\Sigma_{q,1}$ to $\Sigma_{r,1}$ and $M'$ from $\Sigma_{p,1}$ to $\Sigma_{q,1}$,
we obtain a new cobordism $M\circ M'$ from $\Sigma_{p,1}$ to $\Sigma_{r,1}$ in the same way.
We denote by $\I\C$ the set of equivalence classes of homology cylinders,
which forms a monoid with this composition.

Two homology cylinders $M$ and $M'$ over $\Sigma_{g,1}$ are said to be \emph{homology cobordant} if the closed oriented 3-manifold $M\cup_{m \circ (m')^{-1}} (-M')$ bounds a compact oriented smooth 4-manifold $W$ such that the inclusions induce isomorphisms between their integral homology groups.
We denote by $\sim_H$ this equivalence relation,
and also denote the quotient by $\I\H=\I\C/{\sim_H}$,
which is called the \emph{homology cobordism group} of homology cylinders.
We denote the projection map by $q\colon\I\C\to\I\H$.

\subsection{Bottom-top tangles}
\label{sec:bt_tangle}
For  $g\ge1$,
choose $g$ pairs of points $(p_1,q_1),\ldots,(p_g,q_g)$ in $[-1,1]^2$ uniformly in the first direction as $p_i=(\frac{2(2i-g)-3}{2g+1},0)$ and $q_i=(\frac{2(2i-g)-1}{2g+1},0)$.
Let us call a homology cylinder over $[-1,1]^2$ a \emph{homology cube}.
Let $B$ be a homology cube and identify $\partial[-1,1]^3$ with $\partial B$.
For $g,h\ge1$,
let $\gamma=(\gamma^+,\gamma^-)$ be a framed oriented tangle in $B$
with $g$ top components $\gamma_1^{+},\ldots, \gamma_g^{+}$ and $h$ bottom components $\gamma_1^{-},\ldots, \gamma_h^{-}$
such that each $\gamma_j^{-}$ runs from $q_j\times\{-1\}$ to $p_j\times\{-1\}$,
and each $\gamma_j^{+}$ runs from $p_j\times\{1\}$ to $q_j\times\{1\}$.
Such a pair $(B,\gamma)$ is called a \emph{bottom-top tangle of type $(g,h)$} in a homology cube.
We give examples of bottom-top tangles in Table~\ref{tab:ElemCob}.

If we are given a bottom-top tangle $(B,\gamma)$ of type $(g,h)$ in a homology cube,
by digging $B$ along the tangle $\gamma$,
we obtain a cobordism $(M,m)$ from the top surface $\Sigma_{g,1}$ to the bottom surface $\Sigma_{h,1}$, where the homeomorphism $m \colon \Sigma_{g,1} \cup (S^1\times[-1,1]) \cup \Sigma_{h,1} \to \partial M$ is uniquely determined (up to isotopy) by the framing of the tangle $\gamma$.
See \cite[Theorem~2.10]{CHM08}, for details.
Assume that $g=h$ and that the linking matrix
\[
\Lk(\gamma)=
\begin{pmatrix}
O_g&I_g\\
I_g&O_g
\end{pmatrix},
\]
where $O_g$ and $I_g$ are the zero matrix and identity matrix of size $g$, respectively.
In this case, we obtain a homology cylinder over $\Sigma_{g,1}$ (see \cite[Section~8.1]{CHM08}).
Let $\alpha_1, \beta_1,\ldots, \alpha_g,\beta_g$ be the oriented simple closed curves in Figure~\ref{fig:Sigma_g1}.
\begin{figure}[h]
 \centering
 \includegraphics[width=0.5\textwidth]{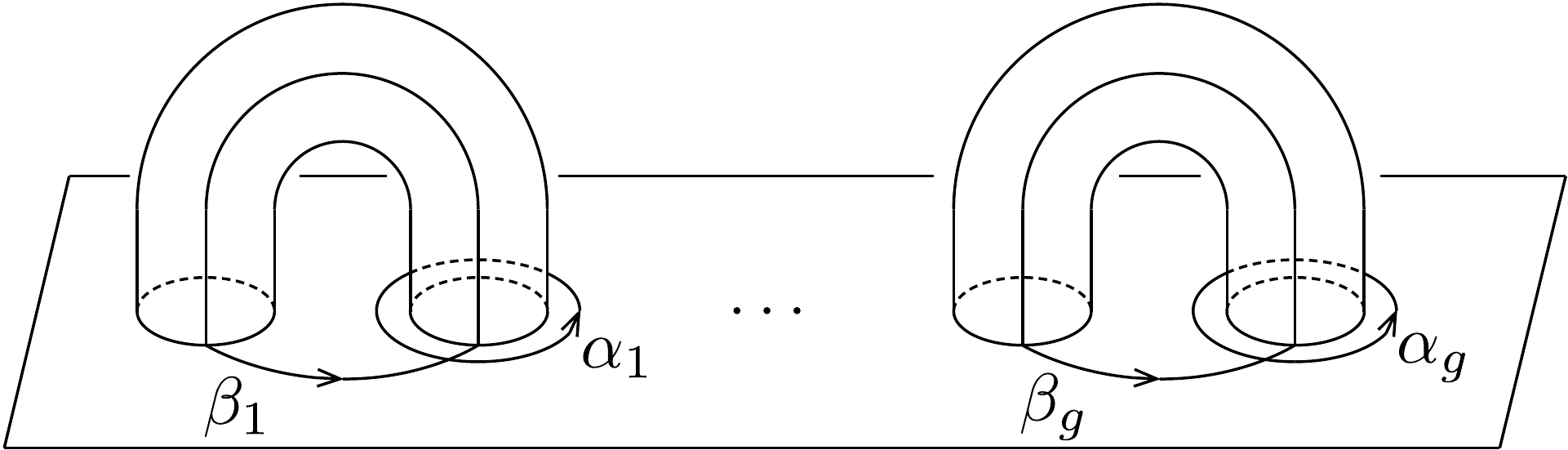}
 \caption{Oriented simple closed curves $\alpha_i$ and $\beta_i$ on $\Sigma_{g,1}$.}
 \label{fig:Sigma_g1}
\end{figure}
Conversely,
if we are given a cobordism $M$ from $\Sigma_{g,1}$ to $\Sigma_{h,1}$ satisfying some homological condition,
by attaching 3-dimensional 2-handles to the boundary of $M$ along each of $\beta_1,\ldots,\beta_g$ in the top surface and $\alpha_1,\ldots,\alpha_h$ in the bottom surface,
we obtain a homology cube $B$.
Denoting by $\gamma$ the co-cores of these 2-handles,
we obtain a bottom-top tangle $(B,\gamma)$ in a homology cube.

Under this correspondence,
the composition of cobordisms $M$ from $\Sigma_{q,1}$ to $\Sigma_{r,1}$ and $M'$ from $\Sigma_{p,1}$ to $\Sigma_{q,1}$ induces that of two bottom-top tangles $\gamma$ of type $(q,r)$ and $\gamma'$ of type $(p,q)$.
We denote the composition by $\gamma\circ\gamma'$,
which is of type $(p,r)$,
and the composition is explicitly described in \cite[Section~2.3]{CHM08}.

\subsection{Graph claspers}\label{section:graph clasper}
Goussarov~\cite{GGP01} and Habiro~\cite{Hab00C} independently constructed tools to study finite-type invariants of links and homology cylinders, 
which are called \emph{clovers} and \emph{claspers}, respectively.
Here, we explain graph claspers which Habiro called allowable graph claspers.

Let $M$ be a compact oriented 3-manifold.
A \emph{graph clasper} $G$ is a (not necessarily connected) compact orientable surface embedded into the interior of $M$
which decomposes into 3 kinds of constituents called \emph{nodes}, \emph{leaves}, and \emph{edges}.
Nodes, leaves, and edges are homeomorphic to disks, annuli, and $[0,1]^2$'s, respectively.
Each edge is glued to exactly two constituents consisting of nodes and leaves along $[0,1]\times \{0\}$ and $[0,1]\times \{1\}$, respectively,
and satisfies the conditions as follows.
Each leaf should be connected to a single node and to no leaf by an edge,
and each node should be connected to exactly three constituents consisting of nodes or leaves.
In other words,
a graph clasper corresponds to a (not necessarily connected) uni-trivalent graph
by assigning nodes, leaves, and edges to trivalent vertices, univalent vertices, and edges of the graph, respectively.
Note that, by the condition we assigned on leaves,
each connected component of a graph clasper has at least one node.
By changing nodes, leaves, and edges into a framed link as in Figure~\ref{fig:clasper} with blackboard framings,
we obtain a framed link $L(G)$ in $M$, and we denote by $M_G$ the 3-manifold obtained by surgery along $L(G)$.
\begin{figure}
 \centering
 \includegraphics[width=0.5\textwidth]{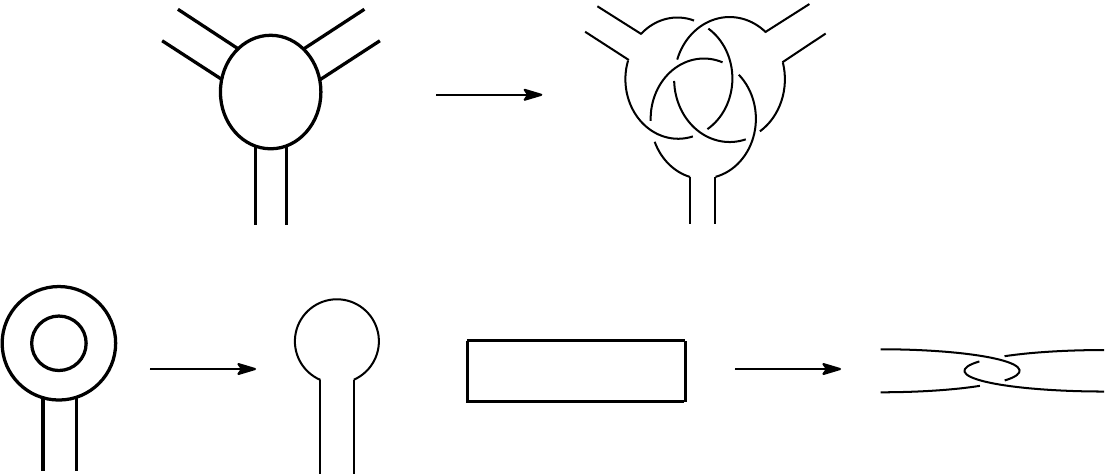}
 \caption{Changing nodes, leaves, and edges into a framed links.}
 \label{fig:clasper}
\end{figure}

\subsection{$Y_n$-equivalence and the $Y$-filtration}
Two homology cylinders $M$ and $N$ are said to be \emph{$Y_n$-equivalent} if one can pass from $M$ to $N$ by surgery on connected graph claspers in $M$ each of which has $n$ nodes.
For details, see \cite{CHM08} and \cite{Hab00C}.

Let us denote a submonoid of $\I\C$ as
\[
Y_n\I\C=\{M\in\I\C \mid \text{$M$ is $Y_n$-equivalent to $\Sigma_{g,1}\times[-1,1]$}\}.
\]
Move~2 in \cite[p.~14]{Hab00C},
which is the same as in \cite[Theorem~2.4]{GGP01},
or Move~9 in \cite[p.~15]{Hab00C}
implies that,
for $1\le k\le l$, if two homology cylinders are $Y_l$-equivalent, they are $Y_k$-equivalent.
Thus, we have a filtration 
\[
\I\C=Y_1\I\C\supset Y_2\I\C\supset Y_3\I\C\supset \cdots
\]
on the monoid $\I\C$ which is called the \emph{$Y$-filtration} on $\I\C$.

As shown in \cite[Theorem~3]{Gou99} and \cite[Section~8.5]{Hab00C},
$\I\C/Y_{n+1}$ is a finitely generated nilpotent group,
and $Y_{n}\I\C/Y_{n+1}$ is an abelian group for $n\ge1$.
Moreover, the graded quotient
\[
\Gr^Y\I\C=\bigoplus_{n\geq 1}Y_n\I\C/Y_{n+1}
\]
forms a graded Lie algebra with the Lie bracket defined by
\[
[\overline{M},\overline{N}]=\overline{M\circ N\circ M'\circ N'}\in Y_{i+j}\I\C/Y_{i+j+1}
\]
for $M\in Y_i\I\C$ and $N\in Y_j\I\C$,
where $M'$ and $N'$ are representatives of the inverse elements of $\overline{M}\in Y_i\I\C/Y_{i+1}$ and $\overline{N}\in Y_j\I\C/Y_{j+1}$.
By setting $Y_n\I\H=q(Y_n\I\C)$,
the projection $q\colon \I\C\to\I\H$ induces a homomorphism $\Gr q\colon \bigoplus_{n=1}^\infty Y_n\I\C/Y_{n+1}\to \bigoplus_{n=1}^\infty Y_n\I\H/Y_{n+1}$.

\subsection{The $\F$-filtration on $\Z\I\C$}
\label{sec:F-filtration}
Let $\Z\I\C$ be the monoid ring of $\I\C$,
that is, the free module generated by the set $\I\C$ of equivalence classes of homology cylinders equipped with the product induced by $\I\C$.
For $M\in\I\C$ and a graph clasper $G$ in $M$,
let $[M,G]$ denote the \emph{surgery bracket} defined by
\[
[M,G]=\sum_{G'\subset G}(-1)^{|G'|}M_{G'}\in \Z\I\C,
\]
where $G'$ runs through all the subsets of the set of connected components of $G$.
The \emph{$\F$-filtration} $\F_n\I\C$ on $\Z\I\C$ is defined to be the submodule
\[
\langle [M,G] \mid M\in\I\C\text{ and }G\text{ is a graph clasper in }M\text{ with }n\text{ nodes} \rangle
\]
spanned by surgery brackets for $n\ge 1$.

\subsection{The module of Jacobi diagrams}
\label{sec:jacobi}
Let $C$ be a finite set.
We consider a uni-trivalent graph whose each univalent vertex is colored with an element in $C$,
and each trivalent vertex is equipped with a cyclic order of its incident edges.
We call such a uni-trivalent graph a \emph{Jacobi diagram}.
and we denote by $\A(C)$ the abelian group defined as
\[
\A(C)
=\frac{\Z\{\text{Jacobi diagrams colored with } C\}}{\AS,\ \IHX,\ \textrm{self-loop}},
\]
where the self-loop relation means that a Jacobi diagram with a self-loop is trivial in $\A(C)$,
and the AS and IHX relations are as follows:
\[
\AS\colon 
\begin{tikzpicture}[baseline=-0.4ex, scale=0.25, dash pattern={on 2pt off 1pt}]
\draw (0,0) -- (0,2);
\draw (0:0)
to [out=200, in=140] (-50:3);
\draw (0:0)
to [out=-20, in=50] (285:0.7);
\draw (260:1.05)
to [out=220, in=40] (230:3);
\end{tikzpicture}
=-\begin{tikzpicture}[baseline=-0.4ex, scale=0.26, dash pattern={on 2pt off 1pt}]
\draw (0:0) -- (90:2); 
\draw (0:0) -- (-50:2.5); 
\draw (0:0) -- (230:2.5); 
\node at (-50:2.5) [right]{,};
\end{tikzpicture}
\qquad
\IHX\colon 
\begin{tikzpicture}[baseline=2.6ex, scale=0.5, dash pattern={on 2pt off 1pt}]
\draw (0,0) to [out=45, in=135] (2,0); 
\draw (1,0.5) -- (1,1.5); 
\draw (0,2) to [out=-45, in=225] (2,2); 
\end{tikzpicture}\,-\,
\begin{tikzpicture}[baseline=2.6ex, scale=0.5, dash pattern={on 2pt off 1pt}]
\draw (0,0)  to [out=45, in=-45] (0,2); 
\draw (0.5,1) -- (1.5,1); 
\draw (2,0)  to [out=135, in=-135] (2,2); 
\end{tikzpicture}\,+\,
\begin{tikzpicture}[baseline=2.6ex, scale=0.5, dash pattern={on 2pt off 1pt}]
\draw (0,0) to (2,2); 
\draw (0,2) to (0.9,1.1);
\draw (1.1,0.9) to (2,0);
\draw (0.6,0.6) to (1.4,0.6); 
\end{tikzpicture}=0.
\]
Note that, everywhere in this paper,
we assume that trivalent vertices of Jacobi diagrams depicted are oriented counterclockwise.

We mainly consider the case when $C$ is the set $\{1^{\pm},\ldots, g^{\pm}\}$ of $2g$ elements,
and simply denote as $\A=\A(\{1^{\pm},\ldots, g^{\pm}\})$.
The \emph{internal degree} $\ideg(J)$ of a Jacobi diagram $J$ is defined to be the number of trivalent vertices in $J$.
Let us denote the submodule of $\A$ generated by the (not necessarily connected) Jacobi diagrams of $\ideg=n$ by $\A_n$.
Recall that a strut is a connected Jacobi diagram without trivalent vertices,
and that we denote by $\A^c$ (resp. $\A^Y$) the submodule of $\A$ generated by connected Jacobi diagrams (resp. by Jacobi diagrams without struts).
We also denote their $\ideg=n$ parts as $\A^c_n=\A^c\cap \A_n$ and $\A^Y_n=\A^Y\cap \A_n$.
The submodule of $\A_n^c$ generated by connected Jacobi diagrams with $k$ loops is denoted by $\A_{n,k}^{c}$, 
namely, connected Jacobi diagrams with Euler characteristic $1-k$.

Let $p,q,r$ be non-negative integers,
and let $J$ and $J'$ be Jacobi diagrams colored with $\{1^+, \ldots, q^+, 1^-, \ldots, p^-\}$ and $\{1^+, \ldots, r^+, 1^-, \ldots, q^-\}$, respectively.
We assume $J$ has no struts with labels in $\{1^+,\ldots, q^+\}$ in the both ends.
In \cite[Section~1.2]{CHM08}, such $J$ is said to be \emph{top-substantial}.
We denote the composition of $J$ and $J'$ as
\[
J\circ J'=\sum\left(
\text{\parbox{21em}{
all ways of gluing the $i^+$-colored vertices of $J$\\
to the $i^-$-colored vertices of $J'$ for all $1\le i\le q$
}}
\right)\in\A(\{1^+,\ldots,r^+,1^-,\ldots,p^-\}).
\]
There is another composition $\star$ on the submodule $\A^Y$.
For Jacobi diagrams $J$ and $J'$ colored with $\{1^\pm,\ldots, g^\pm\}$ without struts,
we define
\[
J\star J'=\sum\left(
\text{\parbox{23em}{
all ways of gluing some of $i^+$-colored vertices of $J$\\
to some of $i^-$-colored vertices of $J'$ for all $1\le i\le g$
}}
\right)\in\A^Y.
\]
The product $\star$ makes $\A^Y$ an associative algebra.
Actually, it is a co-commutative graded Hopf algebra. See \cite[Proposition~8.5]{CHM08}, for details.

Furthermore, we denote by $\widehat{\A}$ the completion of $\A$ with respect to the total degree of Jacobi diagrams, where the \emph{total degree} of a Jacobi diagram $J$ is defined to be half the number of vertices of $J$.
Then, for a series $x \in \widehat{\A}$ of non-empty Jacobi diagrams, $\exp(x) = \sum_{n=0}^\infty x^{\sqcup n}/n!$ and $\log(\emptyset+x) = \sum_{n=1}^\infty (-1)^{n-1}x^{\sqcup n}/n$ make sense, where $x^{\sqcup n}$ denotes the disjoint union of $n$ copies of $x$.
Note that the subspace $\widehat{\A}^Y$ coincides with the completion of $\A^Y$ with respect to the internal degree.

\subsection{The surgery map}
\label{sec:surgery}
Here, we review the surgery maps 
\[
\S\colon\A_n^Y\to \F_n\I\C/\F_{n+1}\I\C,\ 
\ss\colon\A_n^c\to Y_n\I\C/Y_{n+1}.
\]

As explained in Section~\ref{section:graph clasper},
we obtain a (not necessarily connected) compact orientable surface from $J\in \A_n^Y$ by assigning disks and annuli called nodes and leaves to trivalent and univalent vertices in $J$,
and by gluing them to unit squares corresponding to edges.
We embed the surface in $\Sigma_{g,1}\times [-1,1]$ and obtain a graph clasper $G$ as follows.
First, we identify $\Sigma_{g,1}\times [-1,1]$ with a bottom-top tangle $\gamma$ in a cube $[-1,1]^3$,
and embed the surface in $[-1,1]^3 \setminus \gamma$ instead of $\Sigma_{g,1}\times [-1,1]$.
Fix an orientation of the surface,
and embed each leaf along a meridian of the tangle which represents the color of the corresponding univalent vertex.
The annulus is assumed to be vertical to the tangle, namely, the self linking number of the core circle is zero with respect to the framing induced by the annulus.
Also, the orientation of the annulus is assumed to coincide with that of the normal bundle of the tangle in $\Sigma_{g,1}\times[-1,1]$.
Here, we take a sufficiently small meridian,
and assume that neither leaves nor edges go through the hole of the annulus. 
Second, we embed the nodes in an arbitrary way,
and connect nodes and leaves by edges in $\Sigma_{g,1}\times [-1,1]$ so that the orientations of the constituents are compatible.
In this way, we obtain a graph clasper $G$.
Define homomorphisms $\S\colon\A_n^Y\to \F_n\I\C/\F_{n+1}\I\C$ and $\ss\colon\A_n^c\to Y_n\I\C/Y_{n+1}$ by
\[
\S(J)=[\Sigma_{g,1}\times[-1,1],G]\in \F_n\I\C/\F_{n+1}\I\C
\text{ and }
\ss(J)=(\Sigma_{g,1}\times[-1,1])_G\in Y_n\I\C/Y_{n+1},
\]
respectively.
The surgery map $\ss\colon\A_n^c\to Y_n\I\C/Y_{n+1}$ is surjective when $n\ge2$ as shown in \cite[Section~8.5]{Hab00C} (see also \cite[Theorem~8.8]{CHM08}),
and induces an isomorphism $\A_n^c\otimes\Q\cong (Y_n\I\C/Y_{n+1})\otimes\Q$ whose inverse map up to sign is given by the degree $n$ part of $\Ztilde^Y$ as shown in \cite[Theorem~8.8]{CHM08}.
\subsection{The LMO functor}
\label{sec:LMO}
The LMO functor is an invariant of 3-dimensional cobordisms formulated as a functor from a certain category of cobordisms to a category of Jacobi diagrams.
Our homomorphisms $\overline{Z}_{n+1}$ and $\bar{z}_{n+1}$ mentioned in Section~\ref{section:Intro} are constructed from the mod $\Z$ reduction of the LMO functor.

We briefly review the definition of the LMO functor following \cite{CHM08}.
We denote by $\LCob_q$ the (non-strict) monoidal category of Lagrangian $q$-cobordisms and by $\tsA$ the strict monoidal category of top-substantial Jacobi diagrams.
An object of $\LCob_q$ is an element of the free magma generated by a letter $\bullet$, and a morphism $v \to w$ is a Lagrangian $q$-cobordism from $\Sigma_{|v|,1}$ to $\Sigma_{|w|,1}$, where $|v|$ denotes the length of $v$.
For instance, $\bullet((\bullet\bullet)\bullet)$ is an object, and ``$q$'' implies that we take into account parenthesizings.
``Lagrangian'' means a homological condition on cobordisms which is a 3-manifold with boundary explained in Section~\ref{sec:homology cylinder}, especially homology cylinders are Lagrangian.
The composition $\circ$ of $(M,m)$ and $(N,n)$ is defined by stacking $N$ on $M$,
namely, $(M,m)\circ (N,n)=(M\cup_{m_+=n-}N, m_-\cup n_+)$, where $m_+$ and $m_-$ denote identifications of the top surface and bottom surface in $M$ with reference surfaces as in Section~\ref{sec:homology cylinder}.

The tensor product $\otimes$ of $(M,m)$ and $(N,n)$ is given by horizontal juxtaposition of them.
An object of $\tsA$ is a non-negative integer and a morphism $g \to h$ is a series of top-substantial Jacobi diagrams labeled by $\{1^{+},\dots,g^{+},1^{-},\dots,h^{-}\}$.
The composition $\circ$ is the same as in Section~\ref{sec:jacobi}.
The tensor product of $x\colon g \to h$ and $x'\colon g' \to h'$ is given by the disjoint union of $x$ and $(x' / i^{+}\mapsto (i+g)^{+},\ j^{-}\mapsto (j+h)^{-})$, where $i=1,\dots,g'$ and $j=1,\dots,h'$.

The \emph{LMO functor} $\Ztilde$ is a functor $\LCob_q \to \tsA$ defined as follows.
For an object $v$, we set $\Ztilde(v) = |v|$.
Let $(M,m)$ be a morphism $v \to w$ and $(B,\gamma)$ the corresponding bottom-top $q$-tangle obtained as in Section~\ref{sec:bt_tangle}.
Choose a framed link $L$ in $[-1,1]^3$ such that $[-1,1]^3_L \cong B$.
Then the morphism $\Ztilde(M,m)$ which is a series of Jacobi diagrams is defined via the Kontsevich invariant of the tangle $\gamma \cup L \subset [-1,1]^3$.

For $M=(M,m)$, the value $\Ztilde(M) \in \widehat{\A}(\{1^{+},\dots,g^{+},1^{-},\dots,h^{-}\})$ is uniquely decomposed as $\exp(x)\sqcup\exp(y)$, where $x, y \in \widehat{\A}^c(\{1^{+},\dots,g^{+},1^{-},\dots,h^{-}\})$ are a linear combination of struts and a series of connected Jacobi diagrams without struts, respectively.
We define the \emph{$Y$-part} $\Ztilde^Y(M)$ of $\Ztilde(M)$ by $\Ztilde^Y(M) = \exp(y)$.

In this paper, we use the algebra homomorphism $\Ztilde^Y\colon \Q\I\C \to \widehat{\A}^Y\otimes \Q$ which is the linear extension of $\Ztilde^Y$.
Here $\I\C$ is regarded as a submonoid of $\LCob_q(w,w)$, where $w$ is the word $(\cdots((\bullet\bullet)\bullet)\cdots)$ of length $g$.

\begin{remark}
\label{rem:Associator}
The Kontsevich invariant of tangles depends on the choice of an associator, and we have to care about an associator and parenthesizings when we compute the LMO functor.
Throughout this paper, we fix an even rational Drinfel'd associator and consider $\Ztilde_{\leq2}(M)$ or the leading term and the next term of $\Ztilde^Y(M)$.
Then one needs not care about parenthesizings due to $\Ztilde_{\leq1}(P_{u,v,w}^{\pm1})=\Id_3$, and our invariants $\overline{Z}_{n+1}$ and $\bar{z}_{n+1}$ are independent of the choice of an even rational Drinfel'd associator (see \cite[Table~5.2 and Proposition~5.8]{CHM08}).
On the other hand, the LMO functor depends on the choice of curves in Figure~\ref{fig:Sigma_g1}, and so do $\overline{Z}_{n+1}$ and $\bar{z}_{n+1}$.
\end{remark}

\begin{table}[h]
\centering
$\Id_1=\raisegraph{-2em}{5em}{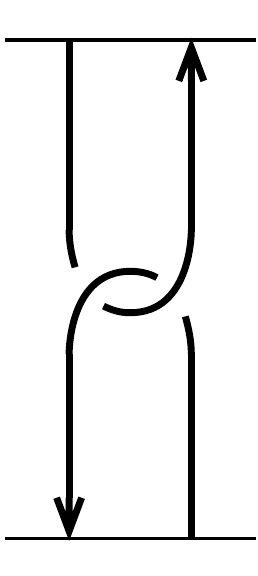}\quad
c=\raisegraph{-2em}{4.7em}{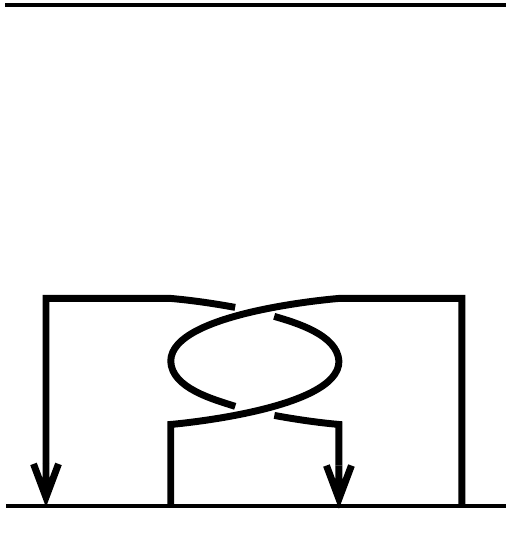}\quad
\Delta=\raisegraph{-2em}{5em}{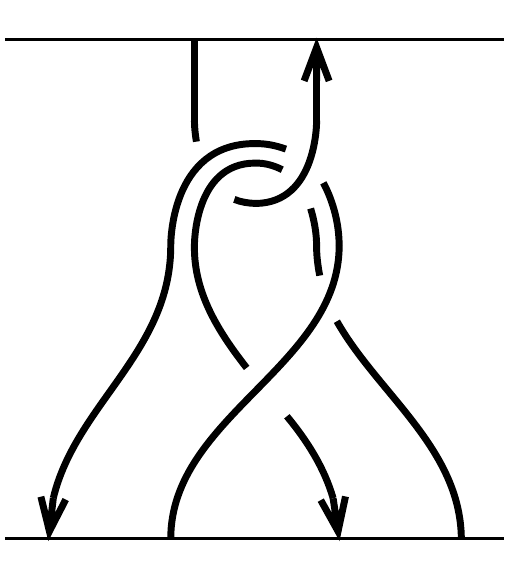}\quad
\psi=\raisegraph{-2em}{5em}{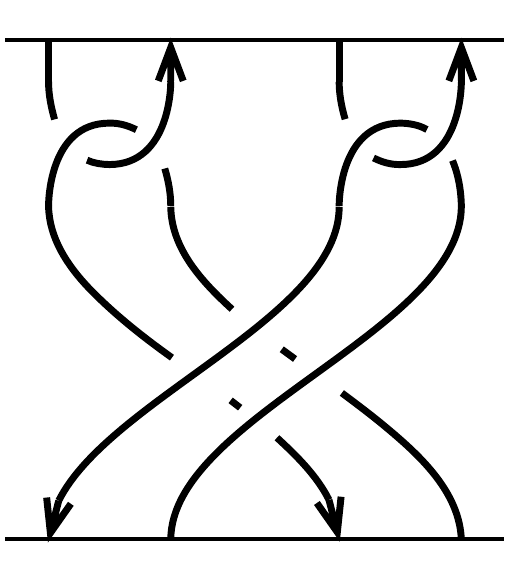}\quad
Y=\raisegraph{-2em}{5em}{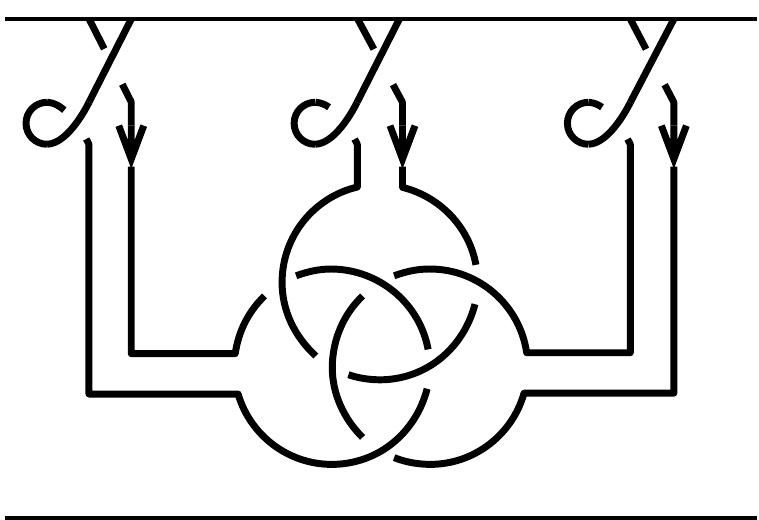}$
\vspace{1em}
\caption{Some bottom-top tangles.}
\label{tab:ElemCob}
\end{table}

For some tangles, we use notation in Table~\ref{tab:ElemCob} as in \cite{CHM08}.
The following values are used in Section~\ref{sec:Z_and_z} (see \cite[Table~5.2]{CHM08}):

\begin{align*}
 \log \Ztilde(\Id_g)
 &= \log\Id_g = \sum_{i=1}^g \strutgraph{i^+}{i^-}\ , \\
 (\log \Ztilde(c))_{\leq 2}
 &= -\dstrutgraph{1^{-}}{2^{-}}\ ,\\
 (\log \Ztilde(\Delta))_{\leq 1}
 &= \strutgraph{1^+}{1^-} +\strutgraph{1^+}{1^-} + \frac{1}{2}\dYgraph{1^+}{1^-}{2^-}\ ,\\
 (\log \Ztilde(\psi^{\pm 1}))_{\leq 1} 
 &= \strutgraph{1^{+}}{2^{-}} + \strutgraph{2^{+}}{1^{-}}\ ,\\
 (\log \Ztilde(Y))_{\leq 2} 
 &= -\uYgraph{1^+}{2^+}{3^+} +\frac{1}{2}\uHgraph{1^+}{1^+}{2^+}{3^+} +\frac{1}{2}\uHgraph{2^+}{2^+}{3^+}{1^+} +\frac{1}{2}\uHgraph{3^+}{3^+}{1^+}{2^+}\ .
\end{align*}

\section{The operation $\delta$ on Jacobi diagrams}
\label{sec:delta}
As we explained in Section~\ref{section:Intro},
the operation $\delta\colon\A_n^Y\to\A_{n+1}^Y\otimes\Z/2\Z$ describes our homomorphism
$\overline{Z}_{n+1}\colon \F_n\I\C/\F_{n+1}\I\C\to \A_{n+1}^Y\otimes \Q/\Z$.
In this section,
we define $\delta\colon\A_n^Y\to\A_{n+1}^Y\otimes\Z/2\Z$ (Section~\ref{section:definedelta}).
We also investigate $\Ker\delta$ when $n$ is odd (Section~\ref{section:kernel-of-delta}) which gives some estimate on the kernel of the surgery map $\ss\colon\A_{n,1}\to Y_n\I\C/Y_{n+1}$.

\subsection{Definition of $\delta\colon\A_n^Y\to\A_{n+1}^Y\otimes\Z/2\Z$}\label{section:definedelta}
Let $J$ be a Jacobi diagram colored with $\{1^{\pm},2^{\pm},\ldots,g^{\pm}\}$ without struts,
and denote by $U(J)$ the set of univalent vertices.
For $v\in U(J)$,
we denote by $\delta_v(J)$ the sum of two Jacobi diagrams obtained by changing the neighborhood of the edge incident to $v$ as follows:
\[
\delta_v\colon\
\begin{tikzpicture}[baseline=3ex, scale=0.25, dash pattern={on 2pt off 1pt}]
\node [left] at (0,-0.1) {$\cdots$};
\setlength{\unitlength}{0.8pt}
\node [right] at (6,-0.1) {$\cdots$};
\node [above] at (3.2,2.8) {$\ell(v)$};
\draw (0,0) -- (6,0);
\draw (3,0) -- (3,3);
\end{tikzpicture}
\mapsto \
\begin{tikzpicture}[baseline=3ex, scale=0.25, dash pattern={on 2pt off 1pt}]
\node [left] at (0,-0.1) {$\cdots$};
\node [right] at (6,-0.1) {$\cdots$};
\node [above] at (1.5,2.8) {$\ell(v)$};
\node [above] at (4.5,2.8) {$\ell(v)$};
\draw (0,0) -- (6,0);
\draw (2,0) -- (2,3);
\draw (4,0) -- (4,3);
\end{tikzpicture}
+\
\begin{tikzpicture}[baseline=3ex, scale=0.25, dash pattern={on 2pt off 1pt}]
\node [left] at (0,-0.1) {$\cdots$};
\node [right] at (6,-0.3) {$\cdots$,};
\node [above] at (1,2.8) {$\ell(v)$};
\node [above] at (5,2.8) {$\ell(v)^*$};
\draw (0,0) -- (6,0);
\draw (3,0) -- (3,1.5);
\draw (1.5,3) -- (3,1.5);
\draw (4.5,3) -- (3,1.5);
\end{tikzpicture}
\]
where we denote the label of $v$ by $\ell(v)\in\{1^{\pm},\ldots,g^{\pm}\}$ and $(j^\pm)^*=j^{\mp}$ for $1\le j\le g$.
For unordered pairs $\{v,w\}$ in $U(J)$ such that $v\ne w$ and $\ell(v)=\ell(w)$, 
we also denote by $\delta_{vw}(J)$ the Jacobi diagram obtained by changing the neighborhoods of the edges incident to $v$ and $w$ as:
\[
\delta_{vw}\colon
\begin{tikzpicture}[baseline=2.5ex, scale=0.2, dash pattern={on 2pt off 1pt}]
\draw (1.1,0) -- (7.1,0);
\draw (4.1,0) -- (4.1,3);
\node at (-0.3,-0.1){$\cdots$};
\node at (8.8,-0.1){$\cdots$};
\node at (4.3, 4.2){$\ell(v)$};
\draw (15.1,0) -- (21.1,0);
\draw (18.1,0) -- (18.1,3);
\node at (13.8,-0.1){$\cdots$};
\node at (22.8,-0.1){$\cdots$};
\node at (18.3, 4.2){$\ell(w)$};
\end{tikzpicture}
\mapsto\
\begin{tikzpicture}[baseline=2.5ex, scale=0.2, dash pattern={on 2pt off 1pt}]
\draw (1.1,0) -- (7.1,0);
\node at (-0.3,-0.1){$\cdots$};
\node at (8.8,-0.1){$\cdots$};
\draw (15.1,0) -- (21.1,0);;
\node at (13.8,-0.1){$\cdots$};
\node at (22.8,-0.1){$\cdots$};
\draw (18.1,0) arc [start angle = 0, end angle = 180, x radius = 7, y radius = 2.2];
\draw (11.1,2.2) -- (11.1,3.8);
\node at (11.1,4.9){$\ell(v)$};
\node [right] at (24.8,0){,};
\end{tikzpicture}
\vspace*{2ex}
\]
where $v$ and $w$ possibly lie in different connected components of $J$.
\begin{definition}
Denoting
\[
\delta' J=\sum_{v\in U(J)}\delta_v J\qquad\text{and}\qquad
\delta'' J=\sum_{\begin{subarray}{c}\{v,w\}\subset U(J)\\v\ne w\\ \ell(v)=\ell(w)\end{subarray}}\delta_{vw}J,
\]
we define a homomorphism $\delta\colon\A_n^Y\to\A_{n+1}^Y\otimes\Z/2\Z$ by
$\delta(J)=\delta'(J)+\delta''(J)$.
\end{definition}

\begin{lemma}
The map $\delta\colon\A_n^Y\to \A_{n+1}^Y\otimes \Z/2\Z$ is well-defined.
\end{lemma}
\begin{proof}
We need to check that two Jacobi diagrams which are equivalent modulo the AS and IHX relations map to the same Jacobi diagram under $\delta$.
Let $v$ be a univalent vertex in a Jacobi diagram $J$ of $\ideg=n$.
For the case of the AS relation, we have equalities
\[
\delta_v\left(\begin{tikzpicture}[baseline=0ex, scale=0.25, dash pattern={on 2pt off 1pt}]
\tiny
\draw (0,0) -- (0,2);
\draw (0:0) to [out=180, in=160] (235:1.1)
to [out=-20, in=140] (-50:3);
\draw (0:0) to [xscale=-1, out=180, in=160] (235:1.1)
to [xscale=-1, out=-20, in=140] (260:1.1);
\draw (260:1.3) to [xscale=-1, out=-20, in=140] (-50:3);
\node [above] at (0,2){$\ell(v)$};
\end{tikzpicture}\right)
=
\begin{tikzpicture}[baseline=0ex, scale=0.25, dash pattern={on 2pt off 1pt}]
\tiny
\draw (0.4,-0.2) -- (0.4,2);
\draw (-0.4,-0.2) -- (-0.4,2);
\draw (0:0) to [out=180, in=160] (235:1.1)
to [out=-20, in=140] (-50:3);
\draw (0:0) to [xscale=-1, out=180, in=160] (235:1.1)
to [xscale=-1, out=-20, in=140] (260:1.1);
\draw (260:1.3) to [xscale=-1, out=-20, in=140] (-50:3);
\node [above] at (-1,2){$\ell(v)$};
\node [above] at (1,2){$\ell(v)$};
\end{tikzpicture}
+
\begin{tikzpicture}[baseline=0ex, scale=0.25, dash pattern={on 2pt off 1pt}]
\tiny
\draw (0,0) -- (0,1);
\draw (0,1) --(-0.8,2);
\draw (0,1) --(0.8,2);
\draw (0:0) to [out=180, in=160] (235:1.1)
to [out=-20, in=140] (-50:3);
\draw (0:0) to [xscale=-1, out=180, in=160] (235:1.1)
to [xscale=-1, out=-20, in=140] (260:1.1);
\draw (260:1.3) to [xscale=-1, out=-20, in=140] (-50:3);
\node [above] at (1.2,2){$\ell(v)^*$};
\node [above] at (-1.2,2){$\ell(v)$};
\end{tikzpicture}
=\begin{tikzpicture}[baseline=0ex, scale=0.26, dash pattern={on 2pt off 1pt}]
\tiny
\draw (0.4,-0.1) -- (0.4,2);
\draw (-0.4,-0.1) -- (-0.4,2);
\draw (0:0) to [out=0, in=135] (-50:2.5);
\draw (0:0) to [xscale=-1,out=0, in=135] (-50:2.5);
\node [above] at (-1,2){$\ell(v)$};
\node [above] at (1,2){$\ell(v)$};
\end{tikzpicture}
+\begin{tikzpicture}[baseline=0ex, scale=0.26, dash pattern={on 2pt off 1pt}]
\tiny
\draw (0,0) -- (0,1);
\draw (0,1) --(-0.8,2);
\draw (0,1) --(0.8,2);
\draw (0:0) to [out=0, in=135] (-50:2.5);
\draw (0:0) to [xscale=-1,out=0, in=135] (-50:2.5);
\node [above] at (1.2,2){$\ell(v)^*$};
\node [above] at (-1.2,2){$\ell(v)$};
\end{tikzpicture}
=\delta_v\left(
\begin{tikzpicture}[baseline=0.2ex, scale=0.26, dash pattern={on 2pt off 1pt}]
\tiny
\draw (0,0) -- (0,2);
\draw (0:0) to [out=0, in=135] (-50:2.5);
\draw (0:0) to [xscale=-1,out=0, in=135] (-50:2.5);
\node [above] at (0,2){$\ell(v)$};
\end{tikzpicture}\right)\in A_{n+1}^{Y}\otimes\Z/2\Z.
\]
For the case of the IHX relation,
we also have
\begin{align*}
&\delta_v\left(\,
\begin{tikzpicture}[baseline=2.6ex, scale=0.5, dash pattern={on 2pt off 1pt}]
\tiny
\draw (0,0) to [out=45, in=135] (2,0); 
\draw (1,0.5) -- (1,1.5); 
\draw (0.3,1.8) to [out=-45, in=225] (2,2); 
\node at (0.2,2.1){$\ell(v)$};
\end{tikzpicture}\,-\,
\begin{tikzpicture}[baseline=2.6ex, scale=0.5, dash pattern={on 2pt off 1pt}]
\tiny
\draw (0,0)  to [out=45, in=-45] (0.3,1.8); 
\draw (0.5,1) -- (1.5,1); 
\draw (2,0)  to [out=135, in=-135] (2,2); 
\node at (0.2,2.1){$\ell(v)$};
\end{tikzpicture}\,+\,
\begin{tikzpicture}[baseline=2.6ex, scale=0.5, dash pattern={on 2pt off 1pt}]
\tiny
\draw (0,0) to (2,2); 
\draw (0.3,1.8) to (0.9,1.1);
\draw (1.1,0.9) to (2,0);
\draw (0.4,0.4) to (1.6,0.4); 
\node at (0.2,2.1){$\ell(v)$};
\end{tikzpicture}\,\right)\\
&=
\left(
\begin{tikzpicture}[baseline=2.6ex, scale=0.5, dash pattern={on 2pt off 1pt}]
\tiny
\draw (0,0) to [out=45, in=135] (2,0); 
\draw (1,0.5) -- (1,1.5); 
\draw (0.4,1.8) to [out=-45, in=180] (1,1.5);
\draw (0,1.6) to [out=-45, in=180] (1,1.1);
\draw (1,1.5) to [out=0, in=225] (2,2); 
\node at (0.3,2.1){$\ell(v)$};
\node at (-0.5,1.6){$\ell(v)$};
\end{tikzpicture}\,-\,
\begin{tikzpicture}[baseline=2.6ex, scale=0.5, dash pattern={on 2pt off 1pt}]
\tiny
\draw (0,0)  to [out=45, in=270] (0.5,1);
\draw (0.5,1) to [out=90, in=-45] (0,1.7);
\draw (0.9,1) to [out=90, in=-45] (0.4,1.8);
\draw (0.5,1) -- (1.5,1); 
\draw (1.8,0)  to [out=135, in=-135] (2,2);
\node at (0.4,2.1){$\ell(v)$};
\node at (-0.5,1.7){$\ell(v)$};
\end{tikzpicture}\,+\,
\begin{tikzpicture}[baseline=2.6ex, scale=0.5, dash pattern={on 2pt off 1pt}]
\tiny
\draw (0,0) to (2,2); 
\draw (0.3,1.8) to (0.9,1.1);
\draw (1.1,0.9) to (2,0);
\draw (0,1.6) to (0.7,0.9);
\draw (0.9,0.7) to (1.2,0.4);
\draw (0.4,0.4) to (1.6,0.4); 
\node at (0.4,2.1){$\ell(v)$};
\node at (-0.5,1.7){$\ell(v)$};
\end{tikzpicture}\,\right)
+\left(\,
\begin{tikzpicture}[baseline=3ex, scale=0.5, dash pattern={on 2pt off 1pt}]
\tiny
\draw (0,0) to [out=45, in=135] (2,0); 
\draw (1,0.5) -- (1,1.5); 
\draw (0.4,1.7) to [out=-45, in=225] (2,2); 
\draw (0.4,1.7) -- (0,1.7);
\draw (0.4,1.7) -- (0.4,2.1);
\node at (-0.5,1.7){$\ell(v)$};
\node at (0.5,2.3){$\ell(v)^*$};
\end{tikzpicture}\,-\,
\begin{tikzpicture}[baseline=3ex, scale=0.5, dash pattern={on 2pt off 1pt}]
\tiny
\draw (0,0)  to [out=45, in=-45] (0.4,1.7); 
\draw (0.6,1) -- (1.5,1); 
\draw (2,0)  to [out=135, in=-135] (2,2); 
\draw (0.4,1.7) -- (0,1.7);
\draw (0.4,1.7) -- (0.4,2.1);
\node at (-0.5,1.7){$\ell(v)$};
\node at (0.5,2.3){$\ell(v)^*$};
\end{tikzpicture}\,+\,
\begin{tikzpicture}[baseline=3ex, scale=0.5, dash pattern={on 2pt off 1pt}]
\tiny
\draw (0,0) to (2,2); 
\draw (0.4,1.7) to (0.9,1.1);
\draw (1.1,0.9) to (2,0);
\draw (0.4,0.4) to (1.6,0.4); 
\draw (0.4,1.7) to (0,1.7);
\draw (0.4,1.7) to (0.4,2.1);
\node at (-0.5,1.7){$\ell(v)$};
\node at (0.5,2.3){$\ell(v)^*$};
\end{tikzpicture}\,\right).
\end{align*}
Here, both three terms in the right-hand side equal to $0$ in $\A_{n+1}^Y\otimes\Z/2\Z$ by the AS and IHX relations.
We see that the map $\delta_{vw}$ also preserves the AS and IHX relations in a similar way for univalent vertices $v$ and $w$ with identical labels.
\end{proof}

The map $\delta$ satisfies the Leibniz rule as follows.
\begin{prop}\label{prop:Leibniz}
Let $k$, $l$ be positive integers.
For $x\in \A_k^Y$ and $y\in \A_l^Y$,
we have
\[
\delta(x\star y)=(\delta x)\star y+x\star (\delta y)\in \A_{k+l+1}^Y\otimes\Z/2\Z. 
\]
\end{prop}
Since we do not use Proposition~\ref{prop:Leibniz} in the rest of the paper,
we give a proof in Appendix~\ref{section:leibniz}.

\subsection{The kernel of $\delta$}\label{section:kernel-of-delta}
Recall that $\A_{n,r}^{c}$ is the submodule of $\A_n^c$ generated by connected Jacobi diagrams with $r$ loops. 
Here, we introduce a map $\Delta_{n,r}\colon \A_{n,r}^c\to A_{2n+1,2r}^c$ whose image is in the kernel of our homomorphism $\bar{z}_{n+1}$.
We use $\Delta_{n,r}$ to investigate the kernel of $\ss$ in Section~\ref{section:kernel-tree}.
\begin{definition}
For $n\ge0$ and $r\ge0$,
we define a map $\Delta_{n,r}\colon \A_{n,r}^c\to A_{2n+1,2r}^c$ by
\[
\Delta_{n,r}(J)=\sum_{v\in U(J)}
\begin{tikzpicture}[baseline=-0.4ex, scale=0.26, dash pattern={on 2pt off 1pt}]
\node [left] at (0,0) {$\ell (v)$};
\draw (0,0) -- (2,0); 
\draw (2,0) -- (4.5,1); 
\draw (2,0) -- (4.5,-1); 
\node [right] at (4.5,1) {$J_v$};
\node [right] at (4.5,-1.1) {$J_v$,};
\end{tikzpicture}
\]
where
$J_v$ denotes the complement of a neighborhood of a vertex $v$ in $J$, namely,
$J=
\begin{tikzpicture}[baseline=-0.6ex, scale=0.26, dash pattern={on 2pt off 1pt}]
\node [left] at (0,0) {$\ell (v)$};
\draw (0,0) -- (1.5,0); 
\node [right] at (1.5,0) {$J_v$.};
\end{tikzpicture}$
\end{definition}
In \cite{CST12W},
Conant, Schneiderman, and Teichner defined $\Delta_{n,0}$ to describe the kernel of the composite map $\Gr q\circ\ss\colon \A_n^c\to Y_n\I\H/Y_{n+1}$, where $\Gr q\colon Y_n\I\C/Y_{n+1}\to Y_n\I\H/Y_{n+1}$ is the induced map by the projection $q\colon\I\C\to\I\H$.
They used $\Delta_{n,0}$ to determine the kernel of the realization map from $\A_{n,0}^c$ to the graded quotient of the Whitney tower filtration on the set of framed links.
The map $\Delta_{n,r}$ is a generalization of their map.

\begin{prop}\label{prop:Delta}
For $n\ge0$ and $r\ge0$,
\[
\delta'\circ \Delta_{n,r}=\delta''\circ\Delta_{n,r}=0\colon\A_{n,r}^c\to \A_{2n+2}^c\otimes\Z/2\Z.
\]
\end{prop}

\begin{proof}
Let $J$ be a connected Jacobi diagram of $\ideg(J)=n$ with $r$ loops.
For $v\in U(J)$,
we denote by $\Delta_v(J)$ the Jacobi diagram obtained by taking the double of the complement of a neighborhood of $v$ as explained in the definition of $\Delta_{n,r}$.
We have 
\[
\Delta_{n,r}(J)=\sum_{v\in U(J)}\Delta_v(J).
\]

First, we show that $(\delta''\circ \Delta_{n,r})(J)$ is zero.
For $v, w\in U(J)$ such that $v\ne w$,
we denote by $w_1,w_2\in U(\Delta_v(J))$ the vertices which arise in $\Delta_v(J)$ as two copies of $w\in U(J)$. 
Since the Jacobi diagram $\Delta_v(J)$ is symmetric with respect to the edge incident to $v$,
$\delta''(\Delta_v(J))$ is equal to
\[
\sum_{\begin{subarray}{c}w\in U(J)\\w\ne v\end{subarray}}\delta_{w_1w_2}(\Delta_v(J))\in \A_{2n+2,2r+1}^c\otimes\Z/2\Z.
\]
For mutually distinct $v,w\in U(J)$,
denote the Jacobi diagram as
$J=
\begin{tikzpicture}[baseline=-0.4ex, scale=0.26, dash pattern={on 2pt off 1pt}]
\node [left] at (0,0) {$\ell (v)$};
\draw (0,0.2) -- (2,0.2); 
\node [right] at (2,0) {$J_{vw}$};
\draw (5.4,0.2) -- (7.4,0.2); 
\node [right] at (7.4,0) {$\ell (w)$,};
\end{tikzpicture}$
where $J_{vw}$ denotes the complement of neighborhoods of  $v,w\in U(J)$.
Since we see that
\[
\delta_{w_1w_2}(\Delta_v(J))=\,
\begin{tikzpicture}[baseline=-0.6ex, scale=0.26, dash pattern={on 2pt off 1pt}]
\node [left] at (0,-0.1) {$\!\!\ell (v)$};
\draw (0,0) -- (2,0); 
\draw (2,0) -- (4.5,1.2); 
\draw (2,0) -- (4.5,-1.2); 
\node [right] at (4.5,1.3) {$J_{vw}$};
\node [right] at (4.5,-1.3) {$J_{vw}$};
\draw (10.5,0) -- (8,1.2); 
\draw (10.5,0) -- (8,-1.2); 
\draw (10.5,0) -- (12.3,0); 
\node [right] at (12.5,-0.1) {$\ell (w)\!\!$};
\end{tikzpicture}
=\delta_{v_1v_2}(\Delta_w(J)),
\]
every Jacobi diagram in the sum
\[
(\delta''\circ \Delta_{n,r})(J)
=\sum_{v\in U(J)}\delta''(\Delta_v(J))
=\sum_{v\in U(J)}\sum_{w\in U(J)\setminus\{v\}}\delta_{w_1w_2}(\Delta_v(J))\in\A_{2n+2,2r+1}^c\otimes\Z/2\Z
\]
appears twice, and the sum is equal to zero.

Next, we consider $(\delta'\circ \Delta_{n,r})(J)$.
For $v\in U(J)$, $\delta'(\Delta_v(J))$ is equal to $\delta_v(\Delta_v(J))$ by the symmetry of $\Delta_v(J)$.
It is described as
\[
\delta_v(\Delta_v(J))=\,
\begin{tikzpicture}[baseline=3ex, scale=0.26, dash pattern={on 2pt off 1pt}]
\node [left] at (0,0) {$J_v$};
\draw (0,0) -- (9,0); 
\node [right] at (9,0) {$J_v$};
\draw (3,0) -- (3,3);
\node [above] at (3,3) {$\ell (v)$};
\draw (6,0) -- (6,3);
\node [above] at (6,3) {$\ell (v)$};
\end{tikzpicture}
+\,
\begin{tikzpicture}[baseline=3ex, scale=0.26, dash pattern={on 2pt off 1pt}]
\node [left] at (0,0) {$J_v$};
\draw (0,0) -- (9,0); 
\node [right] at (9,0) {$J_v$};
\draw (4.5,0) -- (4.5,1.5);
\draw (4.5,1.5) -- (3,3);
\draw (4.5,1.5) -- (6,3);
\node [above] at (2.8,3) {$\ell (v)$};
\node [above] at (6.2,3) {$\ell (v)^*$};
\end{tikzpicture}
\in\A_{2n+2,2r}^c\otimes\Z/2\Z,
\]
and the second term is zero by the IHX relation and its symmetry.
In Lemma~\ref{lem:jacobi} below,
we see that the sum of the first terms in $\delta_v(\Delta_v(J))$ is also zero, where $v$ runs through $U(J)$.
\end{proof}

Let $J$ be a connected Jacobi diagram of $\ideg=n$.
For $v,w\in U(J)$,
let us take two copies of $J$,
and connect the middle point of the edge incident to $v$ in the first copy to that of the edge incident to $w$ in the second copy by a new edge.
Denote the resulting Jacobi diagram by $D_{vw}\in \A_{2n+2}^c\otimes\Z/2\Z$.
Especially, we see that $D_{vv}=\delta_v(\Delta_v(J))$.
\begin{lemma}\label{lem:jacobi}
For a connected Jacobi diagram $J$ of $\ideg=n$,
\[
\sum_{v\in U(J)}D_{vv}=0\in \A_{2n+2}^c\otimes\Z/2\Z.
\]
\end{lemma}
\begin{proof}
Let $v\in U(J)$.
By the AS and IHX relations,
or so-called the Kirchhoff law, we have
\[
\sum_{w\in U(J)}D_{vw}=0
\]
by moving round the attaching point of the new edge in the second copy of $J$.
See, for example, \cite[Remark~5.13]{CDM12}.
Thus, we obtain
\[
0=\sum_{v\in U(J)}\sum_{w\in U(J)}D_{vw}=\sum_{v\in U(J)}\sum_{w\in U(J)\setminus\{v\}}D_{vw}+\sum_{v\in U(J)}D_{vv}.
\]
Since $D_{vw}=D_{wv}$ by the definition,
all the Jacobi diagrams in the sum $\sum_{v\in U(J)}\sum_{w\in U(J)\setminus\{v\}}D_{vw}$ appear twice, 
and the sum is equal to zero.
Thus, we obtain $\sum_{v\in U(J)}D_{vv}=0$.
\end{proof}

\section{The homomorphisms $\overline{Z}_{n+1}$ and $\bar{z}_{n+1}$}
\label{sec:Z_and_z}
In this section, we introduce homomorphisms $\overline{Z}_{n+1}$ and $\bar{z}_{n+1}$ mentioned in Section~\ref{section:Intro},
and prove Theorem~\ref{thm:main} which enables us to compute them.

Let $\iota_n\colon \A_n^Y\to \A_n^Y\otimes \Q$ be the natural map defined by $\iota_n(J)=J\otimes 1$.
We sometimes denote $\iota=\iota_n$.
Recall that $\Ztilde^Y\colon \Z\I\C \to \widehat{\A}^Y\otimes \Q$ is an algebra homomorphism which maps $\F_n\I\C$ to $\widehat{\A}_{\ge n}^Y\otimes \Q$.
Moreover, it satisfies $\Ztilde^Y_n(\F_n\I\C) \subset \Im\iota_n$ by \cite[Proof of Theorem~7.11]{CHM08}.

\begin{definition}
For $n \geq 1$, define a homomorphism
\[
\overline{Z}_{n+1}\colon \F_{n}\I\C/\F_{n+1}\I\C \to \A_{n+1}^Y\otimes \Q/\Z
\]
by $\overline{Z}_{n+1}([x]) = \Ztilde^Y_{n+1}(x) \bmod \Z$.
Similarly, a homomorphism 
\[
\bar{z}_{n+1}\colon Y_{n}\I\C/Y_{n+1} \to \A_{n+1}^c\otimes \Q/\Z
\]
is defined by $\bar{z}_{n+1}([M]) = (\log\Ztilde^Y(M))_{n+1} \bmod \Z$.
In other words,
$$
\bar{z}_{n+1}(M)=
\begin{cases}
\Ztilde_{n+1}^Y(M) \bmod \Z & \text{if $n\ge 2$},\\
\displaystyle \Ztilde_2^Y(M)-\frac{1}{2}\Ztilde_1^Y(M)\sqcup \Ztilde_1^Y(M) \bmod \Z & \text{if $n=1$}.
\end{cases}
$$
\end{definition}

We should check that $\overline{Z}_{n+1}$ and $\bar{z}_{n+1}$ are well-defined.
First, the inclusion $\Ztilde_{n+1}^Y(\F_{n+1}\I\C) \subset \Im\iota_{n+1}$ implies $\overline{Z}_{n+1}$ is well-defined.
It suffices for the well-definedness of $\bar{z}_{n+1}$ to prove that $(\log\Ztilde^Y(M))_{n+1} \equiv (\log\Ztilde^Y(M_{G}))_{n+1} \mod \Z$ for $M \in Y_{n}\I\C$ and a connected graph clasper $G$ of degree $n+1$.
Since $M-M_{G} = [M,G] \in \F_{n+1}\I\C$, we have $\Ztilde^Y_1(M) - \Ztilde^Y_1(M_G) = 0\in \A_1^Y$ and the inclusion $\Ztilde_{n+1}^Y(\F_{n+1}\I\C) \subset \Im\iota_{n+1}$ implies $\Ztilde^Y_{n+1}(M) - \Ztilde^Y_{n+1}(M_G) \in \Im\iota_{n+1}$.
Thus, we conclude that $(\log\Ztilde^Y(M))_{n+1} \equiv (\log\Ztilde^Y(M_{G}))_{n+1} \mod \Z$.

\begin{remark}\label{rem:BC-homo}
In \cite[Theorem~1.3]{MaMe03}, 
an isomorphism between modules
$\I\C/Y_2$ and $\Lambda^3H\times_{\Lambda^3H\otimes\Z/2\Z}B_3$
is given,
where $B_3$ is the space of cubic boolean functions on the set $\Spin(\Sigma_{g,1})$ of spin structures.
In \cite{MaMe13}, $Y_3$-equivalence classes of homology cylinders and values of $\Ztilde_2^Y$ on $\K\C$ are investigated.
As corollaries of these results, we see that the homomorphism
\[
\Ztilde_1^Y\oplus \bar{z}_2\colon\I\C/Y_2\to (\A_1^c\otimes\Q)\oplus (\A_2^c\otimes\Q/\Z)
\]
is injective.

Moreover, we will see in Remark~\ref{rem:BC-homo2} that we can describe the mod $2$ reduction of $\tau_1$ in terms of $\bar{z}_2$.
By combining with \cite{MaMe13},
$\bar{z}_2$ has all information of the Birman-Craggs homomorphisms.
\end{remark}

In the same way as $\bar{z}_{n+1}$, we have the following.

\begin{prop}
For $n\ge1$,
the map
\[
(\log\Ztilde^Y)_{n+2}\bmod\frac{1}{2}\Z \colon Y_n\I\C\to\A_{n+2}^c\otimes\left(\Q / \tfrac{1}{2}\Z\right)
\]
is a homomorphism,
and factors through $Y_n\I\C/Y_{n+1}$.
\end{prop}
\begin{proof}
For $M,N\in Y_n\I\C$, Theorem~\ref{thm:main} and the surjectivity of $\ss$ imply
\[
\Ztilde^Y_{n+2}(M\circ N)-\Ztilde^Y_{n+2}(M)-\Ztilde^Y_{n+2}(N)
\equiv\sum_{i=1}^{n+1}\Ztilde^Y_{n+2-i}(M)\star\Ztilde^Y_i(N)
\equiv 0\mod\frac{1}{2}\Z.
\]
By taking the connected part,
we see that $(\log\Ztilde^Y)_{n+2}\bmod\frac{1}{2}\Z$ is a homomorphism.

Next, let $M,N\in Y_n\I\C$ satisfy $M\sim_{Y_{n+1}}N$,
and denote by $N'\in Y_n\I\C$ a representative of the inverse element of $N$ in $Y_n\I\C/Y_{n+1}$.
We have $M\circ N'\in Y_{n+1}\I\C$, and
\[
0\equiv\Ztilde^Y_{n+2}(M\circ N')
\equiv\Ztilde^Y_{n+2}(M)+\Ztilde^Y_{n+2}(N')
\equiv\Ztilde^Y_{n+2}(M)-\Ztilde^Y_{n+2}(N) \mod\frac{1}{2}\Z.
\]
Taking the connected part,
we obtain $(\log\Ztilde^Y)_{n+2}(M)\equiv(\log\Ztilde^Y)_{n+2}(N)\mod\frac{1}{2}\Z$.
\end{proof}

\begin{theorem}\label{thm:Z-homo}
For $n\ge1$,
the homomorphism $\bar{z}_{n+1}\colon Y_n\I\C/Y_{n+1}\to \A_{n+1}^c\otimes\Q/\Z$
extends to a homomorphism on $Y_{\lfloor \frac{n}{2}\rfloor+1}\I\C/Y_{n+1}$.
\end{theorem}
\begin{proof}
For $k\ge1$ and $M\in Y_k\I\C$,
we have $\Ztilde^Y(M)-\emptyset\in \widehat{\A}_{\ge k}^Y\otimes\Q$.
Thus, for $M,N\in Y_k\I\C$,
\begin{align*}
\Ztilde^Y(M\circ N)
&=\Ztilde^Y(M)\star\Ztilde^Y(N)\\
&=\{\emptyset+(\Ztilde^Y(M)-\emptyset)\}
\star\{\emptyset+(\Ztilde^Y(N)-\emptyset)\}\\
&\equiv \emptyset+(\Ztilde^Y(M)-\emptyset)+(\Ztilde^Y(N)-\emptyset)
+\Ztilde_k^Y(M)\star \Ztilde_k^Y(N) \mod \widehat{\A}_{\ge 2k+1}^Y\otimes \Q.
\end{align*}
Since the leading term of $\Ztilde^Y$ is integral,
$\Ztilde_k^Y(M)$, $\Ztilde_k^Y(N)\in \Im\iota_k$.
Thus, we have 
\begin{align}
\Ztilde_{2k-1}^Y(M\circ N)&=\Ztilde_{2k-1}^Y(M)+\Ztilde_{2k-1}^Y(N)\in \A_{2k-1}^Y\otimes \Q\label{eq:ztilde1},\\
\Ztilde_{2k}^Y(M\circ N)&\equiv\Ztilde_{2k}^Y(M)+\Ztilde_{2k}^Y(N)\mod\Z.\label{eq:ztilde2}
\end{align}
When $n$ is even, 
by setting $k=\frac{n+2}{2}$ in (\ref{eq:ztilde1}),
we  see that $\Ztilde^Y_{n+1} \bmod \Z$ is a homomorphism on $Y_{\frac{n+2}{2}}\I\C/Y_{n+1}$.
When $n$ is odd,
by setting $k=\frac{n+1}{2}$ in (\ref{eq:ztilde2}),
we also see that $\Ztilde^Y_{n+1} \bmod \Z$ is a homomorphism on $Y_{\frac{n+1}{2}}\I\C/Y_{n+1}$.
\end{proof}

The proof of Theorem~\ref{thm:main} is based on a decomposition of a bottom-top tangle into some elementary pieces (see Figure~\ref{fig:Decomposition}), which is a refinement of the decomposition used in \cite[Theorem~7.11]{CHM08}.
Here the bottom-top $q$-tangles $\Delta_t$ and $\Delta_b$ are defined by
$$
\Delta_t = \raisegraph{-2em}{5em}{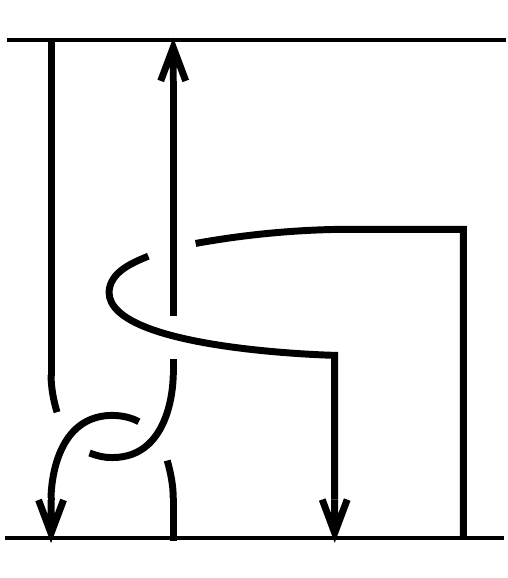} \text{\; and \;} \Delta_b = \raisegraph{-2em}{5em}{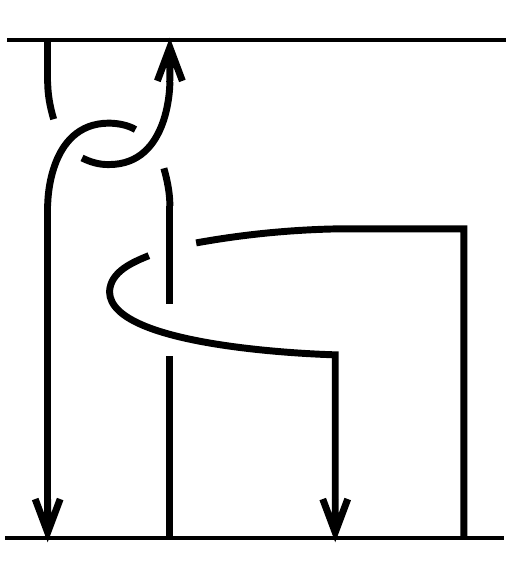},
$$
respectively.
We set $\Delta_t^0=\Id_\bullet$, and for $m \in \Z_{>0}$, let $\Delta_t^m$ denote the composite $(\Delta_t\otimes\Id_{w_{m-1}})\circ\dots\circ(\Delta_t\otimes\Id_{w_1})\circ\Delta_t$, where $w_{j}$ is the word $(\cdots((\bullet\bullet)\bullet)\cdots)$ of length $j$.
We define $\Delta_b^m$ in the same way.

\begin{lemma}
\label{lem:Delta}
 The following equalities hold
\begin{align*}
 (\log\Ztilde(\Delta_t^m))_{\leq 1} 
 &= \strutgraph{1^{+}}{1^{-}} + \sum_{j=2}^{m+1} \strutgraph{1^{+}}{j^{-}} - \sum_{j=2}^{m+1} \frac{1}{2}\dYgraph{1^{+}}{1^{-}}{j^{-}} - \sum_{2\leq j<k\leq m+1} \frac{1}{2}\dYgraph{1^{+}}{j^{-}}{k^{-}}\ , \\
 (\log\Ztilde(\Delta_b^m))_{\leq 1} 
 &= \strutgraph{1^{+}}{1^{-}} + \sum_{j=2}^{m+1} \dstrutgraph{1^{-}}{j^{-}} - \sum_{j=2}^{m+1} \frac{1}{2}\dYgraph{1^{+}}{1^{-}}{j^{-}} - \sum_{2\leq j<k\leq m+1} \frac{1}{2}\ddYgraph{1^{-}}{j^{-}}{k^{-}}\ .
\end{align*}
\end{lemma}

\begin{proof}
 Once we show the case $m=1$, by the induction on $m$ and the functoriality of $\Ztilde$, we complete the proof.
 We first use the decomposition $\Delta_t = \psi^{-1}\circ\Delta$ (see Table~\ref{tab:ElemCob}) and get
\[(\log\Ztilde(\Delta_t))_{\leq 1}
 = \strutgraph{1^{+}}{1^{-}} + \strutgraph{1^{+}}{2^{-}} - \frac{1}{2}\dYgraph{1^{+}}{1^{-}}{2^{-}}\ .\]
 Next, considering a bottom-top tangle $c'=(\mu\otimes\mu)\circ(\Id_1\otimes\Delta_t\otimes\Id_1)\circ(v_+\otimes v_-\otimes v_+)$ similar to $c$ and the decomposition $\Delta_b=(\mu\otimes\Id_1)\circ(\Id_1\otimes c')$ (see \cite[Tables~5.1 and 5.2]{CHM08}), we have
\[(\log\Ztilde(\Delta_b))_{\leq 1}
 = \strutgraph{1^{+}}{1^{-}} + \dstrutgraph{1^{-}}{2^{-}} - \frac{1}{2}\dYgraph{1^{+}}{1^{-}}{2^{-}}\ .\]
\end{proof}

\begin{proof}[Proof of Theorem~\ref{thm:main}]
We first prove the first commutative diagram, that is,
\[
\Ztilde^Y_{n+1}(\S(J)) = \frac{1}{2}(\delta(J)+ \Y(J)) \in \A_{n+1}^Y\otimes\Q/\Z
\]
for a Jacobi diagram $J \in \A_n^Y$.
The commutativity
\[
\bar{z}_{n+1}(\ss(J)) = \frac{1}{2}\delta(J) \in \A_{n+1}^c\otimes\Q/\Z
\]
of the second diagram follows from the the first one.
Indeed, for a connected Jacobi diagram $J \in \A_{n}^c$, if $n\geq 2$, we see
\[
\bar{z}_{n+1}(\ss(J)) = \Ztilde^Y_{n+1}(\ss(J)) = -\Ztilde^Y_{n+1}(\S(J)) = \frac{1}{2}\delta(J).
\]
In the case $n=1$, one has
\[
 \bar{z}_{2}(\ss(J)) = \Ztilde^Y_{2}(\ss(J)) -\frac{1}{2}\Ztilde^Y_{1}(\ss(J))\sqcup\Ztilde^Y_{1}(\ss(J)) 
 = \left( \frac{1}{2}\delta(J) + \frac{1}{2}J\sqcup J \right) + \frac{1}{2}J\sqcup J
 = \frac{1}{2}\delta(J).
\]

Let us prove the first commutative diagram.
There are $3n$ half-edges incident to trivalent vertices of $J$, which are denoted by $e_{g+1}, e_{g+2},\dots,e_{g+3n}$, and let $N=\{g+1,g+2,\dots,g+3n\}$.
Define $V$, $E$, $L_{t,i}$, and $L_{b,i}$ for $1\le i\le g$ by
\begin{align*}
V &= \left\{(j,k,l) \in N^3 \biggm| \text{\parbox{14em}{$e_j, e_k, e_l$ are the three half-edges incident to a trivalent vertex}} \right\} \Big/\text{cyclic permutation}, \\
E &= \left\{(j,k) \in N^2 \biggm| \text{\parbox{17em}{$e_j, e_k$ are the two half-edges of an edge connecting two trivalent vertices}} \right\} \Big/\text{permutation}, \\
L_{t,i} &= \{j \in N \mid \text{the univalent vertex of the edge containing $e_j$ is colored with $i^{+}$}\}, \\
L_{b,i} &= \{j \in N \mid \text{the univalent vertex of the edge containing $e_j$ is colored with $i^{-}$}\}. 
\end{align*}
Set $r_i = \#L_{t,i}$ and $s_i = \#L_{b,i}$.

\begin{figure}[h]
 \centering
 \includegraphics[width=0.5\textwidth]{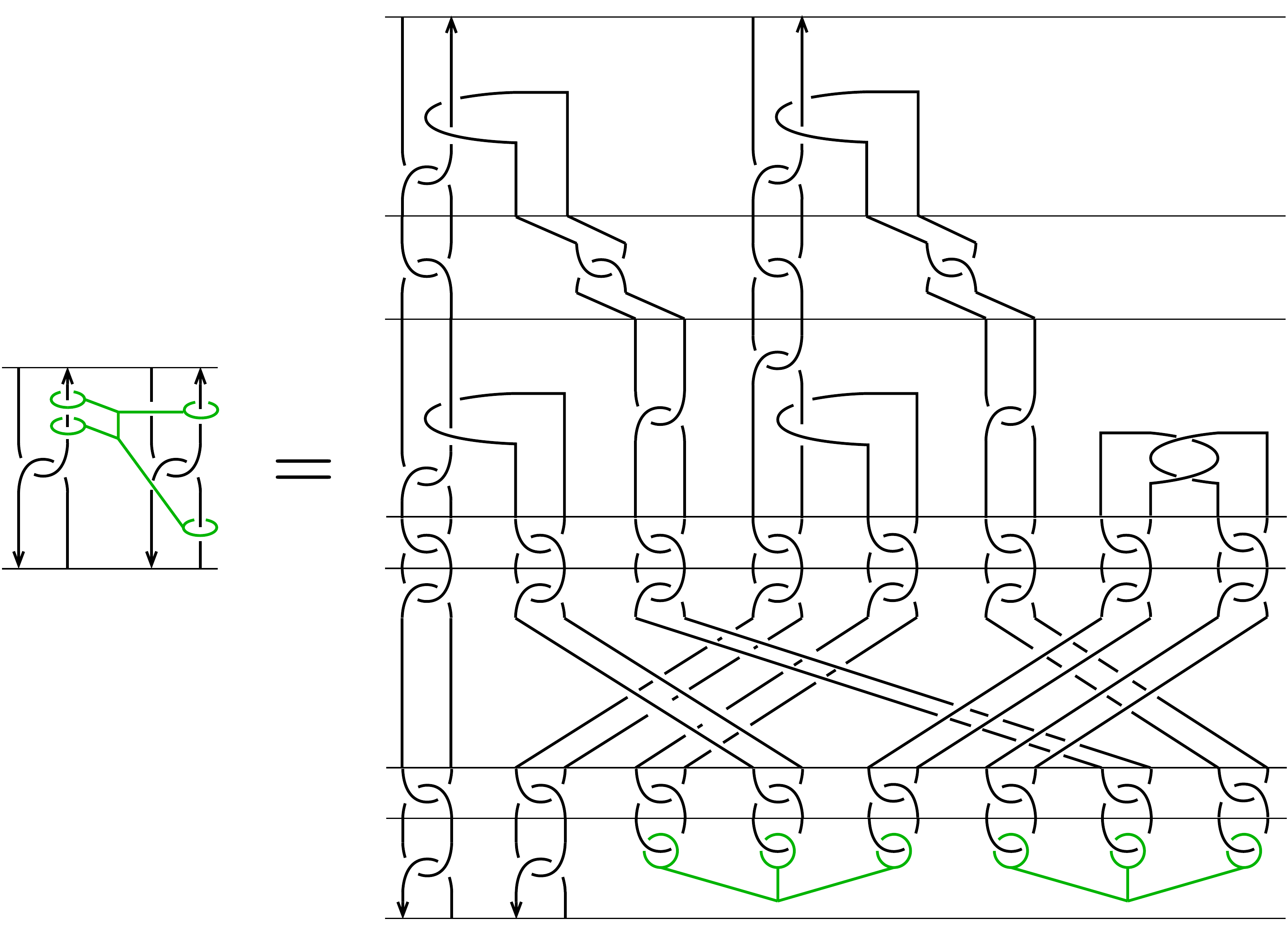}
 \caption{A decomposition of $(\Id_2\otimes Y^{\otimes 2})\circ\Psi \circ (\Delta_t^2\otimes((\Delta_b\otimes \Id_1)\circ \Delta_t) \otimes c)$.
 See \cite[Section~2.3]{CHM08} for the composition of bottom-top tangles.}
 \label{fig:Decomposition}
\end{figure}

Let $G$ be a graph clasper (in the bottom-top tangle $([-1,1]^3,\gamma_g)$ corresponding to the cobordism $\Sigma_{g,1}\times[-1,1]$) realizing $J$ such that $([-1,1]^3,\gamma_g)_G$ is decomposed as
\[
(\Id_g \otimes Y^{\otimes n}) \circ \Psi \circ \left( \left( \bigotimes_{i=1}^g((\Delta_b^{s_i}\otimes\Id_{r_i})\circ\Delta_t^{r_i}) \right) \otimes c^{\otimes\#E} \right),
\]
where $\Psi$ consists of $\psi$, $\psi^{-1}$ (in Section~\ref{sec:LMO}) and $\Id_m$ (see Figure~\ref{fig:Decomposition}).
Note that $G$ can be more complicated in general, but well-definedness of the surgery map (see Section~\ref{sec:surgery}) enables us to choose $G$ so that $([-1,1]^3,\gamma_g)_G$ has the above decomposition.
Indeed, one can put $G$ in the front side except for the leaves, and move the trivalent vertices to the bottom-right.

We denote by $\gamma$ the third factor of the above decomposition.
Also, we may assume that the $j$-th leaf among $3n$ leaves of $n$ $Y$'s in the decomposition arises from the half-edge $e_{g+j}$.
Then, by Section~\ref{sec:LMO} and Lemma~\ref{lem:Delta}, we have the following computation
\begin{gather*}
\begin{split}
 \left( \log\Ztilde(\Psi \circ \gamma) \right)_{\leq 1} 
 &= \sum_{i=1}^g\left( \strutgraph{i^{+}}{i^{-}} + \sum_{l \in L_{t,i}}\strutgraph{i^{+}}{l^{-}} + \sum_{l \in L_{b,i}}\dstrutgraph{i^{-}}{l^{-}} + \sum_{l \in L_{t,i}} -\frac{1}{2}\dYgraph{i^{+}}{i^{-}}{l^{-}} + \sum_{k<l \in L_{t,i}} \pm\frac{1}{2}\dYgraph{i^{+}}{k^{-}}{l^{-}} \right.\\
 &\qquad\qquad \left. + \sum_{l \in L_{b,i}} -\frac{1}{2}\dYgraph{i^{+}}{i^{-}}{l^{-}} + \sum_{k<l \in L_{b,i}} \pm\frac{1}{2}\ddYgraph{i^{-}}{k^{-}}{l^{-}}\ \right) - \sum_{(j,k)\in E} \dstrutgraph{j^{-}}{k^{-}}\ ,
\end{split}\\
 \left( \log\Ztilde(\Id_g \otimes Y^{\otimes n}) \right)_{\leq 2} 
 = \sum_{i=1}^g \strutgraph{i^{+}}{i^{-}} + \sum_{(j,k,l)\in V}\left(
-\uYgraph{l^{+}}{k^{+}}{j^{+}} +\frac{1}{2}\uHgraph{l^{+}}{l^{+}}{k^{+}}{j^{+}} +\frac{1}{2}\uHgraph{k^{+}}{k^{+}}{j^{+}}{l^{+}} +\frac{1}{2}\uHgraph{j^{+}}{j^{+}}{l^{+}}{k^{+}}\ 
\right),
\end{gather*}
where $\pm$'s are determined by $\Psi$.
Here one has
\begin{align*}
 \Ztilde(\S(J)) &= \sum_{G' \subset G} (-1)^{|G'|} \Ztilde((\Sigma_{g,1}\times[-1,1])_{G'}) 
 = (-1)^{n+|G|}(\Id_g \otimes (\emptyset-\Ztilde(Y))^{\otimes n}) \circ \Ztilde(\Psi\circ \gamma),
\end{align*}
where the first equality is just the definition of $\S$ in Section~\ref{sec:F-filtration} and the second equality follows from \cite[Proof of Theorem~7.11]{CHM08}.
Since $\emptyset - \Ztilde(Y) \in \widehat{\A}_{\geq 1}$, we have
\[
(-1)^{n+|G|} \Ztilde_{n+1}(\S(J)) = (\Id_g\otimes(\emptyset - \Ztilde(Y))^{\otimes n}_{n+1})\circ\Ztilde_{0}(\Psi\circ \gamma) + (\Id_g\otimes(\emptyset - \Ztilde(Y))^{\otimes n}_{n})\circ\Ztilde_{1}(\Psi\circ \gamma).
\]

Let us observe the first term in the right-hand side.
By considering the internal degrees, a term in $\Id_g\otimes(\emptyset - \Ztilde(Y))^{\otimes n}_{n+1}$ which might survive after composing with $\Ztilde_{0}(\Psi\circ \gamma)$ is either of the form
\[
\Id_g \sqcup \frac{(-1)^{n-1}}{2} \uHgraph{l'^+}{l'^+}{k'^+}{j'^+}\ \sqcup \bigsqcup_{(j,k,l)\neq(j',k',l') \in V} \uYgraph{l^+}{k^+}{j^+}
\]
for $(j',k',l') \in V$ with $l' \in \bigcup_i L_{t,i}\cup L_{b,i}$ or of the form
\[
\Id_g \sqcup \frac{(-1)^{n+1}}{2!} \uYgraph{l'^+}{k'^+}{j'^+} \sqcup \uYgraph{l'^+}{k'^+}{j'^+} \sqcup \bigsqcup_{(j,k,l)\neq(j',k',l') \in V} \uYgraph{l^+}{k^+}{j^+}
\]
for $(j',k',l') \in V$ with $j', k', l' \in \bigcup_i L_{t,i}\cup L_{b,i}$.
The former corresponds to the first term of $\delta'(J)$, and the latter corresponds to $\Y(J)$.

We next see the second term $(\Id_g\otimes(\emptyset - \Ztilde(Y))^{\otimes n}_{n})\circ\Ztilde_{1}(\Psi\circ \gamma)$.
It follows from $\Ztilde_{1}(\Psi\circ \gamma) = \Ztilde_{0}(\Psi)\circ\Ztilde_{1}(\gamma)$ that $\Ztilde_{1}(\Psi\circ \gamma)$ is the disjoint union of
\[
\exp\left(
\sum_{i=1}^g\left( \strutgraph{i^{+}}{i^{-}} + \sum_{l \in L_{t,i}}\strutgraph{i^{+}}{l^{-}} + \sum_{l \in L_{b,i}}\dstrutgraph{i^{-}}{l^{-}} \right) - \sum_{(j,k)\in E} \dstrutgraph{j^{-}}{k^{-}} \right)
\]
and
\begin{align}
\label{eq:gamma}
\sum_{i=1}^g\left( \sum_{l \in L_{t,i}} -\frac{1}{2}\dYgraph{i^{+}}{i^{-}}{l^{-}} + \sum_{k<l \in L_{t,i}} \pm\frac{1}{2}\dYgraph{i^{+}}{k^{-}}{l^{-}} + \sum_{l \in L_{b,i}} -\frac{1}{2}\dYgraph{i^{+}}{i^{-}}{l^{-}} + \sum_{k<l \in L_{b,i}} \pm\frac{1}{2}\ddYgraph{i^{-}}{k^{-}}{l^{-}}\ \right).
\end{align}
Here one has
\[
\Id_g\otimes(\emptyset - \Ztilde(Y))^{\otimes n}_{n} = \Id_g \sqcup \bigsqcup_{(j,k,l) \in V} \uYgraph{l^{+}}{k^{+}}{j^{+}}\ .
\]
Thus, the non-trivial terms of $(\Id_g\otimes(\emptyset - \Ztilde(Y))^{\otimes n}_{n})\circ\Ztilde_{1}(\Psi\circ \gamma)$ arising from the first and third (resp.\ the second and fourth) terms in (\ref{eq:gamma}) give the second term of $\delta'(J)$ (resp.\ $\delta''(J)$).
\end{proof}

\section{Properties of the homomorphism $\bar{z}_{n+1}$}
In this section, we study the structures of the modules of 1-loop Jacobi diagrams and tree Jacobi diagrams (Sections~\ref{section:oneloop}~and~\ref{section:treepart}).
We also show that the $0$-loop part of the homomorphism $\bar{z}_{n+1}$ is an extension of the higher Sato-Levine invariants defined by Conant, Schneiderman, and Teichner (Section~\ref{section:satolevine}),
and give some information on the kernel of the surgery map $\ss$ (Section~\ref{section:kernel-tree}).

\subsection{The module of 1-loop Jacobi diagrams}\label{section:oneloop}
Here, we determine the structure of the module $\A_{n,1}^c$ of connected 1-loop Jacobi diagrams.

Let $H$ denote the free module generated by $\{1^{\pm},\ldots, g^{\pm}\}$.
We denote by $\D_{2n}=\braket{x,y \mid (xy)^2, x^n, y^2}$ the dihedral group of order $2n$ which acts on $H^{\otimes n}$ as
\[
x\cdot(a_1\otimes a_2\otimes\cdots\otimes a_{n})=a_{n}\otimes  a_1\otimes a_2\otimes\cdots\otimes a_{n-1},\
y\cdot(a_1\otimes a_2\otimes\cdots\otimes a_{n})=(-1)^na_n\otimes \cdots\otimes a_2\otimes a_1.
\]
Let us denote by $(H^{\otimes n})_{\D_{2n}}$ the coinvariants of this action.
The module $\A_{n,1}^c$ is described as follows.
\begin{prop}\label{prop:oneloop}
Let $n\ge2$. 
The homomorphism
\[
\Phi\colon (H^{\otimes n})_{\D_{2n}}\to \A_{n,1}^c,\quad
a_1\otimes a_2\otimes\cdots\otimes a_n\mapsto
O(a_1,a_2,\ldots,a_n)
\]
is well-defined, and it induces an isomorphism.
\end{prop}

To prove Proposition~\ref{prop:oneloop},
we review the free quasi-Lie algebra introduced by Levine.
In \cite{Lev02}, Levine defined the free quasi-Lie algebra,
in which the self-annihilation relation $[x,x]=0$ of the free Lie algebra is replaced by the weaker anti-symmetry relation $[x,y]=-[y,x]$.
Levine identified the free quasi-Lie algebra generated by $\{1^{\pm},\ldots, g^{\pm}\}$ with the module generated by rooted trees whose univalent vertices are colored by the set $\{1^{\pm},\ldots,g^{\pm}\}$ modulo the AS relation and IHX relation in \cite[Section~3]{Lev02}.
The induced Lie bracket on the module of rooted trees is given by gluing neighborhoods of the roots of two rooted trees.
Let us denote by $L_n$ and $L'_n$ the degree $n$ parts of the free Lie algebra and free quasi-Lie algebra generated by $\{1^{\pm},\ldots, g^{\pm}\}$, respectively.
For $k\ge1$, 
Levine showed that the natural projection $\gamma_{2k-1}\colon L'_{2k-1}\to L_{2k-1}$ is an isomorphism in ~\cite[Lemma~2.1]{Lev02}
and that there is an exact sequence
\begin{equation}\label{eq:quasi-lie}
\begin{CD}
0@>>>L_k\otimes\Z/2\Z@>\theta_k>> L'_{2k}@>\gamma_{2k}>> L_{2k}@>>>0,
\end{CD}
\end{equation}
where $\theta_k$ is the induced map by the Lie bracket 
\[
L'_k\to L'_{2k},\quad
J\mapsto [J,J]
\]
in \cite[Theorem~2.2]{Lev06}.
\begin{proof}[Proof of Proposition~\ref{prop:oneloop}]
By the $\AS$ relation,
we see that 
\[
O(a_1,a_2,\ldots, a_n)=(-1)^nO(a_n,\ldots,a_2, a_1).
\]
Thus, $\Phi$ is well-defined.
We construct the inverse map of $\Phi$ as follows.
First, let us pick a connected 1-loop Jacobi diagram $J$.
Note that we consider a Jacobi diagram as a uni-trivalent graph
equipped with a cyclic order of incident edges in each trivalent vertex without the AS and IHX relations here.

In the Jacobi diagram $J$, the set of non-separating edges forms a cycle.
When $k\ge3$, we denote the edges as $e_1,e_2,\ldots, e_k$ in order,
so that $e_i$ and $e_{i+1}$ has one common endpoint $v_i$ for $i=1,2,\ldots, k$,
where $e_{k+1}=e_1$.
When $k=2$, we denote the two edges by $e_1$ and $e_2$ and their two common endpoints by $v_1$ and $v_2$.
Let us cut $J$ at each trivalent vertex $v_i$,
and focus on the tree component which does not contain $e_i$.
We regard it as a rooted tree, and denote it by $T_i$.
If we change the cyclic order of edges in $v_i$ using the AS relation,
we can change $J$ to the Jacobi diagram of the form as below.
\begin{center}
\begin{tikzpicture}[baseline=1.4ex, scale=0.4, dash pattern={on 2pt off 1pt}]
\draw (0,0) circle [radius=2.8]; 
\draw (10:2.8) -- (10:4);
\draw (50:2.8) -- (50:4);
\draw (90:2.8) -- (90:4);
\draw (130:2.8) -- (130:4);
\node [right] at (10:3.8) {$T_3$};
\node [above right] at (50:3.6) {$T_2$};
\node [above] at (90:3.8) {$T_1$};
\node [above left] at (130:3.6) {$T_k$};
\node at (-10:3.5){$\cdot$};
\node at (-25:3.5){$\cdot$};
\node at (-40:3.5){$\cdot$};
\node at (150:3.5){$\cdot$};
\node at (165:3.5){$\cdot$};
\node at (10:2.1){$v_3$};
\node at (50:2.1){$v_2$};
\node at (90:2.1){$v_1$};
\node at (130:2.1){$v_k$};
\end{tikzpicture}
\end{center}
Thus, we have obtained a cyclically ordered set of rooted trees $(T_1,T_2,\ldots, T_k)$ from a connected 1-loop Jacobi diagram modulo the AS relation.
To identify a connected 1-loop Jacobi diagram with that of the form above,
we need to choose an orientation of the cycle formed by the non-separating edges.
Hence, what we actually obtained is a map from the set of connected 1-loop Jacobi diagrams modulo the AS relation to the cyclically ordered set of rooted trees $(T_1,T_2,\ldots, T_k)$ with signs modulo $(T_1,T_2,\ldots, T_k)=(-1)^k(T_k,\ldots, T_2, T_1)$.

Recall that the module of rooted trees is identified with the free quasi-Lie algebra.
Using the composite map of the natural projection $\gamma_m\colon L'_m\to L_m$ and the natural embedding $L_m\to H^{\otimes m}$ for $m\ge1$,
and taking tensor products, 
we have a map from the set of cyclically ordered rooted trees $(T_1,T_2,\ldots, T_k)$ with signs to $(H^{\otimes n})_{\braket{x}}$,
where $n$ is the sum of the numbers of univalent vertices in $T_1, T_2, \ldots, T_k$.

Finally, to erase the ambiguity coming from the choice of an orientation of the cycle,
take the quotient by the inversion $y$ as $(H^{\otimes n})_{\D_{2n}}=((H^{\otimes n})_{\braket{x}})_{\braket{y}}$.
Let us denote the number of univalent vertices of $T_i$ by $m_i$, i.e., $T_i \in L'_{m_i}$ for $1\le i\le k$.
Since $\sum_{i=1}^k(m_i-1)+k=n$, we have 
\[
\gamma_{m_1}(T_1)\otimes\gamma_{m_2}(T_2)\otimes\cdots\otimes\gamma_{m_k}(T_k)
=(-1)^k\gamma_{m_k}(T_k)\otimes\cdots\otimes\gamma_{m_2}(T_2)\otimes\gamma_{m_1}(T_1)\in(H^{\otimes n})_{\D_{2n}}.
\]
Thus, the above element of $(H^{\otimes n})_{\D_{2n}}$ does not depend on the choice of an orientation of the cycle formed by the non-separating edges.

The above map from connected 1-loop Jacobi diagrams to $(H^{\otimes n})_{\D_{2n}}$ is actually invariant under the IHX relation with respect to $e_i$, $e_{i+1}$, and the edge in $T_i$ incident to the root $v_i$.
It is also invariant under the IHX relation with respect to the other edges in $T_i$ since the IHX relation in $T_i$ corresponds to the Jacobi identity in the free quasi-Lie algebra.
Thus, the map induces a homomorphism
$\A_{n,1}^c\to (H^{\otimes n})_{\D_{2n}}$,
and it is the inverse map of $\Phi$.
\end{proof}

For $a_1,\ldots, a_n\in \{1^{\pm},2^{\pm},\ldots, g^{\pm}\}$,
we call a $1$-loop Jacobi diagram $O(a_1,a_2,\ldots,a_n)$ \emph{symmetric} if there exists $1\le k\le n$ such that $a_{n-i}=a_{k+i}$ for $0\le i\le n-1$ with indices modulo $n$,
and denote by $\A_{n,1}^{c,s}$ the submodule of $\A_{n,1}^c$ generated by symmetric connected 1-loop Jacobi diagrams of $\ideg=n$.

\begin{prop}\label{prop:oneloopbasis}
Let $n\ge2$.
The rank of the $\Z$-module $\A_{n,1}^c$ is as follows:
\[
\rank\A_{n,1}^c=
\begin{cases}
\displaystyle\frac{1}{2}N_{2g}(n)+\frac{1}{4}(2g+1)(2g)^{\frac{n}{2}}&\text{if }n\text{ is even},\\[8pt]
\displaystyle\frac{1}{2}N_{2g}(n)-\frac{1}{2}(2g)^{\frac{n+1}{2}}&\text{if }n\text{ is odd},
\end{cases}
\]
where $N_{2g}(n)=(\sum_{d \mid n}\varphi(d)(2g)^{n/d})/n$ and $\varphi(d)$ is Euler's totient function.
The torsion subgroup of $\A_{n,1}^c$ is non-trivial only when $n$ is odd,
and described as
\[
\tor\A_{n,1}^c
\cong \A_{n,1}^{c,s}
\cong H^{\otimes\frac{n+1}{2}}\otimes \Z/2\Z.
\]
\end{prop}
\begin{proof}
Let us denote the set of words in $\{1^{\pm},2^{\pm},\ldots, g^{\pm}\}$ of length $n$ by $W_n$, or simply by $W$.
Recall that we have an isomorphism $\A_{n,1}^c\cong (H^{\otimes n})_{\D_{2n}}$ in Proposition~\ref{prop:oneloop}.
First, we interpret the module $(H^{\otimes n})_{\D_{2n}}$ in terms of $W$.
Consider the action of $\D_{2n}$ on $W$ defined by
\[
x\ast(a_1a_2\cdots a_{n})=a_{n}a_1a_2\cdots a_{n-1},\
y\ast(a_1a_2\cdots a_{n})=a_n\cdots a_2 a_1.
\]
When $n$ is even, 
there is a natural $\D_{2n}$-equivariant injective map 
\[
W\to H^{\otimes n},\quad
a_1a_2\cdots a_n\mapsto a_1\otimes a_2\otimes\cdots\otimes a_n
\]
 since $y\in\D_{2n}$ does not change signs of elements in $H^{\otimes n}$.
It induces an isomorphism $\mathbb{Z}(W/\D_{2n})\cong (H^{\otimes n})_{\D_{2n}}$.
Consider the case when $n$ is odd.
Define an action of $\D_{2n}$ on $\mathbb{Z}W$ by
\[
x\cdot(a_1a_2\cdots a_{n})=a_{n}a_1a_2\cdots a_{n-1},\ 
y\cdot(a_1a_2\cdots a_{n})=(-1)^na_n\cdots a_2 a_1.
\]
Note that this action is different from what is induced by the action of $\D_{2n}$ on $W$.
Under this action,
the homomorphism $\mathbb{Z}W\to H^{\otimes n}$ is $\D_{2n}$-equivariant,
and induces an isomorphism
$(\mathbb{Z}W)_{\D_{2n}}\cong (H^{\otimes n})_{\D_{2n}}$.
Since the action of $x\in\D_{2n}$ on $\Z W$ is the same as what is induced by the action on $W$, 
$(\Z W)_{\braket{x}}$ is a free module with basis identified with the set $W/\braket{x}$, which we denote by $\overline{W}_n$ or  $\overline{W}$.
Let us call an element $a_1a_2\cdots a_n$ in $\overline{W}$ a cyclic word.
If there exists $1\le k\le n$ such that $a_{n-i}=a_{k+i}$ for $0\le i \le n-1$ with indices modulo~$n$,
we call $a_1\cdots a_n\in \overline{W}$ \emph{symmetric} in the same way as symmetric Jacobi diagrams,
and denote the set of symmetric cyclic words by $\overline{W}_n^s$ or $\overline{W}^s$.
Since the fixed point set of the action of $y\in\D_{2n}$ on $\overline{W}=W/\braket{x}$ is $\overline{W}^s$,
we have isomophisms
\[
(H^{\otimes n})_{\D_{2n}}\cong (\mathbb{Z}W)_{\D_{2n}}\cong ((\mathbb{Z}W)_{\braket{x}})_{\braket{y}}\cong \Z((\overline{W}\setminus \overline{W}^s)/\braket{y})\oplus (\Z/2\Z)\overline{W}^s.
\]

Next, we compute the orders of the sets $W/\D_{2n}$, $(\overline{W}\setminus \overline{W}^s)/\braket{y}$, and $\overline{W}^s$.
The order of $W/\D_{2n}$ is the number of $(2g)$-ary bracelets of length $n$,
and is equal to $N_{2g}(n)/2+(2g+1)(2g)^{\frac{n}{2}}/4$ when $n$ is even
and $N_{2g}(n)/2+(2g)^{\frac{n+1}{2}}/2$ when $n$ is odd.
It is calculated using P\'olya's enumeration theorem (see, e.g., in \cite[p.~173]{Bol79}).
We also have the equality
\[
|W/\D_{2n}|
=\frac{1}{2}|\overline{W}\setminus \overline{W}^s|+|\overline{W}^s|
=\frac{1}{2}|\overline{W}|+\frac{1}{2}|\overline{W}^s|.
\]
Since the number of $(2g)$-ary necklaces $|\overline{W}|$ of length $n$ is $N_{2g}(n)$ as in \cite[p.~173]{Bol79},
when $n$ is odd, we have 
\begin{align*}
|\overline{W}^s|
&=2|W/\D_{2n}|-|\overline{W}|=(2g)^{\frac{n+1}{2}},\\
|(\overline{W}\setminus \overline{W}^s)/\braket{y}|
&=\frac{1}{2}|\overline{W}\setminus \overline{W}^s|=\frac{1}{2}N_{2g}(n)-\frac{1}{2}(2g)^{\frac{n+1}{2}}.
\end{align*}

Lastly, we construct an isomorphism $H^{\otimes\frac{n+1}{2}}\otimes\Z/2\Z\cong \tor((H^{\otimes n})_{\D_{2n}})$ when $n$ is odd.
The homomorphism
\[
W_{\frac{n+1}{2}}\to \overline{W}_n^s,\quad
a_1a_2\cdots a_{\frac{n+1}{2}}\mapsto a_1a_2\cdots a_{\frac{n+1}{2}}a_{\frac{n+1}{2}}\cdots a_2
\]
is surjective because every element of $\overline{W}_n^s$ has a form $a_1a_2\cdots a_{\frac{n+1}{2}}a_{\frac{n+1}{2}}\cdots a_2$.
Moreover, we see that it is bijective by comparing their orders.
It induces isomorphisms
\[
H^{\otimes\frac{n+1}{2}}\otimes\Z/2\Z
\cong (\Z/2\Z)W_{\frac{n+1}{2}}
\cong (\Z/2\Z)\overline{W}_n^s
\cong \tor((H^{\otimes n})_{\D_{2n}}),
\]
which maps $a_1\otimes a_2\otimes \cdots \otimes a_{\frac{n+1}{2}}$ to $a_1\otimes a_2\otimes\cdots\otimes a_{\frac{n+1}{2}}\otimes a_{\frac{n+1}{2}}\otimes \cdots \otimes a_2$.
\end{proof}

In \cite[Section~8.5]{Hab00C}, \cite[Section~8.2]{CHM08}, and \cite[Section~1.2]{HaMa09},
they considered a Hopf algebra over $\Q$ with a symplectic group action denoted by $A(\Sigma)$, $A(F_g)$, or $\A^{<}(H_\Q)$.
It can be defined with integral coefficients which results in a Hopf algebra whose primitive part of $\ideg=n$ is isomorphic to $\A_n^c$ as in \cite[Remark~6.6]{HaMa12}.
The authors do not know descriptions of $\A_{n,1}^c$ or $\A_{n,1}^c\otimes\Q$ as the $\Sp(2g;\Z)$- or $\Sp(2g;\Q)$-module under the identification except for the submodule $\tor\A_{n,1}^c\cong H^{\otimes \frac{n+1}{2}}\otimes\Z/2\Z$ when $n$ is odd.

Instead, we consider a submodule of $\A_{n,1}^c$
which we use to estimate the image of $\bar{z}_{n+1}$ in Lemma~\ref{lem:periodicpart}.
Let $\A_{2n,1}^{c,\mathrm{period}}$ denote the abelian group generated by 1-loop Jacobi diagrams $O(a_1,a_2,\ldots,a_n,a_1,a_2,\ldots,a_n)$,
where $a_1,a_2,\ldots, a_n\in\{1^{\pm},2^{\pm},\ldots, g^{\pm}\}$.
We also denote $\A_{2n,1}^{c,s,\mathrm{period}}=\A_{2n,1}^{c,s}\cap \A_{2n,1}^{c,\mathrm{period}}$,
which is equal to the submodule generated by symmetric and periodic 1-loop Jacobi diagrams, namely,
$O(a_1,a_2,\ldots,a_n,a_1,a_2,\ldots,a_n)$ such that there exists $1\le k\le n$ satisfying $a_{n-i}=a_{k+i}$ for all $0\le i\le n-1$ with indices modulo $n$.
\begin{lemma}\label{lem:oneloop-period}
Let $n\ge2$.
There are isomorphisms
\[
\A_{n,1}^c\otimes M\cong \A_{2n,1}^{c,\mathrm{period}}\otimes M,\ 
\A_{n,1}^{c,s}\cong \A_{2n,1}^{c,s,\mathrm{period}}\otimes M,
\]
where $M=\Z$ when $n$ is even, and $M=\Z/2\Z$ when $n$ is odd.
\end{lemma}
\begin{proof}
Let $W_{2n}^{\mathrm{period}}$ be the set of periodic words of length $2n$
of the form $a_1a_2\cdots a_na_1a_2\cdots a_n$ for $a_1,a_2,\ldots,a_n\in \{1^\pm,\ldots,g^{\pm}\}$.
The submodule $\A_{2n,1}^{c,\mathrm{period}}\subset \A_{2n,1}^{c}$
has a basis which corresponds to $W_{2n}^{\mathrm{period}}/\D_{4n}$.
There is a bijection 
\[
W_n\to W_{2n}^{\mathrm{period}},\quad
a_1a_2\cdots a_n\mapsto a_1a_2\cdots a_na_1a_2\cdots a_n
\]
whose inverse map is defined by taking the first $n$-words.
The element $x^n\in \D_{4n}$ acts on $W_{2n}^{\mathrm{period}}$ trivially, 
and the action of $D_{4n}/\braket{x^n}$ on $W_{2n}^{\mathrm{period}}$ coincides with that of $D_{2n}$ on $W_n$ under the above bijection.
Thus, the bijection induces an isomorphism $\A_{n,1}^c\cong \A_{2n,1}^{c,\mathrm{period}}$ when $n$ is even.
Since the bijection maps symmetric cyclic words to symmetric cyclic words,
we also have an isomorphism $\A_{n,1}^{c,s}\cong \A_{2n,1}^{c,s,\mathrm{period}}$.
When $n$ is odd, we do not have a map $\A_{n,1}^c\to\A_{2n,1}^{c,\mathrm{period}}$ because the action of $y\in D_{2n}$ on $\Z W_n$ changes signs,
but we have isomorphisms $\A_{n,1}^c\otimes\Z/2\Z\cong\A_{2n,1}^{c,\mathrm{period}}\otimes\Z/2\Z$ and $\A_{n,1}^{c,s}\cong\A_{2n,1}^{c,s,\mathrm{period}}\otimes\Z/2\Z$ in the same way.
\end{proof}

\subsection{The module structure of $\A_{n,0}^c$}\label{section:treepart}
Recall that $L'_n$ is the degree $n$ part of the free quasi-Lie algebra and that $\gamma_n\colon L'_n\to L_n$ is the natural projection explained right after Proposition~\ref{prop:oneloop}.
Let $D'_n$ denote the kernel of the bracket map $[\ ,\ ]\colon H\otimes L'_{n+1}\to L'_{n+2}$.
Levine computed the structure of $D'_n$ as follows.
\begin{theorem}[{\cite[Corollary~2.3]{Lev06}}]\label{thm:treepart}
For $k\ge1$, the sequences
\[
\begin{CD}
0@>>> D'_{2k} @>\id_H\otimes\gamma_{2k+1}>> D_{2k}@>\sl>> L_{k+1}\otimes\Z/2\Z@>>>0,\\
0@>>> (H\otimes L_k)\otimes \Z/2\Z @>\sq>> D'_{2k-1}@>\id_H\otimes\gamma_{2k}>> D_{2k-1}@>>>0
\end{CD}
\]
are exact.
\end{theorem}

Here, the homomorphism $\sq\colon (H\otimes L_k)\otimes\Z/2\Z\to D'_{2k-1}$ is induced by the map
\[
\id_H\otimes\theta_k\colon
H\otimes L'_k
\to D'_{2k-1},\quad
x\otimes T\,\mapsto\,
x\otimes[T,T],
\]
and the homomorphism $\sl\colon D_{2k}\to L_{k+1}\otimes\Z/2\Z$ is explained in \cite[Remark~2.4]{Lev06} and \cite[Definition~5.6]{CST12W}.

Levine~\cite{Lev02} constructed a homomorphism $\eta'\colon \A_{n,0}^c\to D'_n$ which was shown to be an isomorphism in \cite[Theorem~1.4]{CST12L}.
It is defined as the sum of $\ell(v)\otimes T_v\in H\otimes L'_{n+1}$ over all univalent vertices $v$ in a tree Jacobi diagram $T$,
where $T_v$ is a rooted tree obtained by changing $v$ to a root.
For example, it sends
\begin{align*}
T(a_1,a_2,a_3)&\mapsto a_1\otimes [a_2,a_3]+a_2\otimes [a_3,a_1]+a_3\otimes [a_1,a_2],\\
T(a_1,a_2,a_3,a_4)&\mapsto a_1\otimes [a_2,[a_3,a_4]]+a_2\otimes [[a_3,a_4],a_1]+a_3\otimes[a_4,[a_1,a_2]]+a_4\otimes[[a_1,a_2],a_3].
\end{align*}
\begin{remark}
The definition of the map $\eta'\colon \A_{n,0}^c\to D'_n$ in this paper differs from \cite{CHM08,Lev02} by $(-1)^{n-1}$,
and coincides with those in \cite{CST12L,HaMa12}.
\end{remark}
The composite map $\eta'^{-1}\circ \sq\colon  (H\otimes L_k)\otimes\Z/2\Z\to \A_{2k-1,0}^c$ is induced by a homomorphism
\[
(H\otimes L'_k)\otimes\Z/2\Z\to \A_{2k-1,0}^c,\quad
x\otimes T\,\mapsto\,
\begin{tikzpicture}[baseline=1ex, scale=0.4, dash pattern={on 2pt off 1pt}]
\node [left] at (0,0) {$T$};
\draw (0,0)--(2,0);
\draw (1,0)--(1,1.2);
\node [right] at (2,0) {$T$,};
\node [above] at (1,1.2) {$x$};
\end{tikzpicture}
\]
and it gives an isomorphism $\tor\A_{2k-1,0}^c\cong (H\otimes L_k)\otimes \Z/2\Z$ for $k\ge 1$ as we see from the second exact sequence in Theorem~\ref{thm:treepart}.

\subsection{Relation to the higher Sato-Levine invariants}
\label{section:satolevine}
Let us denote the $0$-loop part of $\bar{z}_{n+1}\colon Y_n\I\C/Y_{n+1}\to\A_{n+1}^c\otimes\Q/\Z$ by $\bar{z}_{n+1,0}$.
Massuyeau~\cite[Theorem~5.13]{Mas12} showed that the $0$-loop part of the LMO functor is written as the total Johnson map induced by some symplectic expansion.
More simply, the $0$-loop part of $\Ztilde_{\le n}^Y$ is essentially the action $\rho_{n+1}$ on the nilpotent quotient $\pi_1\Sigma_{g,1}/\pi_1\Sigma_{g,1}(n+2)$.
See \cite[Sections~3.4 and 4.3]{HaMa12}, for details.
Thus, the homomorphism $\bar{z}_{n+1,0}\colon Y_n\I\C/Y_{n+1}\to\A_{n+1,0}^c\otimes\Q/\Z$ is also essentially the mod $\Z$ reduction of the degree $n+1$ part of the total Johnson map induced by the symplectic expansion.
Here, we show that it is also considered as an extension of the higher Sato-Levine invariants of framed links in \cite{CST16}.

Firstly, we review the Johnson homomorphisms on the $Y$-filtration constructed in \cite{GaLe05}.
Recall the homology cobordism group of homology cylinders over $\Sigma_{g,1}$ which we denote by $\I\H$ and the inclusions $m_+, m_-\colon \Sigma_{g,1}\to\partial M$ for a homology cylinder $(M,m)$ in Section~\ref{sec:homology cylinder}.
They induce a homomorphism
\[
\rho_n\colon\I\H\to \Aut(\pi_1\Sigma_{g,1}/\pi_1\Sigma_{g,1}(n+1)),
\quad
M\mapsto(m_-)_*^{-1}\circ(m_+)_*
\]
for all $n\ge2$ identifying the nilpotent quotients of $\pi_1\Sigma_{g,1}$ and $\pi_1 M$ by Stallings' theorem~\cite[Theorem~3.4]{Sta65}.
Identify the free module $H$ generated by $\{1^{\pm},\ldots, g^{\pm}\}$ with the first homology group $H_1(\Sigma_{g,1};\Z)$ as in right before Theorem~\ref{thm:str-of-y3}.
The quotient group $\pi_1\Sigma_{g,1}(n)/\pi_1\Sigma_{g,1}(n+1)$ is also identified with $L_n$.
By setting $J_n\H=\Ker \rho_n$,
the $n$-th Johnson homomorphism on $J_n\H$ is defined by
\[
\tau_n\colon J_n\H \to \Hom(H,L_{n+1}),\quad
M\mapsto ([\gamma]\mapsto [\rho_{n+1}(M)(\gamma)\gamma^{-1}])
\]
for $\gamma\in \pi_1\Sigma_{g,1}/\pi_1\Sigma_{g,1}(n+2)$,
and we see that $\Ker\tau_n=J_{n+1}\H$ by definition.
The intersection form $\mu\colon H\otimes H\to \Z$ induces an isomorphism
$H\cong H^*$ as $x\mapsto \mu(x,*)$,
and we identify $\Hom(H, L_{n+1})=H^*\otimes L_{n+1}\cong H\otimes L_{n+1}$.
Under the identification, the image $\Im\tau_n$ is equal to the kernel $D_n=\Ker(H\otimes L_{n+1}\to L_{n+2})$ of the bracket map as shown in \cite[Proposition~2.5]{GaLe05}.

\begin{remark}
Since $\bar{z}_{n+1,0}$ is essentially the mod $\Z$ reduction of the degree $n+1$ part of the total Johnson map as in \cite[Theorem~5.13]{Mas12},
it factors through $\Gr q\colon Y_n\I\C/Y_{n+1}\to Y_n\I\H/Y_{n+1}$.
\end{remark}

Secondly, we review a surjective homomorphism 
\[
\kappa\colon \Ker(\tau_{2k-1}\colon Y_{2k-1}\I\H/Y_{2k}\to D_{2k-1})\to L_{k+1}\otimes \Z/2\Z
\]
constructed by Conant, Schneiderman, and Teichner,
which is an analogue of the higher Sato-Levine invariant of framed links.
Recall the isomorphisms $\eta'\colon \A_{n,0}^c\cong D'_n$ in \cite[Theorem~1.4]{CST12L} for $n\ge1$.
As in \cite[Theorem~3]{Lev01},
the diagram
\begin{equation}
\label{diagram:AD}
\vcenter{
 \xymatrix{
\A_{n,0}^c\ar[r]^{\Gr q\circ\ss\qquad}\ar[rd]_(0.45){\eta'}^(0.6){\cong}& Y_n\I\H/Y_{n+1} \ar[r] & J_n\H/J_{n+1}\H \ar[d]_{\cong}^{\tau_n}\\
   &  D_n' \ar[r]^{\id_H\otimes\gamma_{n+1}} & D_n
  }
}
\end{equation}
commutes for $n\ge1$,
where $Y_n\I\H/Y_{n+1}\to J_n\H/J_{n+1}\H$ is the natural homomorphism induced by the inclusion $Y_n\I\H\to J_n\H$.
Here, $\Gr q\circ \ss\colon \A_{n,0}^c\to Y_n\I\H/Y_{n+1}$ is surjective for $n\ge2$,
where $\Gr q\colon Y_n\I\C/Y_{n+1}\to Y_n\I\H/Y_{n+1}$ is defined in Section~\ref{sec:homology cylinder}.
Note that there is an error in \cite{Lev01} which is corrected in \cite{Lev02}.
See also \cite[Theorem~7.7]{HaMa12}.
By using Theorem~\ref{thm:treepart},
it is shown that 
\[
\Coker(Y_{2k}\I\H/Y_{2k+1}\to J_{2k}\H/J_{2k+1}\H)
\cong \Coker(D'_{2k}\to D_{2k})\cong L_{k+1}\otimes\Z/2\Z
\]
and that $\Gr q\circ \ss\colon \A_{2k,0}^c\to Y_{2k}\I\H/Y_{2k+1}$ is an isomorphism
for $k\ge1$ in \cite[Theorem~6.2]{CST12L}.
See also \cite[Corollary~50]{CST16}.
In a similar way to the connecting homomorphism of the snake lemma,
the natural homomorphism
\[
\Ker(Y_{2k-1}\I\H/Y_{2k}\to J_{2k-1}\H/J_{2k}\H)\to
\Coker(Y_{2k}\I\H/Y_{2k+1}\to J_{2k}\H/J_{2k+1}\H)
\]
which maps $M\mapsto [M]$ is well-defined, and it induces a homomorphism
\[
\kappa\colon\Ker(\tau_{2k-1}\colon Y_{2k-1}\I\H/Y_{2k}\to D_{2k-1})\to L_{k+1}\otimes\Z/2\Z
\]
by the identification stated above.
In \cite[Proposition~25]{CST16}, $\kappa$ is shown to be surjective.
\begin{remark}
The map $\kappa$ is essentially the same as the higher Sato-Levine invariants in \cite[Section~5.1]{CST12W} on the Whitney tower filtration on the set of framed links in $S^3$.
\end{remark}

Thirdly, we relate the $0$-loop part $\bar{z}_{2k,0}\colon Y_{2k-1}\I\C/Y_{2k}\to A_{2k,0}^c\otimes \Q/\Z$ of $\bar{z}_{2k}$ with the homomorphism $\kappa$.
Define a homomorphism
\[
\nu\colon L'_{k+1}\otimes \Z/2\Z\to \A_{2k,0}^{c}\otimes\Q/\Z
\]
by $T\mapsto \!
\begin{tikzpicture}[mydash, baseline=-0.7ex, scale=0.25]
\draw (0,0.1) -- (1.5,0.1);
\node [left] at (0,0) {$\displaystyle(T$};
\node [right] at (1.5,0) {$T)$/2.};
\end{tikzpicture}$
As in the exact sequence (\ref{eq:quasi-lie}),
the kernel of the projection
$L'_{k+1}\otimes\Z/2\Z\to L_{k+1}\otimes\Z/2\Z$ induced by $\gamma_{k+1}$
is generated by $\Im\theta_{\frac{k+1}{2}}$ when $k$ is odd,
and trivial when $k$ is even.
By the IHX relation,
we see that $\nu$ factors through the projection.

\begin{prop}
\label{prop:highersatolevine}
For $k\ge1$,
the homomorphism
$\nu\colon L_{k+1}\otimes \Z/2\Z\to \A_{2k,0}^{c}\otimes\Q/\Z$ is injective,
and the following diagram commutes:
\[
\xymatrix{
\Ker(\tau_{2k-1}\colon Y_{2k-1}\I\C/Y_{2k}\to D_{2k-1}) \ar[r]^-{\kappa\circ \Gr q} \ar[d]_-{\inc} & L_{k+1}\otimes \Z/2\Z \ar[d]^-{\nu} \\
Y_{2k-1}\I\C/Y_{2k} \ar[r]^-{\bar{z}_{2k,0}} & \A_{2k,0}^{c}\otimes\Q/\Z.
}
\]
\end{prop}

\begin{proof}[Proof of Proposition~\ref{prop:highersatolevine}]
Let $\eta=(\id_H\otimes\gamma_{2k+1})\circ \eta'\colon \A_{2k,0}^c\to D_{2k}$,
which is equal to $\tau_{2k}\circ \ss$ by the commutative diagram~(\ref{diagram:AD}).
Consider the commutative diagram
\[
\xymatrix{
0 \ar[r] & \A_{2k,0}^c \ar[r]^-{\eta} \ar[d] & D_{2k} \ar[r]^-{\sl} \ar[d] & L_{k+1}\otimes\Z/2\Z \ar[r] \ar[d] & 0 \\
0 \ar[r] & \A_{2k,0}^c\otimes\Q \ar[r]^-{\eta} & D_{2k}\otimes\Q \ar[r] & 0, &
}
\]
where the first horizontal sequence is obtained from the isomorphism $\eta'\colon \A_{2k,0}^c\cong D'_{2k}$ in \cite[Theorem~1.4]{CST12L} and Theorem~\ref{thm:treepart}.
As explained in \cite[Section~3.8]{CST16},
the homomorphism $\sl$ maps
$\frac{1}{2}\eta
\begin{tikzpicture}[mydash, baseline=-0.5ex, scale=0.25]
\draw (0,0.1) -- (1.5,0.1);
\node [left] at (0,0) {$(T$};
\node [right] at (1.5,0) {$T)$};
\end{tikzpicture}
\in D_{2k}$
to $T\in L_{k+1}\otimes\Z/2\Z$ for a rooted tree $T$ of $\ideg=k$.
Thus, the connecting homomorphism, which is injective, is equal to $\nu$.

The homomorphism $\sl$ induces an isomorphism $\overline{\sl}\colon D_{2k}/\eta(\A_{2k,0}^c)\cong L_{k+1}\otimes\Z/2\Z$.
By the definition of $\kappa$,
we have 
\[
\overline{\sl}^{-1}\circ \kappa=\tau_{2k}\bmod \eta(\A_{2k,0}^c)\colon \Ker(\tau_{2k-1}\colon Y_{2k-1}\I\C/Y_{2k}\to D_{2k-1})\to D_{2k}/\eta(\A_{2k,0}^c).
\]
On the other hand, 
by \cite[Theorem~5.13]{Mas12},
we see that 
\[
\eta\circ \Ztilde^Y_{2k,0}\bmod\eta(\A_{2k,0}^c)=-\tau_{2k}^{\theta}\bmod\eta(\A_{2k,0}^c)\colon Y_{2k-1}\I\C/Y_{2k}\to (D_{2k}\otimes\Q)/\eta(\A_{2k,0}^c),
\]
where 
$\Ztilde^Y_{2k,0}$ is the (not necessarily connected) $0$-loop part of $\Ztilde^Y_{2k}$,
and $\tau_{2k}^\theta$ is the degree $2k$ part of the total Johnson map induced by the symplectic expansion $\theta$ which is denoted by $\theta^{\Ztilde}$ in \cite[Proposition~5.6]{Mas12}.
Thus, we obtain 
\[
\eta\circ\Ztilde^Y_{2k,0}\bmod\eta(\A_{2k,0}^c)=-\overline{\sl}^{-1}\circ \kappa\in (D_{2k}\otimes\Q)/\eta(\A_{2k,0})
\] on $\Ker\tau_{2k-1}$, and we see that $\bar{z}_{2k,0}=\nu\circ\kappa$ since $\overline{\sl}^{-1}=-\overline{\sl}^{-1}$ and $\nu$ is the connecting homomorphism.
\end{proof}

We also have some interesting relation between $\bar{z}_{2,0}$ and $\tau_1$.
Note that $D_1\cong \Lambda^3H$,
and the module $\C/Y_2$ is generated by $T(a,b,c)$ for $a,b,c\in\{1^{\pm},\ldots,g^{\pm}\}$  and $\K\C/Y_2$.
\begin{remark}\label{rem:BC-homo2}
The diagram
\[
\begin{CD}
\I\C/Y_2 @>\bar{z}_{2,0}>> \A_{2,0}^c\otimes\Q/\Z\\
@V(\tau_1\otimes\frac{1}{2})\circ \Gr qVV@VV\delta''V\\
\Lambda^3H\otimes \Q/\Z@>\cong>>\A_{3,1}^c\otimes \Q/\Z
\end{CD}
\]
commutes,
where the lower horizontal isomorphism will be explained in Section~\ref{section:proof-of-thmY3},
and the right vertical homomorphism is the induced one by $\delta''\colon\A_2^c\to \A_3^c$.
\end{remark}
We will see the image $\bar{z}_2(\K\C/Y_2)$ later in Lemma~\ref{lem:value-of-z2},
and it implies that $(\delta''\circ \bar{z}_{2,0})(\K\C/Y_2)=0\in\A_{3,1}^c\otimes \Q/\Z$.
Thus, Remark~\ref{rem:BC-homo2} follows from the equalities
\begin{align*}
(\delta''\circ \bar{z}_{2,0}\circ \ss)(T(a,b,c))
&=\frac{1}{2}(\delta''\circ \delta')(T(a,b,c))
=\frac{1}{2}O(a,b,c)\in\A_{3,1}^c\otimes \Q/\Z,\\
\Bigl(\Bigl(\tau_1\otimes\frac{1}{2}\Bigr)\circ \Gr q\circ \ss\Bigr)(T(a,b,c))
&=\frac{1}{2}((\id_H\otimes \gamma_2)\circ \eta')(T(a,b,c))
=\frac{1}{2}a\wedge b\wedge c\in\Lambda^3H\otimes \Q/\Z,
\end{align*}
which are deduced from Theorem~\ref{thm:main} and the commutative diagram~(\ref{diagram:AD}).

\subsection{The kernel of the surgery map on the $0$-loop part}\label{section:kernel-tree}
Here, using our homomorphism $\bar{z}_{n+1}$,
we investigate the kernel of the surgery map $\ss$.
More precisely, we show:
\begin{theorem}\label{thm:inj of tree part}
For $k\ge 1$, we have the following.
\begin{enumerate}
\interlinepenalty=10000
\item
$\Ker(\bar{z}_{2k}\circ \ss\colon\tor\A_{2k-1,2r}^c\to \A_{2k}^c\otimes\Q/\Z)
\supset \Im\Delta_{k-1,r}$ for $r\ge0$.\nopagebreak[2]
\item
$\Ker(\bar{z}_{2k}\circ \ss\colon\tor\A_{2k-1,0}^c\to \A_{2k}^c\otimes\Q/\Z)
=\Im\Delta_{k-1,0}$.
\end{enumerate}
\end{theorem}
Theorem~\ref{thm:inj of tree part} gives an information on the kernel of the surgery map $\ss\colon \A_n^c\to Y_n\I\C/Y_{n+1}$.
Massuyeau and Meilhan studied $Y_n\I\C/Y_{n+1}$ for $n=1,2$ in \cite{MaMe03} and \cite{MaMe13}.
They essentially showed that $\Ker(\ss|_{\A_1^c})=\Im\Delta_{0,0}$ in \cite[Theorem~1.3]{MaMe03}.
For example, $T(a,a,b)+T(a,b,b) \in \Ker\ss$ for $a,b \in \{1^{\pm},\ldots,g^{\pm}\}$ is deduced from the slide relation in \cite[Section~2.1]{MaMe03}.
They also showed that $\ss\colon\A_2^c\to Y_2\I\C/Y_3$ is an isomorphism in \cite[Corollary~5.1]{MaMe13}.
We determine $\Ker(\ss|_{\A_3^c})$ in Theorem~\ref{thm:str-of-y3}.
In \cite{CST16}, the kernel of the composition map of $\Gr q\circ\ss \colon \A_{n,0}^c\to Y_{n}\I\H/Y_{n+1}$ is also investigated.
\begin{theorem}[{\cite[Lemma~42 and Corollaries 50 and 51]{CST16}}]
Let $k\ge1$.
When $n=2k$,
the surjective homomorphism $\Gr q\circ\ss \colon \A_{n,0}^c\to Y_n\I\H/Y_{n+1}$ is an isomorphism.
When $n=2k-1$, $\Ker (\Gr q\circ\ss)\supset\Im\Delta_{k-1,0}$.
Moreover, when $n=4k-1$,
$\Ker (\Gr q\circ\ss)=\Im\Delta_{2k-1,0}$.
\end{theorem}

For $k\ge1$, let us define a map
$\xi\colon H^{\otimes (k+1)}\otimes\Z/2\Z\to \A_{2k,1}^c\otimes \Z/2\Z$
by 
\[
\xi(a_0\otimes a_1\otimes \cdots \otimes a_k)
=O(a_0, a_1,\ldots, a_{k-1},a_k,a_{k-1},\ldots, a_1).
\]
It factors through $(H^{\otimes (k+1)})_{\braket{y}}\otimes\Z/2\Z$,
since 
\[
O(a_0, a_1,\ldots, a_{k-1},a_k,a_{k-1},\ldots, a_1)
=O(a_k,a_{k-1},\ldots, a_1,a_0, a_1,\ldots, a_{k-1}).
\]
\begin{lemma}\label{lem:xi}
The homomorphism
\[
\delta''\circ \sq\colon H\otimes L_k\otimes\Z/2\Z
\to\A_{2k,1}^{c}\otimes\Z/2\Z.
\]
coincides with the map $\xi$ restricted to $H\otimes L_k\otimes\Z/2\Z$.
\end{lemma}
\begin{proof}
Pick $a_0,a_1,\ldots,a_k\in \{1^{\pm},\ldots,g^{\pm}\}$,
and let us denote $v_i=[a_i,[a_{i+1},[\cdots[a_{k-1},a_k]\cdots]]]\in L'_{k-i+1}$ for $1\le i\le k$ and $v_{k+1}=1\in\Z$.
A straightforward computation shows
\begin{align*}
&\quad(\delta''\circ \sq)(a_0\otimes [a_1,[a_2,[\cdots[a_{k-1},a_k]])\\
&=\delta''(T(a_k,a_{k-1},\ldots,a_1,a_0,a_1,\ldots,a_{k-1},a_k))\\
&=O(a_1,v_2,a_0,v_2)+\sum_{i=2}^{k}O(a_i,v_{i+1},a_{i-1},\ldots,a_1,a_0,a_1,\ldots,a_{i-1},v_{i+1})\\
&=\sum_{i=1}^{k}\xi(a_0\otimes a_1\otimes \cdots \otimes a_{i-1}\otimes v_{i+1}\otimes a_i)\\
&=\xi(a_0\otimes [a_1,[a_2,[\cdots[a_{k-1},a_k]\cdots]]]),
\end{align*}
where we slightly abuse the notation $O(a_i,v_{i+1},a_{i-1}\ldots,a_1,a_0,a_1,\ldots,a_{i-1},v_{i+1})$
which is a 1-loop Jacobi diagram whose ends of two edges are not univalent vertices but attached to the two copies of the tree $v_{i+1}$.
Since the degree $k$ part of the free Lie algebra $L'_k$ is generated by elements $[a_1,[a_2,[\cdots[a_{k-1},a_k]\cdots]]]$,
we obtain $\delta''\circ \sq=\xi$.
\end{proof}

As explained in \cite[Section~5.4]{CST12W},
the bracket map 
\[
[\ ,\ ]\colon H\otimes L_k'\otimes \Z/2\Z\to L'_{k+1}\otimes\Z/2\Z
\]
factors through $H\otimes L_k\otimes\Z/2\Z$.
If we take the quotient of the target of the isomorphism
\[
\eta'^{-1}\circ\sq\colon (H\otimes L_k)\otimes\Z/2\Z\to \tor\A_{2k-1,0}^c,
\]
we have a homomorphism $(H\otimes L_k)\otimes\Z/2\Z\to (\tor\A_{2k-1,0}^c)/\Im\Delta_{k-1,0}$.
It factors through the bracket map $[\ ,\ ]\colon H\otimes L_k\otimes \Z/2\Z\to 
L'_{k+1}\otimes\Z/2\Z$,
and induces an isomorphism
\[
\overline{\sq}\colon L'_{k+1}\otimes\Z/2\Z\to (\tor\A_{2k-1,0}^c)/\Im\Delta_{{k-1},0}.
\]
In \cite[Section~3.8]{CST16},
it is shown that,
for $T'\in L'_k\otimes\Z/2\Z$ and $a\in H$,
\[
(\kappa\circ \Gr q\circ \ss\circ \overline{\sq})([a,T'])
=(\kappa\circ \Gr q\circ \ss)
\Biggl(\begin{tikzpicture}[baseline=1ex, scale=0.4, dash pattern={on 2pt off 1pt}]
\node [left] at (0,0) {$T'$};
\draw (0,0)--(2,0);
\draw (1,0)--(1,1.2);
\node [right] at (2,0) {$T'$};
\node [above] at (1,1.2) {$a$};
\end{tikzpicture}\Biggr)
=[a,T']\in L_{k+1}\otimes\Z/2\Z,
\]
where the map $\Gr q\circ \ss$ and the rooted tree $[a,T']$ in this paper 
are denoted by $\tilde{\theta}_{2k-1}$ and $J$ in \cite[Section~3.8]{CST16}, respectively.
It implies the following.
\begin{lemma}[\cite{CST16}]\label{lem:kernelsatolevine}
The composite map
\[
\begin{CD}
L_{k+1}'\otimes\Z/2\Z@>\overline{\sq}>>(\tor\A_{2k-1,0}^c)/\Im\Delta_{{k-1},0}@>\kappa\circ \Gr q\circ\ss>> L_{k+1}\otimes\Z/2\Z
\end{CD}
\]
coincides with what is induced by the natural projection $\gamma_{k+1}\colon L_{k+1}'\to L_{k+1}$.
\end{lemma}
Note that the surgery map $\ss$ may not be well-defined on $(\tor\A_{2k-1,0}^c)/\Im\Delta_{{k-1},0}$.
However, $\Gr q\circ \ss\colon (\tor\A_{2k-1,0}^c)/\Im\Delta_{{k-1},0}\to \Ker(\tau_{2k-1}\colon Y_{2k-1}\I\H/Y_{2k}\to D_{2k-1})$ is well-defined by \cite[Lemma~42]{CST16}.
We are interested in the kernel of the composite map in Lemma~\ref{lem:kernelsatolevine},
which is non-trivial only when $k$ is odd,
and coincides with the image of $\theta_{\frac{k+1}{2}}$.
Setting $m=(k+1)/2$, we show the following.
\begin{lemma}\label{lem:periodicpart}
For $m\ge1$,
the image of the homomorphism
\[
\begin{CD}
L_m\otimes\Z/2\Z@>\theta_m>> L_{2m}'\otimes\Z/2\Z
@>\overline{\sq}>> (\tor\A_{4m-3,0}^c)/\Im\Delta_{2m-2,0}
@>\delta''>>\A_{4m-2,1}^c\otimes\Z/2\Z
\end{CD}
\]
is contained in $\A_{4m-2,1}^{c,s,\mathrm{period}}\otimes\Z/2\Z$.
Under the isomorphism
\[
\A_{4m-2,1}^{c,s,\mathrm{period}}\otimes\Z/2\Z
\cong \A_{2m-1,1}^{c,s}
\cong H^{\otimes m}\otimes\Z/2\Z
\]
which comes from Lemma~\ref{lem:oneloop-period} and Proposition~\ref{prop:oneloopbasis},
the above map $L_m\otimes\Z/2\Z\to H^{\otimes m}\otimes\Z/2\Z$
coincides with the natural inclusion.
\end{lemma}

\begin{proof}[Proof of Lemma~\ref{lem:periodicpart}]
Let $v=[a_1,[a_2,[\cdots,[a_{m-1},a_m]\cdots]]]\in L'_m\otimes\Z/2\Z$.
For $1\le i\le m$,
we also denote $v_i=[a_i,[a_{i+1},[\cdots,[a_{m-1},a_m]\cdots]]]\in L'_{m-i+1}\otimes\Z/2\Z$.
By the IHX relations or the Jacobi identity, we have
\[
[v,v]
=\sum_{i=1}^{m-1}[a_i,[v_{i+1},[a_{i-1},[\cdots,[a_2,[a_1,v]]\cdots]]]]\in L'_{2m}\otimes\Z/2\Z.
\]
Thus, we obtain by Lemma~\ref{lem:xi}:
\begin{align}
(\delta''\circ\overline{\sq}\circ\theta_m)(v)
&=(\delta''\circ\overline{\sq})([v,v])\notag\\
&=(\delta''\circ \sq)\left(\sum_{i=1}^{m}a_i\otimes[v_{i+1},[a_{i-1},[\cdots,[a_2,[a_1,v]]\cdots]]]\right)\notag\\
&=\xi(a_m\otimes[a_{m-1},[\cdots,[a_2,[a_1,v]]\cdots]])
+\sum_{i=1}^{m-1}\xi(a_i\otimes[v_{i+1},[a_{i-1},[\cdots,[a_2,[a_1,v]]\cdots]]]).\label{eq:sumvv1}
\end{align}
In the following, we denote the tensor product $a\otimes b$ simply as $ab$.
For $1\le i\le m-1$, we have
\begin{align}
&\quad\xi(a_i[v_{i+1},[a_{i-1},[\cdots,[a_2,[a_1,v]]\cdots]]])\notag\\
&=\xi(a_iv_{i+1}[a_{i-1},[\cdots,[a_2,[a_1,v]]\cdots]]+a_i[a_{i-1},[\cdots,[a_2,[a_1,v]]\cdots]]v_{i+1})\notag\\
&=\xi(v_i[a_{i-1},[\cdots,[a_2,[a_1,v]]\cdots]]
+v_{i+1}a_i[a_{i-1},[\cdots,[a_2,[a_1,v]]\cdots]]
+a_i[a_{i-1},[\cdots,[a_2,[a_1,v]]\cdots]]v_{i+1})\notag\\
&=\xi(v_i[a_{i-1},[\cdots,[a_2,[a_1,v]]\cdots]]
+v_{i+1}[a_i,[a_{i-1},[\cdots,[a_2,[a_1,v]]\cdots]]]).\label{eq:sumvv2}
\end{align}
Here, we use the fact that $\xi$ is invariant under the action of $y\in\D_{2m}$.
By Equations~(\ref{eq:sumvv1}) and (\ref{eq:sumvv2}),
we obtain $(\delta''\circ\overline{\sq}\circ\theta_m)(v)=\xi(vv)$.
Since
\[
v=[a_1,[a_2,[\cdots,[a_{m-1},a_m]\cdots]]]
=\sum_{i=1}^{m}a_iv_{i+1}a_{i-1}\cdots a_2a_1
=\sum_{i=1}^ma_1a_2\cdots a_{i-1}v_{i+1}a_i,
\]
we have
\[
\xi(vv)=\sum_{i=1}^{m}\xi(a_iv_{i+1}a_{i-1}\cdots a_2a_1^2a_2\cdots a_{i-1}v_{i+1}a_i).
\]
Each Jacobi diagram which appears in the right-hand side is periodic and invariant under the inversion $y$.
Thus, we have
\[
\xi(vv)\in \A_{4m-2,1}^{c,s,\mathrm{period}}\otimes\Z/2\Z.
\]
Under the isomorphism 
\[
\A_{4m-2,1}^{c,s,\mathrm{period}}\otimes\Z/2\Z
\cong H^{\otimes m}\otimes\Z/2\Z,
\]
given by Lemma~\ref{lem:oneloop-period} and Proposition~\ref{prop:oneloopbasis}, 
it corresponds to 
\[
\sum_{i=1}^{m}a_iv_{i+1}a_{i-1}\cdots a_2a_1
=[a_1,[a_2,[\cdots,[a_{m-1},a_m]\cdots]]]
=v\in L_m\otimes\Z/2\Z.
\]
\end{proof}

\begin{proof}[Proof of Theorem~\ref{thm:inj of tree part}]
(1)
The inclusion
$\Ker(\bar{z}_{2k}\circ \ss)\supset \Im\Delta_{k-1,r}$
is a corollary of Proposition~\ref{prop:Delta} and Theorem~\ref{thm:main}.

(2)
Since $\overline{\sq}\colon L_{k+1}'\otimes\Z/2\Z\to \tor\A_{2k-1,0}^c/\Im\Delta_{k-1,0}$ is an isomorphism as explained in Section~\ref{section:kernel-tree},
it suffices to show that the composite map
\[
\bar{z}_{2k}\circ\ss\circ \overline{\sq}\colon L_{k+1}'\otimes\Z/2\Z\to \A_{2k}^c\otimes\Q/\Z
\]
is injective, where the composite map is well-defined by Theorem~\ref{thm:inj of tree part}~(1).
By Proposition~\ref{prop:highersatolevine}, we have
\[
\bar{z}_{2k,0}\circ\ss\circ \overline{\sq}=\nu\circ\kappa\circ\Gr q\circ\ss\circ \overline{\sq},
\]
where $\bar{z}_{2k,0}$ is the $0$-loop part of $\bar{z}_{2k}$.
When $k$ is even, this is injective since $\nu$ is injective by Proposition~\ref{prop:highersatolevine}, and $\kappa\circ\Gr q\circ\ss\circ \overline{\sq}$ is injective by Lemma~\ref{lem:kernelsatolevine} and \cite[Lemma~2.1]{Lev02}.
When $k$ is odd,
$\Ker(\bar{z}_{2k,0}\circ\ss\circ \overline{\sq})=\Im\theta_m$ by Lemma~\ref{lem:kernelsatolevine} and the exact sequence~(\ref{eq:quasi-lie}), where we set $m=(k+1)/2$.
Thus, it suffices to show that the homomorphism
\[
\bar{z}_{4m-2,1}\circ \ss\circ \overline{\sq}\circ\theta_m
=\Bigl(\id\otimes\frac{1}{2}\Bigr)\circ \delta''\circ \overline{\sq}\circ\theta_m\colon
L_m\otimes\Z/2\Z\to \A_{4m-2,1}^c\otimes\Q/\Z
\]
is injective,
where $\bar{z}_{4m-2,1}$ is the $1$-loop part of $\bar{z}_{4m-2}$.
It follows from the facts that 
\[
\delta''\circ \overline{\sq}\circ\theta_m\colon L_m\otimes\Z/2\Z\to \A_{4m-2,1}^c\otimes\Z/2\Z
\]
is injective by Lemma~\ref{lem:periodicpart}
and that $\id\otimes\frac{1}{2}\colon \A_{4m-2,1}^c\otimes\Z/2\Z\to \A_{4m-2,1}^c\otimes\Q/\Z$ is also injective since $\A_{4m-2,1}^c$ is torsion-free.
\end{proof}

\section{Proofs of Theorems \ref{thm:restriction-torelli}, \ref{thm:JohnsonKernel}, and \ref{thm:str-of-y3}}\label{section:applications}
In this section, 
we prove Theorems~\ref{thm:restriction-torelli}, \ref{thm:JohnsonKernel}, and \ref{thm:str-of-y3}.
Recall that Theorems~\ref{thm:restriction-torelli} and \ref{thm:JohnsonKernel} state that our homomorphism
$\bar{z}_{n+1}\colon Y_{\lfloor\frac{n}{2}\rfloor+1}\I\C\to \A_{n+1}^c\otimes\Q/\Z$
restricted to subgroups of the Torelli group $\I$ gives new abelian quotients,
and Theorem~\ref{thm:str-of-y3} determines the module structure of $Y_3\I\C/Y_4$.

\subsection{Proof of Theorem~\ref{thm:restriction-torelli}} \label{section:proof-mainthm}
Here, we prove Theorem~\ref{thm:restriction-torelli} which shows that the image of the natural homomorphism
$\I(2n-1)/\I(2n)\to Y_{2n-1}\I\C/Y_{2n}$
has many torsion elements of order $2$ whose images under 
$\bar{z}_{2n}\colon Y_{2n-1}\I\C/Y_{2n}\to\A_{2n}^c\otimes\Q/\Z$
are non-trivial.

To prove Theorem~\ref{thm:restriction-torelli},
we recall three Lie algebras $\A^c$,
$\Gr^Y\I\C=\bigoplus_{n=1}^\infty Y_n\I\C/Y_{n+1}$,
and $\Gr\I=\bigoplus_{n=1}^\infty \I(n)/\I(n+1)$.
The module $\A^c$ of connected Jacobi diagrams becomes a Lie algebra under the commutator bracket induced by the associated algebra $\A^Y$,
and the surgery map $\ss\colon \A^c\to \Gr^Y\I\C$ is a Lie algebra homomorphism as explained in \cite[Section~8.5]{Hab00C}.
See also \cite[Section~1.2 and Section~2.2]{HaMa09} and \cite[Section~8.4]{CHM08}.
Especially, $\ss\colon \A^c\to \Gr^Y\I\C$ sends the Lie subalgebra of $\A^c$ generated by the $\ideg=1$ part $\A^c_1$ to the Lie subalgebra of $\Gr^Y\I\C$  generated by $\I\C/Y_2$.
The latter one is equal to the image of the Lie algebra homomorphism
$\Gr \cc\colon\Gr\I\to \Gr^Y\I\C$ induced by the inclusion $\cc\colon \I\to\I\C$
since $\cc$ induces an isomorphism $\I/\I(2)\cong \I\C/Y_2$ as shown in \cite[Theorem~1.3]{MaMe03}.
As a conclusion, $\ss$ sends the Lie subalgebra of $\A^c$ generated by $\A^c_1$ to $\Im(\Gr\cc)$.

We construct Jacobi diagrams of odd internal degrees and of order $2$ in the Lie subalgebra generated by $\A^c_1$ to prove Theorem~\ref{thm:restriction-torelli} via combinatorial calculus.
\begin{lemma}~\label{lem:commutator}
Let $2\le n\le g-2$ and $\{a_1,a_2,\ldots, a_n,a_{n+1},a_{n+2}\}\subset \{1^{\pm},2^{\pm},\ldots,g^{\pm}\}$.
If $a_i^*\ne a_j$ and $a_i\ne a_j$ for $1\le i< j \le n+2$,
the elements
\begin{align*}
&T(a_{n+2},a_{n-1},a_{n-2},\ldots,a_2,a_1,a_2,\ldots,a_{n-1},a_n),\\
&T(a_{n+2},a_{n-1},a_{n-2},\ldots,a_2,a_1,a_2,\ldots,a_{n-1},a_n,a_{n+1}),
\end{align*}
are contained in the Lie subalgebra generated by $\A_1^c$.
\end{lemma}
\begin{proof}
For simplicity, we assume $a_i\in\{1^-,2^-,\ldots,g^-\}$ for $1\le i\le n+2$.
Apparently, $T(a_4,a_1,a_2)\in\A_1^c$,
and $T(a_4,a_1,a_2,a_3)=[T(a_4,a_1,a_1^*),T(a_1,a_2,a_3)]$ is a commutator of $\A_1^c$.

Assume that 
\[
T(a_{n+1},a_{n-2},\ldots,a_2,a_1,a_2,\ldots,a_{n-1})\text{ and }
T(a_{n+1},a_{n-2},\ldots,a_2,a_1,a_2,\ldots,a_{n-1},a_n)
\]
are contained in the Lie subalgebra generated by $\A_1^c$.
Then, we have
\begin{align*}
&[T(a_{n+2},a_{n-1},a_{n+1}^*),T(a_{n+1},a_{n-2},\ldots,a_2,a_1,a_2,\ldots,a_{n-1},a_n)]\\
&=T(a_{n+2},a_{n-1},a_{n-2},\ldots,a_2,a_1,a_2,\ldots,a_{n-1},a_n),\\
&[T(a_{n+2},a_{n-1},a_{n-2},\ldots,a_2,a_1,a_2,\ldots,a_{n-1},a_n),T(a_n^*,a_n,a_{n+1})]\\
&=T(a_{n+2},a_{n-1},a_{n-2},\ldots,a_2,a_1,a_2,\ldots,a_{n-1},a_n,a_{n+1}).
\end{align*}
We see inductively these elements are contained in the Lie subalgebra generated by $\A_1^c$.
\end{proof}

\begin{lemma}\label{lem:oneloopproduct}
Let $1\le n\le g-2$ and $\{a_1,a_2,\ldots, a_n, a_{n+1}\}\subset \{1^{\pm},2^{\pm},\ldots,g^{\pm}\}$.
If $a_i^*\ne a_j$ and $a_i\ne a_j$ for $1\le i< j \le n+1$,
the elements
\[
\ss(T(a_{n+1},a_n,\ldots,a_2,a_1,a_2,\ldots,a_n,a_{n+1})),\
\ss(O(a_n,a_{n-1},\ldots,a_2,a_1,a_2,\ldots,a_{n-1},a_n))
\]
are contained in
\[
\Im\left(\I(2n-1)\to Y_{2n-1}\I\C/Y_{2n}\right).
\]
\end{lemma}
\begin{proof}
Choose $a_{n+2}$ so that $a_i$'s satisfy the condition in Lemma~\ref{lem:commutator}.
For simplicity, we assume $a_i\in\{1^-,2^-,\ldots,g^-\}$ for $1\le i\le n+2$.
When $n\ge2$,
we have
\begin{align*}
&[T(a_{n+1},a_n,a_{n+2}^*),
T(a_{n+2},a_{n-1},a_{n-2},\ldots,a_2,a_1,a_2,\ldots,a_n,a_{n+1})] \\
&=T(a_{n+1},a_n,a_{n-1},\ldots,a_2,a_1,a_2,\ldots,a_n,a_{n+1}).
\end{align*}
By Lemma~\ref{lem:commutator}, it is contained in the Lie subalgebra generated by $\A_1^c$,
This is also true for $n=1$ since $T(a_2,a_1,a_2)\in\A_1^c$.
We also have
\begin{align*}
&[T(a_n,a_{n+1}^*,a_{n+2}^*), T(a_{n+2},a_{n-1},\ldots,a_2,a_1,a_2,\ldots,a_{n-1},a_n,a_{n+1})]\\
&=-T(a_{n+1}^*,a_n,a_{n-1},\ldots,a_2,a_1,a_2,\ldots,a_{n-1},a_n,a_{n+1})\\
&\quad-T(a_{n+2},a_{n-1},\ldots,a_2,a_1,a_2,\ldots,a_{n-1},a_n,a_n,a_{n+2}^*)\\
&\quad+O(a_n,a_{n-1},\ldots,a_2,a_1,a_2,\ldots,a_{n-1},a_n),\\
&[T(a_n,a_{n+1},a_{n+2}^*), T(a_{n+2},a_{n-1},\ldots,a_2,a_1,a_2,\ldots,a_{n-1},a_n,a_{n+1}^*)]\\
&=-T(a_{n+1},a_n,a_{n-1},\ldots,a_2,a_1,a_2,\ldots,a_{n-1},a_n,a_{n+1}^*)\\
&\quad+T(a_{n+2},a_{n-1},\ldots,a_2,a_1,a_2,\ldots,a_{n-1},a_n,a_n,a_{n+2}^*)
\end{align*}
for $n\ge2$.
The sum of these elements is equal to $O(a_n,a_{n-1},\ldots,a_2,a_1,a_2,\ldots,a_{n-1},a_n)$,
and it is also contained in the Lie subalgebra generated by $\A_1^c$.
When $n=1$, $O(a_1)$ is zero by the self-loop relation.

As we explained in the beginning of this section,
the images of the two elements above under the surgery map $\ss$ is contained in
$\Im(\I(2n-1)\to Y_{2n-1}\I\C/Y_{2n})$.
\end{proof}

\begin{proof}[Proof of Theorem~\ref{thm:restriction-torelli}]
Let $n\ge2$.
By the AS relation, we see that the elements
\[
T(a_{n+1},a_n,\ldots,a_2,a_1,a_2,\ldots,a_n,a_{n+1}),\
O(a_n,a_{n-1},\ldots,a_2,a_1,a_2,\ldots,a_{n-1},a_n)
\]
are of order 2.
Thus, it suffices to check that the two elements in $Y_{2n-1}\I\C/Y_{2n}$ obtained in Lemma~\ref{lem:oneloopproduct} are non-trivial.
By Theorem~\ref{thm:main}, we have
\begin{align*}
(\bar{z}_{2n,0}\circ\ss)(T(a_{n+1},a_n,\ldots,a_2,a_1,a_2,\ldots,a_n,a_{n+1}))&=\frac{1}{2}T(a_{n+1},\ldots,a_2,a_1,a_1,a_2,\ldots,a_{n+1}),\\
(\bar{z}_{2n,1}\circ\ss)(O(a_n,a_{n-1},\ldots,a_2,a_1,a_2,\ldots,a_{n-1},a_n))&=\frac{1}{2}O(a_n,\ldots,a_2,a_1,a_1,a_2,\ldots,a_n).
\end{align*}
The first one is equal to $\nu([a_1,[a_2,[\cdots,[a_{n-1},a_n]\cdots]]])\in A_{2n,0}^c\otimes\Q/\Z$,
and it is non-trivial since $\nu$ is injective.
The second one is also non-trivial since $O(a_n,\ldots,a_2,a_1,a_1,a_2,\ldots,a_n)$ corresponds to an element of  $W_{2n}/\D_{4n}$, which is the basis of the free module $\A_{2n,1}^c$ given in the proof of Proposition~\ref{prop:oneloopbasis}.
Thus, the two elements in $Y_{2n-1}\I\C/Y_{2n}$ in Lemma~\ref{lem:oneloopproduct} are also non-trivial.

When $n=1$,
\[
(\bar{z}_{2,0}\circ\ss)(T(a_2,a_1,a_2))=\frac{1}{2}T(a_2,a_1,a_1,a_2)\text{ and }
(\bar{z}_{2,1}\circ\ss)(T(a_2,a_1,a_2))=\frac{1}{2}O(a_1,a_2)
\]
are also non-trivial.
\end{proof}

Morita's refinement $\tilde{\tau}_n$ extends to a homomorphism on $Y_n\I\C$ as explained in \cite{Mas12}.
Since the target of $\tilde{\tau}_n$ is a free module,
the elements $\ss(T(a_n,\ldots,a_2,a_1,a_2,\ldots,a_n))$ and $\ss(O(a_n,\ldots,a_2,a_1,a_2,\ldots,a_n))$ are in the kernel.
Thus, our homomorphism $\bar{z}_{2n}$ is non-trivial on $\Im(\I(n)\to Y_n\I\C)\cap\Ker\tilde{\tau}_n$,
and Corollary~\ref{cor:kerneljohnson} follows.

The $Y_{2n}$-equivalence class $\ss(O(a_n,\ldots,a_2,a_1,a_2,\ldots,a_n))$ obtained from the graph clasper with 1 loop is trivial on $Y_{2n-1}\I\H/Y_{2n}$ by \cite[Theorem~2]{Lev01}.
Thus, we also have the following.
\begin{cor}\label{cor:cobordism}
For $2\le n\le g-2$,
the $1$-loop part of $\bar{z}_{2n}\colon Y_{2n-1}\I\C\to \A_{2n}^c\otimes\Q/\Z$ restricted to $\I(2n-1)$ does not factor through $Y_{2n-1}\I\H/Y_{2n}$.
\end{cor}

\subsection{Proof of Theorem~\ref{thm:JohnsonKernel}}
\label{sec:JohnsonKernel}
Here, we prove Theorem~\ref{thm:JohnsonKernel} which states that $\bar{z}_4$ gives a homomorphism on $\K\C$ whose restriction to the intersection of $\K$ and the kernels of Morita's refinement $\tilde{\tau}_2$ and the Casson invariant is non-trivial.

In Lemma~\ref{lem:oneloopproduct},
we construct an element
\[
\ss(T(a_3,a_2,a_1,a_2,a_3))\in \Im(\I(3)\to Y_3\I\C/Y_4)
\]
of order $2$ whose image under $\bar{z}_{4,0}$ is non-trivial.
The homomorphisms $\tilde{\tau}_2$ and the Casson invariant extend to homomorphisms on $\K\C$ as explained in \cite{Mas12} and \cite[Section~8.6]{CHM08}, respectively,
and the element is in the kernels of these homomorphisms because their targets are free modules.

Thus, $\bar{z}_{4,0}$ is non-trivial in the intersection of $\I(3)$ and the kernels of these homomorphisms,
and Theorem~\ref{thm:JohnsonKernel} follows from the following.
\begin{lemma}\label{lem:barz-homo}
The map
\[
\bar{z}_4=(\log\Ztilde^Y)_4\bmod\Z\colon\K\C \to \frac{\A_4^c\otimes\Q/\Z}{\braket{\frac{1}{2}\delta O(a,b,b) \mid a,b \in \{1^{\pm},2^{\pm},\ldots, g^{\pm}\}}},
\]
is a homomorphism.
\end{lemma}
To prove Lemma~\ref{lem:barz-homo},
we compute the image of
$\bar{z}_2\colon \I\C/Y_2\to \A_2^c\otimes\Q/\Z$ restricted to $\K\C/Y_2$.
\begin{lemma}\label{lem:value-of-z2}
The image $\bar{z}_2(\K\C/Y_2)$ in $\A_2^c\otimes\Q/\Z$ is generated by
\[
\frac{1}{2}T(a,b,b,a)+\frac{1}{2}O(a,b)\text{ and }
\frac{1}{2}\,
\begin{tikzpicture}[baseline=-0.5ex, scale=0.25, dash pattern={on 2pt off 1pt}]
\draw (0,0) circle [x radius=2, y radius=1]; 
\draw (90:1) -- (-90:1);
\end{tikzpicture}
\]
for $a,b\in\{1^{\pm}, 2^{\pm}, \ldots, g^{\pm}\}$.
\end{lemma}
\begin{proof}
Since $\bar{z}_2=(\log \Ztilde^Y)_2 \bmod \Z = \Ztilde^Y_2 \bmod \Z$ on $\K\C$,
this follows from the commutative diagram just after Claim~5.7 in \cite{MaMe13}
under the isomorphism \cite[Lemma~8.4]{CHM08}.
\end{proof}

\begin{proof}[Proof of Lemma~\ref{lem:barz-homo}]
Recall that the leading term $(\log\Ztilde^Y)_1\colon \I\C\to A_1^c\otimes \Q$ equals the first Johnson homomorphism with opposite sign.
Thus, for $M$, $N\in\K\C$, we have $\Ztilde^Y(M)-\emptyset$, $\Ztilde^Y(N)-\emptyset\in\A_{\ge2}^Y\otimes \Q$.
Since
\[
\Ztilde^Y(M\circ N)
=\Ztilde^Y(M)\star \Ztilde^Y(N)
=(\Ztilde^Y(M)-\emptyset)\star (\Ztilde^Y(N)-\emptyset)
+\Ztilde^Y(M)+\Ztilde^Y(N)-\emptyset,
\]
we have
\[
\Ztilde_4^Y(M\circ N)-\Ztilde_4^Y(M)-\Ztilde_4^Y(N)
=\Ztilde_2^Y(M)\star \Ztilde_2^Y(N).
\]
Thus, it suffices to show that the connected part of $\Ztilde_2^Y(M)\star \Ztilde_2^Y(N)$ is in
\[
\Im \iota+\left\langle\frac{1}{2}\delta O(a,b,b) \biggm| a,b\in\{1^{\pm},2^{\pm},\ldots, g^{\pm}\}\right\rangle.
\]
Let us denote
\[
w(i^{\epsilon},j^{\epsilon'})=
\begin{cases}
1&\text{if }i=j,\ \epsilon=+\text{ and }\epsilon'=-,\\
0&\text{otherwise}
\end{cases}
\]
for $1\le i\le g$, $1\le j\le g$ and $\epsilon,\epsilon'\in\{+,-\}$, and
$\Theta(a_1,a_2)=
\begin{tikzpicture}[baseline=-0.4ex, scale=0.25, dash pattern={on 2pt off 1pt}]
\draw (0,0) circle [x radius=2, y radius=1]; 
\draw (0:2) -- (0:4);
\draw (180:2) -- (180:4);
\draw (90:1) -- (-90:1);
\node [right] at (0:3.8) {$a_2$};
\node [left] at (180:3.8) {$a_1$};
\end{tikzpicture}$
for $a_1,a_2\in\{1^{\pm},\ldots,g^{\pm}\}$.
We have
\begin{align*}
&\quad\frac{1}{2}\{T(a,b,b,a)+O(a,b)\}
\}\star\frac{1}{2}\{T(c,d,d,c)+O(c,d)\}\\
&=\frac{1}{2}\{w(a,c)(O(b,b,d,d)+\Theta(b,d))
+w(a,d)(O(b,b,c,c)+\Theta(b,c))\\
&\quad+w(b,c)(O(a,a,d,d)+\Theta(a,d))
+w(b,d)(O(a,a,c,c)+\Theta(a,c))\}\\
&\quad +(\textrm{disconnected part})\in\A_4^Y\otimes\Q/\Z.
\end{align*}
By Lemma~\ref{lem:value-of-z2}, 
the connected part of $\Ztilde_2^Y(M)\star \Ztilde_2^Y(N)$
is in $\Im \iota+\braket{\frac{1}{2}\delta O(a,b,b)\mid a,b\in\{1^{\pm},2^{\pm},\ldots, g^{\pm}\}}$.
\end{proof}

\subsection{The case of closed surfaces}
We have so far considered surfaces with one boundary component and homology cylinders over such surfaces.
This subsection is devoted to discussing the case of closed surfaces.
The goal is to prove that the abelianization of the Johnson kernel has torsion elements.

Let us denote the Torelli group, the monoid of homology cylinder, and its Johnson kernel in the case of closed surface $\Sigma_g$ of genus $g$ by $\I_g$, $\I\C_g$, and $\K\C_g$, respectively.
It follows from \cite{HaMa12} that $\Ztilde^Y\colon \I\C \to \widehat{\A}\otimes\Q$ induces a homomorphism $\I\C_g \to (\widehat{\A}\otimes\Q)/(s\circ\varphi)^{-1}(I^{<})$ (also denoted by $\Ztilde^Y$) satisfying the commutative diagram
\[
\xymatrix{
\I\C \ar@{->>}[d] \ar[r]^-{\Ztilde^Y}&\widehat{\A}\otimes\Q \ar@{->>}[d]\\
\I\C_g \ar[r]^-{\Ztilde^Y}& (\widehat{\A}\otimes\Q)/(s\circ\varphi)^{-1}(I^{<}),
}
\]
where $I^{<}$ is an ideal introduced in \cite[Section~7.1]{HaMa09} and $\varphi$ and $s$ are isomorphisms used in \cite[Section~7.4]{HaMa09}.

Let $\overline{I}_n^c$ denote the image of $I_n=(s\circ\varphi)^{-1}(I^{<})\cap(\A_n\otimes\Q)$ under the projection $\widehat{\A}\otimes\Q \twoheadrightarrow \widehat{\A}^c\otimes\Q/\Z$ to the connected part.
We define the homomorphism $\wideparen{z}_{n+1}\colon Y_n\I\C_g/Y_{n+1} \to (\A_{n+1}^c\otimes\Q/\Z)/\overline{I}_{n+1}^c$ by $\wideparen{z}_{n+1}([M]) = (\log\Ztilde^Y(M))_{n+1}$.
By Theorem~\ref{thm:main} and the definitions of $\bar{z}_{n+1}$ and $\wideparen{z}_{n+1}$, the diagram
\[
\xymatrix{
\A^c_n \ar[r]^-{\ss} \ar[d]^-{\delta} & Y_n\I\C/Y_{n+1} \ar@{->>}[r] \ar[d]^-{\bar{z}_{n+1}} & Y_n\I\C_g/Y_{n+1} \ar[d]^-{\wideparen{z}_{n+1}} \\
\A_{n+1}^c\otimes\Z/2\Z \ar[r]^-{\id\otimes\frac{1}{2}} & \A_{n+1}^c\otimes\Q/\Z \ar@{->>}[r] & (\A_{n+1}^c\otimes\Q/\Z)/\overline{I}_{n+1}^c
}
\]
is commutative.
Furthermore, the same argument in the proof of Theorem~\ref{thm:Z-homo} shows that the homomorphism $\bar{z}_{n+1}\colon Y_n\I\C_{g}/Y_{n+1}\to (\A_{n+1}^c\otimes\Q/\Z)/\overline{I}_{n+1}^c$ extends to a homomorphism on $Y_{\lfloor \frac{n}{2}\rfloor+1}\I\C_{g}/Y_{n+1}$, which is proved in the same way as Theorem~\ref{thm:Z-homo}.

\begin{proof}[Proof of Corollary~\ref{cor:H1Kg}]
By the same argument as the proof of Theorem~\ref{thm:JohnsonKernel} in Section~\ref{sec:JohnsonKernel}, the composite map
\[
\I_{g}(2) \xrightarrow{\cc} Y_2\I\C_g/Y_4 \xrightarrow{\wideparen{z}_4} (\A_{4}^c\otimes\Q/\Z)/ \left( \overline{I}_{4}^c + \langle\tfrac{1}{2}\delta O(a,b,b) \mid a,b \in \{1^{\pm},\dots, g^{\pm}\}\rangle \right)
\]
extends to $\K_{g}$ (and $\K\C_g$).

We next focus on the $0$-loop part of this map and show its non-triviality.
Let $\overline{I}_{4,0}^c$ be the image of $\overline{I}_{4}^c$ under the projection $\A_{4}^c\otimes\Q/\Z \twoheadrightarrow \A_{4,0}^c\otimes\Q/\Z$, and consider the composite map
\[
\K_{g} \to (\A_{4}^c\otimes\Q/\Z)/ \left( \overline{I}_{4}^c + \langle\tfrac{1}{2}\delta O(a,b,b) \mid a,b \in \{1^{\pm},\dots, g^{\pm}\}\rangle \right) \twoheadrightarrow (\A_{4,0}^c\otimes\Q/\Z)/\overline{I}_{4,0}^c.
\]
Here we see that $\overline{I}_{4,0}$ is generated by $\left\{ \sum_{i=1}^{g} T(i^+,i^-,a_1,a_2,a_3,a_4) \mid a_j \in \{1^\pm,\dots,g^\pm\} \right\}$ over $\Q/\Z$.
Hence, if $a_j,a_j^\ast$ ($j=1,2,3$) are distinct labels, then $\frac{1}{2}\delta(T(a_3,a_2,a_1,a_2,a_3)) \notin \overline{I}_{4}^c$.
It follows that, if we restrict the above map to the intersection of the kernels of Morita's refinement of the second Johnson homomorphism and the Casson invariant, it is still non-trivial.

On the other hand, $H_1(\K_g;\Q)$ is completely determined in \cite[Theorem~1.3]{MSS20} and described by Morita's refinement of the second Johnson homomorphism and the Casson invariant.
Therefore, $H_1(\K_g;\Z)$ must have torsion elements.
\end{proof}

\subsection{Proof of Theorem~\ref{thm:str-of-y3}}
\label{section:proof-of-thmY3}
Here, we prove Theorem~\ref{thm:str-of-y3} which determines the module structure of the graded quotient $Y_3\I\C/Y_4$.
We need the following lemma to prove Theorem~\ref{thm:str-of-y3}.
\begin{lemma}\label{lem:kernel of surgery map}
For $a,b,c\in\{1^{\pm},\ldots,g^{\pm}\}$,
\begin{enumerate}
\item
$\ss(O(a,a,b))=\ss(O(a,b,b))\in Y_3\I\C/Y_4$.
\item
$\ss(\Delta_{1,0}(T(a,b,c)))=0\in Y_3\I\C/Y_4$.
\end{enumerate}
\end{lemma}

\begin{proof}[Proof of Theorem~\ref{thm:str-of-y3}]
Recall that the surgery map $\ss\colon \A_3^c\to Y_3\I\C/Y_4$ is surjective
and that the composite map
$\Ztilde^Y_3\circ \ss\colon \A_3^c\to\A_3^c\otimes\Q$ coincides with $\iota$ up to sign.
Thus, we have an exact sequence
\[
\begin{CD}
0@>>>\Ker\ss@>>>\tor\A_3^c@>>>\tor(Y_3\I\C/Y_4)@>>>0.
\end{CD}
\]
Each $2$-loop Jacobi diagram of $\ideg=3$ has only one univalent vertex,
and we have $\A^c_{3,2}=0$ by \cite[Proof of Lemma~5.30]{CDM12} and the self-loop relation.
Thus, as we saw in Proposition~\ref{prop:oneloopbasis} and Theorem~\ref{thm:treepart},
the module $\tor\A_3^c$ is generated by $T(c,b,a,b,c)$ and $O(a,b,b)$ for $a,b,c\in\{1^{\pm},2^{\pm},\dots,g^{\pm}\}$,
and it is isomorphic to $(H\otimes L_2\oplus H^{\otimes 2})\otimes \Z/2\Z$.

The submodules of $\A_3^c$ generated by $\Delta_{1,0}(T(a,b,c))$ and $O(a,a,b)+O(a,b,b)$ are isomorphic to $\Lambda^3H\otimes \Z/2\Z$ and $\Lambda^2H\otimes \Z/2\Z$, respectively.
By Lemma~\ref{lem:kernel of surgery map}, we have
\[
|\Ker \ss|\ge |(\Lambda^3H\oplus \Lambda^2H)\otimes\Z/2\Z|.
\]
On the other hand,
by Proposition~\ref{prop:highersatolevine},
we have 
\[
\bar{z}_{4,0}(\ss(\tor A^c_{3,0}))=\bar{z}_{4,0}(\tor(Y_3\I\C/Y_4))=(\nu\circ\kappa\circ\Gr q)(\tor(Y_3\I\C/Y_4)),
\]
and it is isomorphic to $L_3\otimes\Z/2\Z$ since $\nu$ is injective and $\kappa$ is surjective.
The value $\bar{z}_4(\ss(O(a,a,b)))=\frac{1}{2}O(a,a,b,b)+\frac{1}{2}\Theta(a,b)\ne0\in \A_4^c\otimes\Q/\Z$ shows that the image $\bar{z}_{4,1}(\ss(\tor\A^c_{3,1}))$ has order greater than or equal to that of $S^2H\otimes\Z/2\Z$.
Thus, we obtain
\[
|\tor(Y_3\I\C/Y_4)|\ge |(L_3\oplus S^2H)\otimes\Z/2\Z|.
\]
Comparing the orders of $\Ker \ss$,
$\tor(Y_3\I\C/Y_4)$, and $\tor\A_3^c$,
we see that $\Ker \ss$ is isomorphic to $(\Lambda^3H\oplus \Lambda^2H)\otimes\Z/2\Z$
and that $\bar{z}_4$ gives an isomorphism
\[
\tor(Y_3\I\C/Y_4) \cong  (L_3\oplus S^2H)\otimes\Z/2\Z.
\]

Next, we consider the free part of $Y_3\I\C/Y_4$.
We have isomorphisms $\A_{n,0}^c/\tor A_{n,0}^c\cong D'_n/\tor D'_n\cong D_n$ when $n$ is odd,
where the first isomorphism follows from \cite[Theorem~1]{Lev02} or \cite[Theorem~1.4]{CST12L}, and the second one follows from Theorem~\ref{thm:treepart} since $D_n$ is torsion-free.
On the other hand, the homomorphism $\Lambda^3H\to \A_{3,1}^c/\tor\A_{3,1}^c$ defined by $a\wedge b\wedge c\mapsto O(a,b,c)$ is well-defined and surjective.
Moreover, it is an isomorphism because their ranks coincide as we see from Proposition~\ref{prop:oneloopbasis}.
Thus, we obtain isomorphisms of abelian groups
\[
(Y_3\I\C/Y_4)/\tor(Y_3\I\C/Y_4)
\cong A_3^c/\tor A_3^c
\cong D_3\oplus \Lambda^3H,
\]
where the first isomorphism is induced by $\ss \colon \A_3^c\to Y_3\I\C/Y_4$ as in \cite[Theorem~8.8]{CHM08}.
\end{proof}

\begin{remark}
\label{rem:z4inj}
It follows from the proof of Theorem~\ref{thm:str-of-y3} that $\bar{z}_4|_\mathrm{tor} \colon \tor(Y_3\I\C/Y_4) \to \A_4^c\otimes\Q/\Z$ is injective.
Indeed, the inequalities $|\tor(Y_3\I\C/Y_4)| \geq |\Im \bar{z}_4| \geq |(L_3\oplus S^2H)\otimes\Z/2\Z|$ are equalities.
\end{remark}

To show Lemma~\ref{lem:kernel of surgery map},
we review some moves on graph claspers given in~\cite{Hab00C}, \cite{Gou99}, \cite{Gus00}, and \cite{GGP01},
and graph claspers with special leaves investigated in \cite{Mei06}, \cite{MaMe13}, and \cite{CST16}.
A leaf of a graph clasper in a 3-manifold $M$ is said to be \emph{special} if it bounds a disk in $M$ and it is $(-1)$-framed with respect to the disk.
Graph claspers with special leaves are called \emph{twisted claspers} in \cite{CST16}.
As shown in \cite[Lemma~E.21]{Oht02} and \cite[Lemma~4.9]{GGP01},
the surgery along a connected twisted clasper with $n$ nodes in a 3-manifold $M$ does not change its $Y_{n+1}$-equivalence class when $n\ge2$.

In the following, we use additional constituents of claspers which are introduced in \cite[Section~2.2]{Hab00C}.
\begin{figure}[h]
 \centering
 \includegraphics{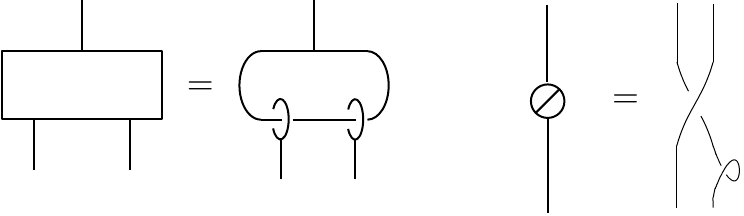}
\caption{A box and a twist of an edge.}
 \label{fig:box}
\end{figure}
The left-hand side in Figure~\ref{fig:box} of the first equality is called a \emph{box} which has three incident ends of edges.
The end of the edge in the upper side of the box is called the \emph{input end} and the other two ends in the lower side are called the \emph{output ends}.
The box denotes three leaves of claspers as in the right-hand side.
The left-hand side of the second equality denotes a positive half twist of an edge,
which is equal to the segment of the framed link depicted in the right-hand side.

We have several equivalences between graph claspers with special leaves which are analogous to those listed in \cite[Appendix E]{Oht02} and \cite[Appendix A]{MaMe13}.

\begin{lemma}\label{lem:crossingchange}
Let $G$ be a connected twisted clasper with $n$ nodes and no boxes in a 3-manifold $M$,
and let $K$ be a framed knot in $M$ disjoint from $G$.
Choose a band $b$ which connects an edge of $G$ and $K$ in $M$,
and denote by $G'$ a twisted clasper obtained from $G$ by taking a connecting sum of the edge with $K$ along $b$.
Then, we have
\[
M_{G}\sim_{Y_{n+2}} M_{G'}.
\]
\end{lemma}
\begin{proof}
This immediately follows from \cite[Lemma~A.1]{MaMe13}
since the surgery along a twisted clasper with $n+1$ nodes in a 3-manifold $M$ does not change its $Y_{n+2}$-equivalence class.
The case when $n=1$ is proved in \cite[Lemma~A.10]{MaMe13}.
\end{proof}

\begin{lemma}\label{lem:leafchange}
Let $G$ be a connected twisted clasper with $n$ nodes and no boxes in $M$ whose two leaves are locally described as in the left-hand side in the Figure~\ref{fig:leafchange},
where the thick lines imply the same bundle which may contain edges and leaves of claspers.
\begin{figure}[h]
 \centering
 \includegraphics{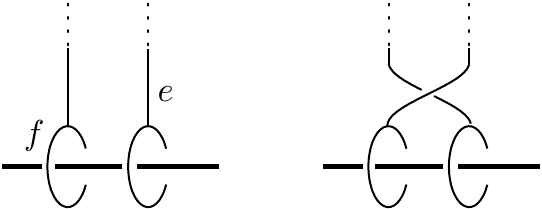}
 \caption{Swapping the positions of two leaves.}
 \label{fig:leafchange}
\end{figure}
Let $G'$ denote another connected twisted clasper obtained from $G$ by swapping the positions of the two leaves as in the right-hand side.
Then, we have
\[
G\sim_{Y_{n+2}}G'.
\]
\end{lemma}
\begin{proof}
This follows from Lemma~\ref{lem:crossingchange} by taking a meridian of the leaf $f$ as $K$ and a band which connects $e$ and $K$.
In other words, the proof is given by applying a crossing change to the edge $e$ and the leaf $f$ of $G$.
See, for example, figures in the proof of \cite[Lemma~E.6]{Oht02}.
\end{proof}

\begin{lemma}\label{lem:asrelation}
Let $G$ be a connected twisted clasper with $n$ nodes and no boxes in $\Sigma_{g,1}\times [-1,1]$,
and let $G'$ denote another connected twisted clasper obtained from $G$ by inserting a positive half twist in an edge $e$.
Then, we have an equivalence
\[
(\Sigma_{g,1}\times [-1,1])_G\sim  (\Sigma_{g,1}\times [-1,1])_{G'}
\]
if $e$ is incident to special leaves.
When $n\ge 2$,
we also have a $Y_{n+2}$-equivalence
\[
(\Sigma_{g,1}\times [-1,1])_G\circ (\Sigma_{g,1}\times [-1,1])_{G'}\sim_{Y_{n+2}} \Sigma_{g,1}\times [-1,1]
\]
if $e$ is not incident to special leaves.
\end{lemma}
\begin{proof}
If $e$ is incident to special leaves,
inserting a half twist to $e$ does not change the isotopy class of the twisted clasper $G$.

Next, assume that $e$ is not incident to special leaves.
In this case, the proof is almost the same as \cite[Lemma~A.15]{MaMe13}.
In the following, we consider only the case when the number of special leaves is $1$, and $e$ and the special leaf are not incident to the same node.
The other cases are proved in a similar way.
For example, consider the case when $e$ is an edge in a twisted clasper
\begin{center}
  \includegraphics{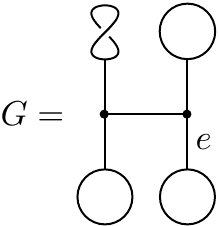}
\end{center}
with $2$ nodes.
We have the following equivalences:
\begin{center}
  \includegraphics{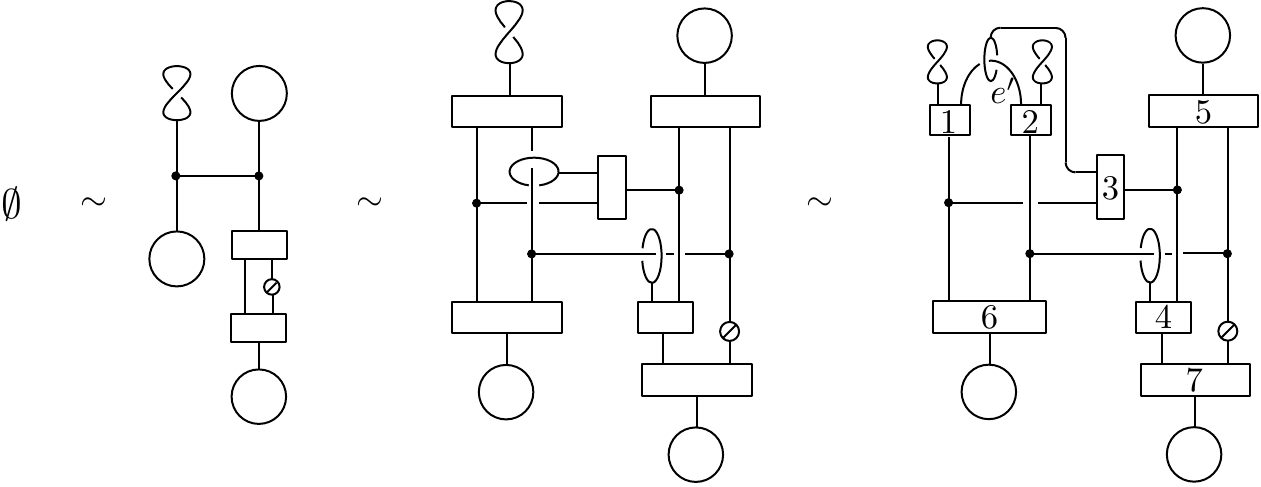}\ \raisebox{1.7cm}{.}
\end{center}
Here, the first equivalence means that the $3$-manifold obtained by the surgery along the clasper in the right-hand side is homeomorphic to $\Sigma_{g,1}\times[-1,1]$,
and comes from \cite[Move~4]{Hab00C}.
The second equivalence is given by applying the zip construction in \cite{Hab00C} with the marking being the right output of the upper box (in the general case, we apply the zip construction to the both ends of $e$).
The third equivalence is Move~(b) in \cite[Lemma~A.6]{Mei06},
and we numbered the boxes in the resulting clasper.
We denote by $e'$ the edge connecting the left output of Box~1 and the right output of Box~2.
After applying the zip construction with the marking being the both ends of $e'$,
the edge $e'$ becomes a connected graph clasper with $3$~nodes (in the general case, this graph clasper has $n+1$ nodes).
By a crossing change using \cite[Lemma~E.5]{Oht02},
we can make the leaf attached to the upper output of Box~3 bound a disk under the $Y_4$-equivalence.
Apply the inverse of the zip construction stated above,
and erase Box~3 by \cite[Move~3]{Hab00C}.
Then, we obtain the left clasper below.
\begin{center}
  \includegraphics{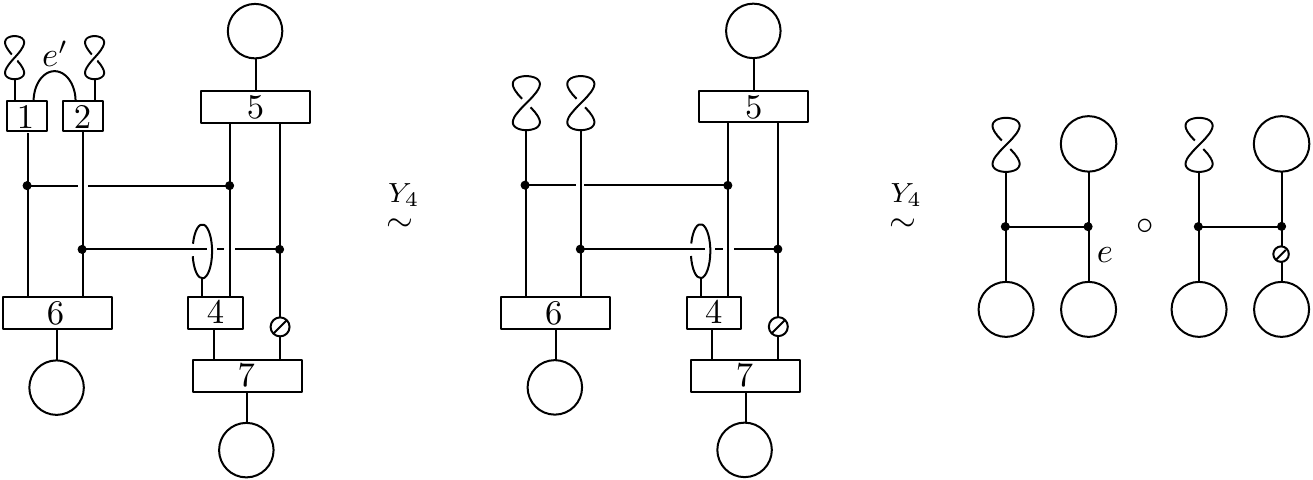}\ \raisebox{1.7cm}{.}
\end{center}
Applying the zip construction with the same marking again,
the edge $e'$ becomes a graph clasper with 4~nodes (in the general case, this clasper has $n+2$~nodes when $n\ge2$),
and we can erase it under the $Y_4$-equivalence.
Here, we use the assumption that the special leaf and $e$ are not incident to the same node.
When they are incident to the same node,
the sum of the number of nodes in two graph claspers which attached to the inputs of Boxes~1 and 2 in the left figure is $n+1$.
Even in this case,
when $n\ge2$,
we can erase several boxes in the same way as Box~$3$ under the $Y_{n+2}$-equivalence before applying the zip construction with the marking being the ends of $e'$, and increase the sum of the number of nodes up to $n+2$ or more.
Moreover, when the number of special leaves is not $1$,
we apply the same moves one by one in neighborhoods of special leaves,
and we can erase each edge which corresponds to $e'$ arising from Move~(b) in \cite[Lemma~A.6]{Mei06}.

After erasing boxes arising from the zip construction by \cite[Move~3]{Hab00C},
the rest is the clasper in the middle.
By a crossing change using Lemma~\ref{lem:crossingchange},
we can make the leaf attached to the left output of Box~4 bound a disk,
and can erase Box~4 by \cite[Move~3]{Hab00C}.
Applying \cite[Move~6]{Hab00C} to Boxes~5, 6, and 7 we can change each box into two leaves,
and we obtain two disjoint graph claspers.
By Lemma~\ref{lem:crossingchange}, we can move these claspers into disjoint horizontal layers $\Sigma_{g,1}\times[-1,0)$ and $\Sigma_{g,1}\times (0,1]$,
and we obtain $(\Sigma_{g,1}\times [-1,1])_G\circ (\Sigma_{g,1}\times [-1,1])_{G'}\sim_{Y_{n+2}} \Sigma_{g,1}\times [-1,1]$.
\end{proof}

\begin{proof}[Proof of Lemma~\ref{lem:kernel of surgery map}~(1)]
We have the following equivalences:
\begin{center}
  \includegraphics{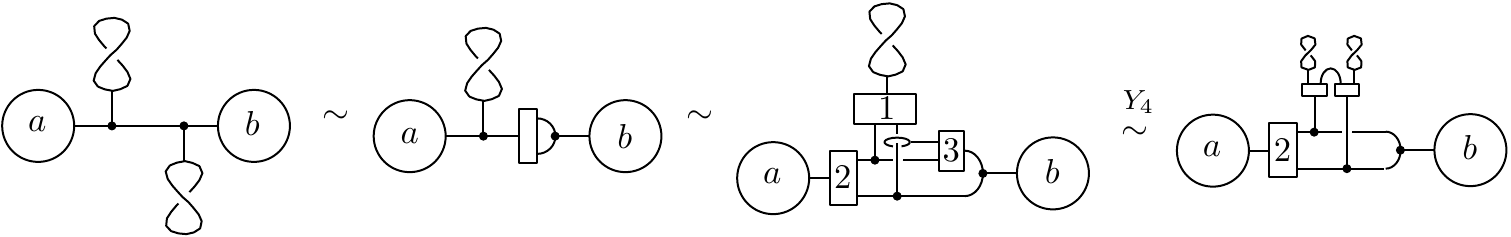}\ \raisebox{0.8cm}{.}
\end{center}
The first equivalence is Move~(a) in \cite[Lemma~A.6]{Mei06}
which derives from \cite[Theorem~3.1]{GGP01},
and the second equivalence is \cite[Move~11]{Hab00C}.
The third equivalence is obtained by the same moves as in Lemma~\ref{lem:asrelation}.
More precisely, the moves in the third equivalence are as follows. 
First, apply Move~(b) in \cite[Lemma~A.6]{Mei06} which changes Box~1 into the two boxes, and apply the zip construction with the marking being the both ends of the edge connecting them.
Next, by a crossing change using \cite[Lemma~E.5]{Oht02}, we can make the leaf in the upper output of Box~3 bound a disk.
Lastly, apply the inverse of the zip construction stated above,
and erase Box~3 by \cite[Move~3]{Hab00C}.

There are two unnumbered boxes in the right-hand side.
Apply the zip construction with the marking being the output of one of the two boxes attached to a special leaf.
Since the twisted clasper which arises from these moves has 3~nodes,
we can erase it under the $Y_4$-equivalence by \cite[Lemma~E.21]{Oht02} or \cite[Lemma~4.9]{GGP01}.
Applying the same moves to the other unnumbered box,
we can also erase the twisted clasper arising from the other special leaf.
The rest is the graph clasper with 3 ~nodes
which is represented by the Jacobi diagram $-O(a,a,b)$.

On the other hand, if we rotate the graph clasper 180 degrees in the left-hand side,
we also see that it is also $Y_4$-equivalent to $-O(a,b,b)$.
Thus, we obtain $\ss(O(a,a,b))\sim_{Y_4} \ss(O(a,b,b))$. 
\end{proof}

As in the notation in Section~\ref{section:Intro},
we denote by $T(a_1,a_2,\ldots, a_{n+2})$ the tree Jacobi diagram with colors $a_i\in \{1^{\pm},2^{\pm},\ldots,g^{\pm},\twist\}$.
Let us consider the case when at least one of the labels of the univalent vertices is $\twist$.
By assigning a special leaf to each univalent vertices with the label $\twist$,
we obtain a twisted clasper in $\Sigma_{g,1}\times[-1,1]$ in the same way as in Section~\ref{sec:surgery}.
By Lemmas~\ref{lem:crossingchange}, \ref{lem:leafchange}, and \ref{lem:asrelation},
the $Y_{n+2}$-equivalence class of a manifold obtained by the surgery along the twisted clasper does not depend on the choice of a representing surface.

By slightly extending the definition of $\ss$,
we denote by $\ss(T(a_1,a_2,a_3,a_4))\in Y_3\I\C/Y_4$ the $Y_4$-equivalence class of the manifold obtained by surgery along a twisted clasper that corresponds to $T(a_1,a_2,a_3,a_4)$, where $a_i\in \{1^{\pm},2^{\pm},\ldots,g^{\pm},\twist\}$.
Note that the $Y_4$-equivalence class $\ss(T(a_1,a_2,a_3,a_4))\in Y_3\I\C/Y_4$ is invariant under the generalized $\AS$ relation which derives from Lemma~\ref{lem:asrelation}, but it may change under the $\IHX$ relation of the Jacobi diagrams.

\begin{lemma}\label{lem:relationclasper}
For $a,b,c\in\{1^{\pm},\ldots,g^{\pm}\}$, we have the following.
\begin{enumerate}
\item $\ss(T(a,\twist,b,c))\sim_{Y_4}\ss(T(c,b,a,b,c))\circ \ss(O(a,b,c))^{-1}$.
\item $\ss(T(a,\twist,b,c))\circ \ss(T(a,b,\twist,c))^{-1}\sim_{Y_4}\ss(T(a,c,b,c,a))$.
\end{enumerate}
\end{lemma}
\begin{proof}
\begin{enumerate}
\item
We have the following equivalences.
\begin{center}
  \includegraphics{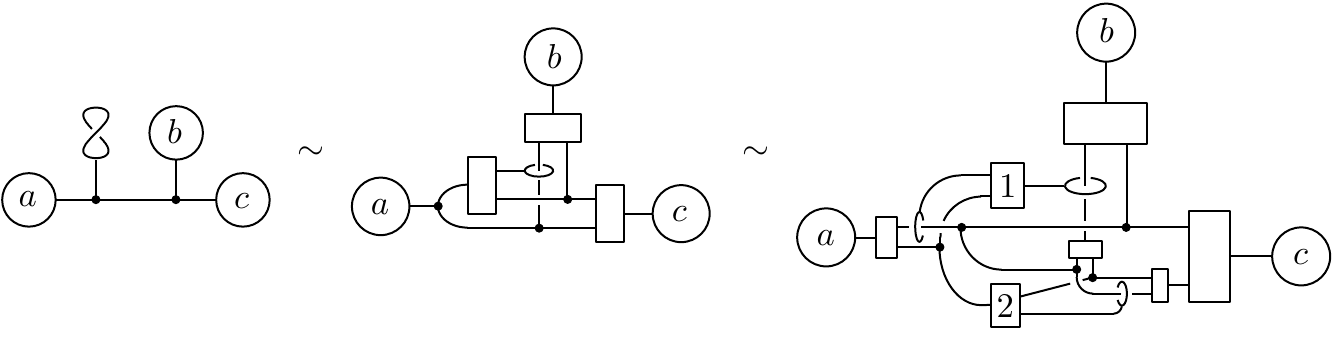}\ \raisebox{0.6cm}{.}
\end{center}
Here, the first equivalence is Move~(a) in \cite[Lemma~A.6]{Mei06} and \cite[Move~11]{Hab00C},
and the second one is given by applying the zip construction with the marking being the lower output in the left box.
These moves produce a tree clasper with 3 nodes whose leaves are outputs of the unnumbered 5 boxes as depicted in the right-hand side.
We can erase the 5 unnumbered boxes by \cite[Move~5]{Hab00C}.
By a sequence of crossing changes using \cite[Lemmas~E.5 and E.6]{Oht02},
we can move the tree clasper with 3 nodes to a horizontal layer $\Sigma_{g,1}\times[-1,-1+\epsilon]$ for small $\epsilon>0$ which is disjoint from the rest.
The 3-manifold obtained by applying the surgery to the horizontal layer along the tree clasper is the same as $\ss(T(c,b,a,b,c))$.
After moving the tree clasper and erasing Boxes~1 and 2 by \cite[Move~3]{Hab00C},
the rest is a clasper whose one leaf bounds a disk which intersects only one of its own edge.
By \cite[Lemma~E.13]{Oht02} and the AS relation,
it is $Y_4$-equivalent to the graph clasper represented by $-O(a,b,c)$.

\item
Applying \cite[Move~12]{Hab00C} and next \cite[Move~11]{Hab00C} twice, we have equivalences
\begin{center}
  \includegraphics{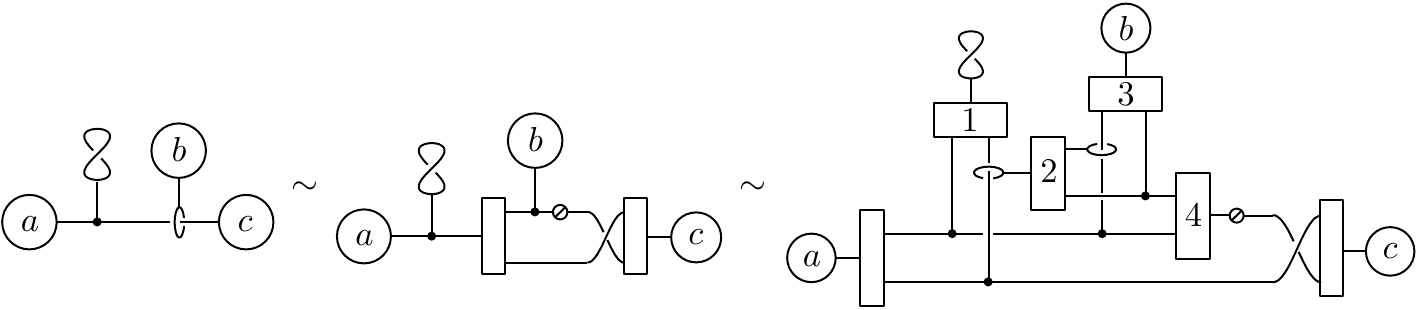}\ \raisebox{0.3cm}{.}
\end{center}
Applying the same procedure as the third equivalence in the proof of Lemma~\ref{lem:kernel of surgery map}~(1),
we can remove the leaf attached to the input of Box~2 from the clasper by \cite[Lemma~E.5]{Oht02}, and erase Box~2 and all edges attached to it.
We can also erase Boxes~3 and 4 by \cite[Move~3]{Hab00C}.
Applying \cite[Move~5]{Hab00C} to the unnumbered boxes,
the rest consists of the graph clasper represented by $T(c,a,a,b,c)$,
the twisted clasper represented by $-T(a,\twist,b,c)$,
and a twisted tree clasper with one node and a special leaf which we denote by $G$.
By \cite[Lemmas~E.5 and E.6]{Oht02} and Lemmas~\ref{lem:crossingchange} and \ref{lem:leafchange},
the three claspers can be isotoped into disjoint horizontal layers $\Sigma_{g,1}\times [-1, -1/3]$, $\Sigma_{g,1}\times (-1/3, 1/3]$, and $\Sigma_{g,1}\times (1/3, 1]$ under the $Y_4$-equivalence.
As a conclusion, we see that the graph clasper is $Y_4$-equivalent to
\[
\ss(T(c,a,a,b,c))\circ \ss(T(a,\twist,b,c))^{-1}\circ (\Sigma_{g,1}\times[-1,1])_G.
\]
Here, note that the $Y_4$-equivalence class of $(\Sigma_{g,1}\times[-1,1])_G$ depends on the isotopy class of the twisted clasper $G$ in $\Sigma_{g,1}\times [-1,1]$.

On the other hand, by moving a leaf through the special leaf,
we have an isotopy
\begin{center}
  \includegraphics{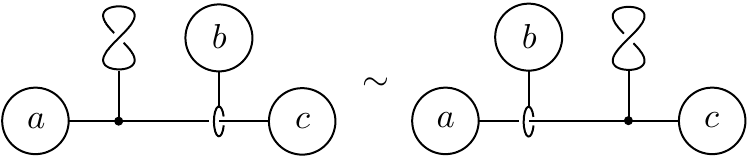}\ \raisebox{0cm}{.}
\end{center}
We see that the graph clasper in the right-hand side above is $Y_4$-equivalent to
\[
\ss(T(a,c,c,b,a))^{-1}\circ \ss(T(a,b,\twist,c))^{-1}\circ (\Sigma_{g,1}\times[-1,1])_G
\]
in the same way.
Comparing the left- and right-hand sides, we obtain
\[
\ss(T(a,\twist,b,c))\circ \ss(T(a,b,\twist,c))^{-1}\sim_{Y_4}
\ss(T(c,a,a,b,c))\circ \ss(T(a,c,c,b,a)).
\]
The AS and IHX relations imply
\[
\ss(T(c,a,a,b,c))\circ \ss(T(a,c,c,b,a))\sim_{Y_4}\ss(T(a,c,b,c,a)),
\]
and we obtain what we want.
\end{enumerate}
\end{proof}

\begin{proof}[Proof of Lemma~\ref{lem:kernel of surgery map}~(2)]
Applying Lemma~\ref{lem:asrelation} twice,
we see that 
\[
\ss(T(a,\twist,b,c))\sim_{Y_4} \ss(T(c,b,\twist,a))^{-1}.
\]
By Lemma~\ref{lem:relationclasper}~(2) and (1),
we obtain
\begin{align*}
&\ss(T(a,c,b,c,a))\\
&\sim_{Y_4} \ss(T(a,\twist,b,c))\circ \ss(T(c,\twist,b,a))\\
&\sim_{Y_4} \ss(T(c,b,a,b,c))\circ \ss(O(a,b,c))^{-1}\circ \ss(T(a,b,c,b,a))\circ \ss(O(c,b,a))^{-1}.
\end{align*}
The AS relation implies $\ss(O(a,b,c))=\ss(O(c,b,a))^{-1}\in Y_3\I\C/Y_4$,
and we obtain
\[
\Delta_{1,0}(T(a,b,c))=T(c,b,a,b,c)+T(a,c,b,c,a)+T(a,b,c,b,a)\in \Ker\ss.
\]
\end{proof}

\subsection{Finite-type invariants of homology cylinders}
\label{sec:FTI}
Finally, we see the proof of Theorem~\ref{thm:str-of-y3} from a viewpoint of finite-type invariants introduced in \cite{Gou99} and \cite{Hab00C}.
First note that $\overline{Z}_{n+1}$ is a finite-type invariant of degree at most $n$ since it vanishes on $\F_{n+1}\I\C$.
The following proposition is an affirmative answer for the Goussarov-Habiro conjecture for the $Y_4$-equivalence.
Here the conjecture is true for the $Y_2$- and $Y_3$-equivalences (see \cite[Corollary~1.5]{MaMe03} and \cite[Theorem~A]{MaMe13} for instance).

\begin{prop}
\label{prop:GHC}
For $M, M' \in \I\C$, they are $Y_{4}$-equivalent if and only if $f(M)=f(M')$ for every finite-type invariant $f$ of degree at most $3$.
\end{prop}

\begin{proof}
The idea of the proof is the same as \cite[Section~5.1]{MaMe13}.
If $M \sim_{Y_4} M'$, then they are not distinguished by finite-type invariants of degree at most $3$ (see \cite[Lemma~2.3]{MaMe13} for instance).
Conversely, suppose that $f(M)=f(M')$ for every finite-type invariant $f$ of degree at most $3$.
Then we see $M \sim_{Y_3} M'$ by \cite[Corollary~1.5]{MaMe03} and \cite[Theorem~A]{MaMe13}.
Since $\I\C/Y_4$ is a group (\cite[Theorem~3]{Gou99}, \cite[Section~8.5]{Hab00C}), there is $N \in Y_3\I\C$ such that $M \sim_{Y_4} N\circ M'$.
We next use the identities $\Ztilde^Y_{\leq 3}(M)=\Ztilde^Y_{\leq 3}(M')$ and $\overline{Z}_4(M)=\overline{Z}_4(M')$.
It follows from $\Ztilde^Y_n(M) = \sum_{i=0}^{n} \Ztilde^Y_i(N)\star\Ztilde^Y_{n-i}(M')$ that $\Ztilde^Y_{\leq 3}(N)=\emptyset$ and $\overline{Z}_4(N)=0$, and thus $\bar{z}_4(N)=0$.
Since $\bar{z}_4\colon \tor(Y_3\I\C/Y_4) \to \A_4^c\otimes\Q/\Z$ is injective (see Remark~\ref{rem:z4inj}), $N=0 \in Y_3\I\C/Y_4$, and hence $M$ is $Y_{4}$-equivalent to $M'$.
\end{proof}

\appendix
\section{Values of the LMO functor on the Torelli group}
Here, we compute some values of the LMO functor on the Torelli group using a computer.
This appendix is not used in the previous sections, and is just for illustration of Theorem~\ref{thm:main}.

Let $M_1$, $M_2$, and $M_3$ be the bottom top tangles obtained by the surgeries along framed links depicted in Figure~\ref{fig:links}.
\begin{figure}[h]
 \centering
 \includegraphics{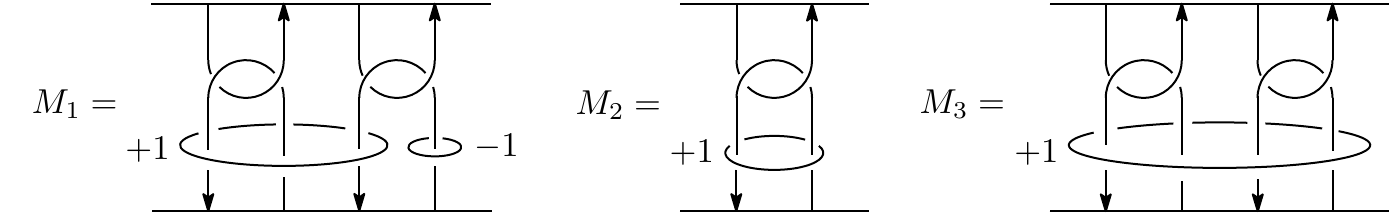}
\caption{Bottom-top tangles $M_1$, $M_2$, and $M_3$.}
 \label{fig:links}
\end{figure}
These three tangles correspond to the images of a genus 1 bounding pair map and the Dehn twists along bounding simple closed curves of genus 1 and 2 under $\cc\colon\I\to\I\C$, respectively.
A computer calculation shows the following.
\begin{prop}
The images of $M_1$, $M_2$ and $M_3$ under
$\log\Ztilde^Y\colon\I\C\to \widehat{\A}^c\otimes\Q$
$\mod \widehat{\A}_{\ge4}^c\otimes\Q$ are
\begin{align*}
\log\Ztilde^Y(M_1)
&=
T(1^+,1^-,2^-)-\frac{1}{2}O(1^+,1^-)\\
&\quad
+\frac{1}{2}\left\{
T(1^+,1^-,1^-,1^+)
+T(1^+,2^-,2^-,1^-)
+T(1^+,1^-,2^-,2^+)
\right\}\\
&\quad
-\frac{3}{4}O(1^+,1^-,2^-)
-T(1^-,1^+,1^+,2^-,1^-)
+\frac{1}{6}T(1^-,2^-,2^-,2^-,1^+)\\
&\quad
+\frac{1}{4}T(1^-,2^-,2^-,2^+,1^+)
+\frac{1}{12}T(1^-,2^-,2^+,2^+,1^+)
-\frac{1}{4}T(1^-,2^+,2^-,2^-,1^+)\\
&\quad
+\frac{1}{12}T(1^-,2^+,2^+,2^-,1^+)
-\frac{1}{6}T(1^-,2^+,2^-,2^+,1^+),\\
\log\Ztilde^Y(M_2)&
=-\frac{1}{2}O(1^+,1^-)+\frac{1}{2}T(1^+,1^-,1^-,1^+),\\
\log\Ztilde^Y(M_3)&
=-\frac{1}{2}O(1^+,1^-)-\frac{1}{2}O(2^+,2^-)\\
&\quad +T(1^+,1^-,2^-,2^+)+\frac{1}{2}T(1^+,1^-,1^-,1^+)+\frac{1}{2}T(2^+,2^-,2^-,2^+).
\end{align*}
\end{prop}

\begin{remark}
As explained in Remark~\ref{rem:BC-homo}, 
$Y_2$-equivalence classes of homology cylinders are determined by values of $\Ztilde_1^Y$ and $\bar{z}_2$.
Thus, we see from Theorem~\ref{thm:main} that
\[
M_1\sim_{Y_2}\ss(-T(1^+,1^-,2^-)+T(1^+,1^+,1^-)).
\]
By \cite[Theorem~A and Propositions~C.1 and C.2]{MaMe13}, we also have
\begin{align*}
M_2\sim_{Y_3}&\ss(T(\twist,1^+,1^-)),\\
M_3\sim_{Y_3}&\ss(T(\twist,1^+,1^-)+T(\twist,2^+,2^-)-T(1^+,1^-,2^-,2^+)),
\end{align*}
where $\ss(T(\twist,i^+,i^-))\in \I\C/Y_3$ for $i=1,2$ is the $Y_3$-equivalence class obtained by the surgery along $T(\twist,i^+,i^-)$ as explained in right before Lemma~\ref{lem:relationclasper}.
\end{remark}

\section{The Leibniz rule}\label{section:leibniz}
Here, we prove Proposition~\ref{prop:Leibniz},
which states that $\delta\colon\A^Y_n\to\A^Y_{n+1}\otimes\Z/2\Z$ satisfies the Leibniz rule.
\begin{remark}
Since the surgery map
$\S\colon \A^Y\to \bigoplus_{n=1}^\infty \F_n\I\C/\F_{n+1}\I\C$
is an algebra homomorphism as shown in \cite[Theorem~5~(i)]{GaLe05},
and $\overline{Z}_{n+1}$ is defined to be the mod $\Z$ reduction of $\Ztilde_{n+1}^Y$,
we deduce that
\[
(\overline{Z}_{k+l+1}\circ\S)(x\star y)=
(\overline{Z}_{k+1}\circ\S)(x)\star (\Ztilde^Y_{l}\circ\S)(y)
+(\Ztilde^Y_{k}\circ\S)(x)\star (\overline{Z}_{l+1}\circ\S)(y)
\in\A_{k+l+1}^Y\otimes\Q/\Z
\]
for $x\in\A_k^Y$ and $y\in\A_l^Y$.
Since the homomorphism $\Y\colon \A^Y_n\to\A^Y_{n+1}\otimes\Z/2\Z$ also satisfies the Leibniz rule,
Theorem~\ref{thm:main} implies the weaker statement than Proposition~\ref{prop:Leibniz} that $(\id\otimes\frac{1}{2})\circ\delta$ satisfies the Leibniz rule.
Moreover, it implies $\delta$ also satisfies the Leibniz rule if $\id\otimes\frac{1}{2}\colon \A_{n+1}^Y\otimes\Z/2\Z\to \A_{n+1}^Y\otimes\Q/\Z$ is injective,
i.e. $\A_{n+1}^Y$ has no $2$-torsion elements.
\end{remark}

For $a\in\{1^\pm,2^\pm,\ldots,g^\pm\}$ and a Jacobi diagram $J$ without strut components,
let us denote by $U_a(J)$ the set of univalent vertices in $J$ colored by $a$.
If we denote
\[
\delta^a(J)=
\sum_{v\in U_a(J)}\delta_v(J)
+\sum_{\begin{subarray}{c}v,w\in U_a(J)\\v\ne w\end{subarray}}\delta_{vw}(J),
\]
we have $\delta(J)=\sum_{i=1}^g\delta^{i^+}(J)+\sum_{i=1}^g\delta^{i^-}(J)$.
Thus, it suffices to prove
\[
\delta^{i^+}(J\star J')+\delta^{i^-}(J\star J')
=\delta^{i^+} J\star J'+J\star \delta^{i^+} J'
+\delta^{i^-} J\star J'+J\star \delta^{i^-} J'
\]
for $1\le i\le g$ and Jacobi diagrams $J$ and $J'$ without strut components.
We denote
\[
J\star_{i}J'=\sum_{\begin{subarray}{c}V\subset U_{i^+}(J)\\W\subset U_{i^-}(J')\\ \beta\colon V\cong W\end{subarray}}
J\cup_{\beta} J',
\]
where $\beta$ runs through all the bijections between subsets of $U_{i^+}(J)$ and $U_{i^-}(J')$.
This is the sum of Jacobi diagrams obtained by connecting some of univalent vertices in $J$ labeled by $i^+$ with those in $J'$ labeled by $i^-$.
Note that the maps $\delta^{i^+}$ and $\delta^{i^-}$ do not change the univalent vertices with labels in $\{1^\pm,\ldots,g^\pm\}\setminus \{i^{\pm}\}$.
More precisely,
the maps $\delta^{i^+}$ and $\delta^{i^-}$ commute with the operation connecting univalent vertices with labels $\{1^\pm,\ldots,g^\pm\}\setminus \{i^{\pm}\}$.
Thus, it suffices to show that
\[
\delta^{i^+}(J\star_{i}J')+\delta^{i^-}(J\star_{i}J')=(\delta^{i^+}J)\star_iJ'+J\star_i(\delta^{i^+} J')+(\delta^{i^-}J)\star_iJ'+J\star_i(\delta^{i^-} J'),
\]
and Proposition~\ref{prop:Leibniz} follows from the next lemma.
\begin{lemma}\label{lem:property-delta}
The following identities hold.
\begin{align}
\delta^{i^+}J\star_iJ'+J\star_i\delta^{i^-}J'
&=\sum_{\begin{subarray}{c}
V\subset U_{i^+}(J)\\
W\subset U_{i^-}(J')\\
\beta\colon V\cong W
\end{subarray}}
\Biggl(
\sum_{v\in U_{i^+}(J)\setminus V}
(\delta_v J)\cup_{\beta} J'+
\sum_{w\in U_{i^-}(J')\setminus W}
J\cup_{\beta} (\delta_wJ')\notag\\
&+\sum_{
\begin{subarray}{c}
v,v'\in U_{i^+}(J)\setminus V\\
v\ne v'
\end{subarray}}
(\delta_{vv'} J)\cup_{\beta} J'
+\sum_{
\begin{subarray}{c}
w,w'\in U_{i^-}(J')\setminus W\\
w\ne w'
\end{subarray}}
J\cup_{\beta} (\delta_{ww'}J')
\Biggr),\label{eq:B1}\\
\delta^{i^+}(J\star_iJ')+J\star_i\delta^{i^+}J'
&=\sum_{\begin{subarray}{c}
V\subset U_{i^+}(J)\\
W\subset U_{i^-}(J')\\
\beta\colon V\cong W
\end{subarray}}
\Biggl(
\sum_{v\in U_{i^+}(J)\setminus V}
(\delta_v J)\cup_\beta J'
+\sum_{
\begin{subarray}{c}
v,v'\in U_{i^+}(J)\setminus V\\
v\ne v'
\end{subarray}}
(\delta_{vv'} J)\cup_\beta J'
\Biggr),\label{eq:B2}\\
\delta^{i^-}(J\star_iJ')+\delta^{i^-}J\star_iJ'
&=\sum_{\begin{subarray}{c}
V\subset U_{i^+}(J)\\
W\subset U_{i^-}(J')\\
\beta\colon V\cong W
\end{subarray}}
\Biggl(
\sum_{w\in U_{i^-}(J')\setminus W}
J\cup_\beta (\delta_wJ')
+\sum_{
\begin{subarray}{c}
w,w'\in U_{i^-}(J')\setminus W\\
w\ne w'
\end{subarray}}
J\cup_\beta (\delta_{ww'}J')
\Biggr).\label{eq:B3}
\end{align}
\end{lemma}

\begin{proof}
First, we prove (\ref{eq:B1}).
Recall that $\delta_v(J)$ for $v\in U_{i^+}(J)$ is defined to be the sum of Jacobi diagrams obtained by changing a neighborhood of the edge incident to the vertex $v$ as
\[
\delta_v\colon
\begin{tikzpicture}[mydash, baseline=-0.7ex, scale=0.25]
\small
\draw  (1.2,-1)  arc [start angle = -90, end angle = 90, x radius = 1.3, y radius = 1];
\draw (2.5,0) -- (4.5,0);
\node [right] at (4.5,0) {$i^+$};
\end{tikzpicture}
\mapsto
\begin{tikzpicture}[mydash, baseline=-0.7ex, scale=0.25]
\draw  (1.2,-1)  arc [start angle = -90, end angle = 90, x radius = 1.3, y radius = 1];
\draw (2.4,0.4) -- (4.4,0.4);
\draw (2.4,-0.4) -- (4.4,-0.4);
\node [right] at (4.4,0.9) {$i^+$};
\node [right] at (4.4,-0.9) {$i^+$};
\end{tikzpicture}
+
\begin{tikzpicture}[mydash, baseline=-0.7ex, scale=0.25]
\small
\draw  (1.2,-1)  arc [start angle = -90, end angle = 90, x radius = 1.3, y radius = 1];
\draw (2.5,0) -- (3.5,0);
\draw (3.5,0) -- (4.4,0.6);
\draw (3.5,0) -- (4.4,-0.6);
\node [right] at (4.4,0.9) {$i^+$};
\node [right] at (4.4,-0.9) {$i^-$};
\end{tikzpicture}
.
\]
We denote by $v_1$ and $v_2$ the two univalent vertices colored with $i^+$ which arise from $v$ in the first term of $\delta_v(J)$.
In the product $\delta_vJ\star_i J'$,
there are 3 kinds of terms coming from the first term of $\delta_v J$:
terms in which neither $v_1$ nor $v_2$ is connected to a univalent vertex in $J'$,
terms in which only one of $v_1$ and $v_2$ is connected to a univalent vertex $w$ in $J'$,
terms in which $v_1$ and $v_2$ are connected to univalent vertices $w$ and $w'$, respectively, in $J'$.
Thus, we can write $\delta_v J\star_i J'$ as
\begin{align}
\delta_v J\star_i J'
=\sum_{\begin{subarray}{c}
V\subset U_{i^+}(J)\setminus \{v\}\\
W\subset U_{i^-}(J')\\
\beta\colon V\cong W
\end{subarray}}
\Biggl(&
\delta_vJ\cup_\beta J'
+\sum_{w\in U_{i^-}(J')\setminus W}
(J\cup_{\beta}J')\cup_{\begin{subarray}{c}y_1=v\\y_2=w\end{subarray}}
\left(\begin{tikzpicture}[mydash, baseline=1.3ex, scale=0.25]
\small
\draw (0,0) -- (2,0);
\draw (1,0) -- (1,1.3);
\node [left] at (0,0) {$y_1$};
\node [right] at (2,0) {$y_2$};
\node [above] at (1,1.3) {$i^+$};
\end{tikzpicture}\right)\notag\\
&+\sum_{\begin{subarray}{c}
w,w'\in U_{i^-}(J')\setminus W\\
w\ne w'
\end{subarray}}
(J\cup_{\beta}J')\cup_{\begin{subarray}{c}y_1=v\\y_2=w\\y_3=w'\end{subarray}}
\left(\begin{tikzpicture}[mydash, baseline=1.3ex, scale=0.25]
\small
\draw (0,0) -- (2,0);
\draw (1,0) -- (1,1.3);
\node [left] at (0,0) {$y_1$};
\node [right] at (2,0) {$y_3$};
\node [above] at (1,1.3) {$y_2$};
\end{tikzpicture}\right)\notag\\
&+\sum_{w\in U_{i^-}(J')\setminus W}
(J\cup_{\beta}J')\cup_{\begin{subarray}{c}y_1=v\\y_2=w\end{subarray}}
\left(\begin{tikzpicture}[mydash, baseline=1.3ex, scale=0.25]
\small
\draw (0,0) -- (2,0);
\draw (1,0) -- (1,1.3);
\node [left] at (0,0) {$y_1$};
\node [right] at (2,0) {$y_2$};
\node [above] at (1,1.3) {$i^-$};
\end{tikzpicture}\right)
\Biggr).\label{eq:leibnizrule1}
\end{align}
Here, the second term inside the first summation is obtained by applying the IHX relation to the sum of two Jacobi diagrams in which the vertex $v_1$ is connected to $w$ and $v_2$ is connected to $w$, respectively, as
\[
\begin{tikzpicture}[mydash, baseline=-0.7ex, scale=0.25]
\small
\draw (0,-1)  arc [start angle = -90, end angle = 90, x radius = 1.3, y radius = 1];
\draw (1.2,0.4) -- (3.2,0.4);
\draw (1.2,-0.4) -- (3.2,-0.4);
\node [below] at (3.2,-0.4) {$i^+$};
\draw (6.5,1)  arc [start angle = 90, end angle = 270, x radius = 1.3, y radius = 1];
\draw (5.2,0) to [out=160, in=0] (3.2,0.4);
\end{tikzpicture}
\,+\,
\begin{tikzpicture}[mydash, baseline=-0.7ex, scale=0.25]
\small
\draw (0,-1)  arc [start angle = -90, end angle = 90, x radius = 1.3, y radius = 1];
\draw (1.2,0.4) -- (3.2,0.4);
\draw (1.2,-0.4) -- (3.2,-0.4);
\node [above] at (3.2,0.4) {$i^+$};
\draw (6.5,1)  arc [start angle = 90, end angle = 270, x radius = 1.3, y radius = 1];
\draw (5.2,0) to [out=200, in=0] (3.2,-0.4);
\end{tikzpicture}
\,=\,
\begin{tikzpicture}[mydash, baseline=-0.7ex, scale=0.25]
\small
\draw (0,-1)  arc [start angle = -90, end angle = 90, x radius = 1.3, y radius = 1];
\draw (6.5,1)  arc [start angle = 90, end angle = 270, x radius = 1.3, y radius = 1];
\draw (5.2,0) to (1.2,0);
\draw (3.2,0) to (3.2,1);
\node [above] at (3.2,1) {$i^+$};
\end{tikzpicture}\quad.
\]
The third term is obtained by applying the IHX relation to the sum of two Jacobi diagrams where the vertices $v_1$ and $v_2$ are connected to $w$ and $w'$ as
\[
\begin{tikzpicture}[mydash, baseline=-0.7ex, scale=0.25]
\small
\draw (0,-1)  arc [start angle = -90, end angle = 90, x radius = 1.3, y radius = 1];
\draw (1.2,0.4) -- (5.2,1.5);
\draw (1.2,-0.4) -- (5.2,-1.5);
\draw (6.5,2.5) arc [start angle = 90, end angle = 270, x radius = 1.3, y radius = 1];
\draw (6.5,-0.5)  arc [start angle = 90, end angle = 270, x radius = 1.3, y radius = 1];
\end{tikzpicture}
\,+\,
\begin{tikzpicture}[mydash, baseline=-0.7ex, scale=0.25]
\small
\draw (0,-1)  arc [start angle = -90, end angle = 90, x radius = 1.3, y radius = 1];
\draw (1.2,0.4)  to [out=0, in=180] (3.2,0.4) to [out=0, in=140]  (3.7,0.1);
\draw  (4.2,-0.4) to [out=-40, in=140]  (5.2,-1.5);
\draw (1.2,-0.4)  to [out=0, in=180] (3.2,-0.4) to [out=0, in=230] (5.2,1.5);
\draw (6.5,2.5) arc [start angle = 90, end angle = 270, x radius = 1.3, y radius = 1];
\draw (6.5,-0.5)  arc [start angle = 90, end angle = 270, x radius = 1.3, y radius = 1];
\end{tikzpicture}
\,=\,
\begin{tikzpicture}[mydash, baseline=-0.7ex, scale=0.25]
\small
\draw (0,-1)  arc [start angle = -90, end angle = 90, x radius = 1.3, y radius = 1];
\draw (1.2,0) -- (3.2,0);
\draw (3.2,0) -- (5.2,1.5);
\draw (3.2,0) -- (5.2,-1.5);
\draw (6.5,2.5) arc [start angle = 90, end angle = 270, x radius = 1.3, y radius = 1];
\draw (6.5,-0.5)  arc [start angle = 90, end angle = 270, x radius = 1.3, y radius = 1];
\end{tikzpicture}\quad.
\]
For a Jacobi diagram $J_0$,
Let $\rev(J_0)$ denote the Jacobi diagram obtained by changing the signs of the labels of all the univalent vertices in $J_0$.
By the definitions of $\star_i$ and $\delta_v$ for a univalent vertex $v$ in $J$, we have
\[
\rev(J\star_i J')=\rev J'\star_i\rev J, \text{ and }
\rev(\delta_v J)=\delta_v (\rev J).
\]
By applying Equation~(\ref{eq:leibnizrule1}) to the right-hand side of
\[
J\star_i \delta_wJ'
=\rev(\delta_w(\rev J')\star_i \rev J)
\]
for $w\in U_{i^-}(J')$, we obtain
\begin{align*}
J\star_i \delta_wJ'
=\sum_{\begin{subarray}{c}
V\subset U_{i^+}(J)\\
W\subset U_{i^-}(J')\setminus \{w\}\\
\beta\colon V\cong W
\end{subarray}}
\Biggl(&
J\cup_\beta \delta_wJ'
+\sum_{v\in U_{i^+}(J)\setminus V}
(J\cup_{\beta}J')\cup_{\begin{subarray}{c}y_1=v\\y_2=w\end{subarray}}
\left(\begin{tikzpicture}[mydash, baseline=1.3ex, scale=0.25]
\small
\draw (0,0) -- (2,0);
\draw (1,0) -- (1,1.3);
\node [left] at (0,0) {$y_1$};
\node [right] at (2,0) {$y_2$};
\node [above] at (1,1.3) {$i^-$};
\end{tikzpicture}\right)\\
&+\sum_{\begin{subarray}{c}
v,v'\in U_{i^+}(J)\setminus V\\
v\ne v'
\end{subarray}}
(J\cup_{\beta}J')\cup_{\begin{subarray}{c}y_1=v\\y_2=v'\\y_3=w\end{subarray}}
\left(\begin{tikzpicture}[mydash, baseline=1.3ex, scale=0.25]
\small
\draw (0,0) -- (2,0);
\draw (1,0) -- (1,1.3);
\node [left] at (0,0) {$y_1$};
\node [right] at (2,0) {$y_3$};
\node [above] at (1,1.3) {$y_2$};
\end{tikzpicture}\right)\\
&+\sum_{v\in U_{i^+}(J)\setminus V}
(J\cup_{\beta}J')\cup_{\begin{subarray}{c}y_1=v\\y_2=w\end{subarray}}
\left(\begin{tikzpicture}[mydash, baseline=1.3ex, scale=0.25]
\small
\draw (0,0) -- (2,0);
\draw (1,0) -- (1,1.3);
\node [left] at (0,0) {$y_1$};
\node [right] at (2,0) {$y_2$};
\node [above] at (1,1.3) {$i^+$};
\end{tikzpicture}\right)
\Biggr).
\end{align*}
In the same way, for $v,v'\in U_{i^+}(J)$ and $w,w'\in U_{i^-}(J')$,
we have
\begin{align*}
\delta_{vv'}J\star_i J'&=
\sum_{\begin{subarray}{c}
V\subset U_{i^+}(J)\setminus\{v,v'\}\\
W\subset U_{i^-}(J')\\
\beta\colon V\cong W
\end{subarray}}
\Biggl(
\delta_{vv'}J\cup_{\beta} J'
+\sum_{w\in U_{i^-}(J')\setminus W}
(J\cup_{\beta}J')
\cup_{\begin{subarray}{c}y_1=v\\y_2=v'\\y_3=w\end{subarray}}
\left(\begin{tikzpicture}[mydash, baseline=1.3ex, scale=0.25]
\small
\draw (0,0) -- (2,0);
\draw (1,0) -- (1,1.3);
\node [left] at (0,0) {$y_1$};
\node [right] at (2,0) {$y_3$};
\node [above] at (1,1.3) {$y_2$};
\end{tikzpicture}\right)
\Biggr),\\
J\star_i \delta_{ww'}J'&=
\sum_{\begin{subarray}{c}
V\subset U_{i^+}(J)\\
W\subset U_{i^-}(J')\setminus\{w,w'\}\\
\beta\colon V\cong W
\end{subarray}}
\Biggl(
J\cup_{\beta} \delta_{ww'}J'
+\sum_{v\in U_{i^+}(J)\setminus V}
(J\cup_{\beta}J')
\cup_{\begin{subarray}{c}y_1=v\\y_2=w\\y_3=w'\end{subarray}}
\left(\begin{tikzpicture}[mydash, baseline=1.3ex, scale=0.25]
\small
\draw (0,0) -- (2,0);
\draw (1,0) -- (1,1.3);
\node [left] at (0,0) {$y_1$};
\node [right] at (2,0) {$y_3$};
\node [above] at (1,1.3) {$y_2$};
\end{tikzpicture}\right)
\Biggr).
\end{align*}
From the above calculations,
we obtain
\begin{align*}
&\delta^{i^+}J\star_iJ'+J\star_i\delta^{i^-}J'\\
&=\sum_{v\in U_{i^+}(J)}\delta_vJ\star_iJ'
+\sum_{\begin{subarray}{c}v,v'\in U_{i^+}(J)\\v\ne v'\end{subarray}}\delta_{vv'}J\star_iJ'
+\sum_{w\in U_{i^-}(J')}J\star_i\delta_wJ'
+\sum_{\begin{subarray}{c}w,w'\in U_{i^-}(J')\\w\ne w'\end{subarray}}J\star_i\delta_{ww'}J'\\
&=\sum_{\begin{subarray}{c}
V\subset U_{i^+}(J)\\
W\subset U_{i^-}(J')\\
\beta\colon V\cong W
\end{subarray}}
\Biggl(
\sum_{v\in U_{i^+}(J)\setminus V}
(\delta_v J)\cup_{\beta} J'+
\sum_{w\in U_{i^-}(J')\setminus W}
J\cup_{\beta} (\delta_wJ')\\
&\qquad\qquad\quad\qquad
+\sum_{\begin{subarray}{c}
v,v'\in U_{i^+}(J)\setminus V\\
v\ne v'
\end{subarray}}
(\delta_{vv'} J)\cup_{\beta} J'
+\sum_{
\begin{subarray}{c}
w,w'\in U_{i^-}(J')\setminus W\\
w\ne w'
\end{subarray}}
J\cup_{\beta} (\delta_{ww'}J')
\Biggr).
\end{align*}

Next, we prove Equation~(\ref{eq:B2}).
For $w,w'\in U_{i^+}(J')$, we have
\begin{align*}
J\star_i\delta_wJ'
&=\sum_{\begin{subarray}{c}
V\subset U_{i^+}(J)\\
W\subset U_{i^-}(J')\\
\beta\colon V\cong W
\end{subarray}}
\Biggl(
J\cup_{\beta}\delta_wJ'
+\sum_{v\in U_{i^+}(J)\setminus V}
(J\cup_{\beta}J')
\cup_{\begin{subarray}{c}y_1=v\\y_2=w\end{subarray}}
\left(\begin{tikzpicture}[mydash, baseline=1.3ex, scale=0.25]
\small
\draw (0,0) -- (2,0);
\draw (1,0) -- (1,1.3);
\node [left] at (0,0) {$y_1$};
\node [right] at (2,0) {$y_2$};
\node [above] at (1,1.3) {$i^+$};
\end{tikzpicture}\right)
\Biggr),\\
J\star_i\delta_{ww'}J'
&=\sum_{\begin{subarray}{c}
V\subset U_{i^+}(J)\\
W\subset U_{i^-}(J')\\
\beta\colon V\cong W
\end{subarray}}
J\cup_{\beta}\delta_{ww'}J'.
\end{align*}
Thus, we obtain
\begin{align*}
J\star_i\delta^{i^+}J'
&=J\star_i\Biggl(\sum_{w\in U_{i_+}(J')}\delta_wJ'
+\sum_{\begin{subarray}{c}
w,w'\in U_{i^+}(J')\\
w\ne w'
\end{subarray}}
\delta_{ww'}J'\Biggr)\\
&=\sum_{\begin{subarray}{c}
V\subset U_{i^+}(J)\\
W\subset U_{i^-}(J')\\
\beta\colon V\cong W
\end{subarray}}
\Biggl(\sum_{w\in U_{i^+}(J')}
J\cup_{\beta}\delta_wJ'
+\sum_{\begin{subarray}{c}w\in U_{i^+}(J')\\v\in U_{i^+}(J)\setminus V\end{subarray}}
(J\cup_{\beta}J')
\cup_{\begin{subarray}{c}y_1=v\\y_2=w\end{subarray}}
\left(\begin{tikzpicture}[mydash, baseline=1.3ex, scale=0.25]
\small
\draw (0,0) -- (2,0);
\draw (1,0) -- (1,1.3);
\node [left] at (0,0) {$y_1$};
\node [right] at (2,0) {$y_2$};
\node [above] at (1,1.3) {$i^+$};
\end{tikzpicture}\right)
\\
&\qquad\qquad\qquad
+\sum_{\begin{subarray}{c}w,w'\in U_{i^+}(J')\\w\ne w'\end{subarray}}
J\cup_{\beta}\delta_{ww'}J'\Biggr).
\end{align*}
We also have
\begin{align*}
\delta^{i^+}(J\star_i J')
=\sum_{\begin{subarray}{c}
V\subset U_{i^+}(J)\\
W\subset U_{i^-}(J')\\
\beta\colon V\cong W
\end{subarray}}
\Biggl(&
\sum_{v\in U_{i^+}(J)\setminus V}\delta_vJ\cup_{\beta}J'
+\sum_{\begin{subarray}{c}v,v'\in U_{i^+}(J)\setminus V\\v\ne v'\end{subarray}}
\delta_{vv'}J\cup_{\beta}J'\\
&+\sum_{w\in U_{i^+}(J')}
J\cup_{\beta}\delta_wJ'
+\sum_{\begin{subarray}{c}w,w'\in U_{i^+}(J')\\w\ne w'\end{subarray}}
J\cup_{\beta}\delta_{ww'}J'\\
&+\sum_{\begin{subarray}{c}v\in U_{i^+}(J)\setminus V\\w\in U_{i^+}(J')\end{subarray}}
(J\cup_{\beta}J')\cup_{\begin{subarray}{c}y_1=v\\y_2=w\end{subarray}}
\left(\begin{tikzpicture}[mydash, baseline=1.3ex, scale=0.25]
\small
\draw (0,0) -- (2,0);
\draw (1,0) -- (1,1.3);
\node [left] at (0,0) {$y_1$};
\node [right] at (2,0) {$y_2$};
\node [above] at (1,1.3) {$i^+$};
\end{tikzpicture}\right)\Biggr).
\end{align*}
By taking the sum
$J\star_i\delta^{i^+}J'+\delta^{i^+}(J\star_i J')$,
we obtain what we want.

Equation~(\ref{eq:B3}) is obtained by applying Equation~(\ref{eq:B2}) to the right-hand side of 
\[
\delta^{i^-}(J\star_i J')+\delta^{i^-}J\star_i J'
=\rev(\delta^{i^+}(\rev J'\star_i\rev J)+\rev J'\star_{i}\delta^{i^+}(\rev J)).
\]
\end{proof}

\def\cprime{$'$} \def\cprime{$'$} \def\cprime{$'$}

\end{document}